\documentclass{amsart}
\usepackage[all]{xy}
\usepackage{comment}

\newtheorem{theorem}{Theorem}
\numberwithin{theorem}{section}
\newtheorem{lemma}[theorem]{Lemma}
\newtheorem{proposition}[theorem]{Proposition}

\newtheorem{definition}[theorem]{Definition}
\newtheorem{example}[theorem]{Example}

\theoremstyle{remark}
\newtheorem{remark}[theorem]{Remark}
\newtheorem{notation}[theorem]{Notation}
\newcommand{\dgcat}{\mathsf{dgcat}}
\newcommand{\rep}{\mathsf{rep}}
\newcommand{\ca}{{\mathcal A}}
\newcommand{\cc}{{\mathcal C}}

\newcommand{\cg}{{\mathcal G}}
\newcommand{\cd}{{\mathcal D}}

\newcommand{\ce}{{\mathcal E}}
\newcommand{\cm}{{\mathcal M}}

\newcommand{\cb}{{\mathcal B}}
\newcommand{\cu}{{\mathcal U}}

\newcommand{\ul}{ \rule[0.2mm]{0.12mm}{2.2mm} \rule[2.3mm]{2.2mm}{0.12mm}}
\newcommand{\lr}{\rule{2.3mm}{0.12mm} \rule{0.12mm}{2.2mm}}
\newcommand{\kar}{\rule[-0.1mm]{0.12mm}{1.5mm}\hspace{-0.36mm} \ni}

\newcommand{\internalcomment}[1]{}

\begin{document}

\title[Higher $K$-theory via universal invariants]{Higher $K$-theory
  via universal invariants}
\author{Gon{\c c}alo Tabuada}
\address{Universit{\'e} Paris 7 - Denis Diderot, UMR 7586
  du CNRS, case 7012, 2 Place Jussieu, 75251 Paris cedex 05, France and
Departamento de Matem{\'a}tica, FCT-UNL, Quinta da Torre, 2829-516 Caparica,~Portugal}

\thanks{Supported by FCT-Portugal, scholarship {\tt SFRH/BD/14035/2003}.}

\email{
\begin{minipage}[t]{5cm}
tabuada@math.jussieu.fr tabuada@fct.unl.pt
\end{minipage}
}

\begin{abstract}
Using the formalism of Grothendieck's derivators, we construct `the
universal localizing invariant' of dg categories. By this, we mean a
morphism $\mathcal{U}_l$ from the pointed derivator $\mathsf{HO}(\mathsf{dgcat})$ associated
with the Morita homotopy theory of dg categories to a triangulated
strong derivator $\mathcal{M}_{dg}^{loc}$ such that $\mathcal{U}_l$ commutes with filtered homotopy
colimits, preserves the point, sends each exact sequence of dg categories to a triangle and is universal for these properties.

Similary, we construct the `universal additive invariant' of dg categories, i.e. the
universal morphism of derivators $\mathcal{U}_a$ from
$\mathsf{HO}(\mathsf{dgcat})$ to a strong triangulated derivator
$\mathcal{M}_{dg}^{add}$ which satisfies the first two properties but
the third one only for split exact sequences. We prove that Waldhausen's
$K$-theory becomes co-representable in the target of the universal
additive invariant. This is the first conceptual characterization of Quillen-Waldhausen's $K$-theory since its definition in the early $70$'s. As an application we obtain for free the higher Chern characters from $K$-theory to cyclic homology.
\end{abstract}

\maketitle

\tableofcontents

\section{Introduction}
Differential graded categories (=dg categories) enhance our
understanding of triangulated categories appearing in algebra and
geometry, see \cite{ICM}.

They are considered as non-commutative schemes by Drinfeld
\cite{Drinfeld} \cite{Chitalk} and Kontsevich \cite{IHP} \cite{ENS} in their program
of non-commutative algebraic geometry, i.e. the study of dg categories
and their homological invariants.

In this article, using the formalism of Grothendieck's derivators, we construct `the
universal localizing invariant' of dg categories \emph{cf.}~\cite{Keller2002}. By this, we mean a
morphism $\mathcal{U}_l$ from the pointed derivator $\mathsf{HO}(\mathsf{dgcat})$ associated
with the Morita homotopy theory of dg categories, see \cite{addendum} \cite{IMRN} \cite{cras}, to a triangulated
strong derivator $\mathcal{M}_{dg}^{loc}$ such that $\mathcal{U}_l$ commutes with filtered homotopy
colimits, preserves the point, sends each exact sequence of dg categories to a triangle and is universal for these properties.
Because of its universality property reminiscent of motives, see section $4.1$ of Kontsevich's preprint \cite{motiv}, we call $\mathcal{M}_{dg}^{loc}$
the (stable) {\em localizing motivator} of dg categories.

Similary, we construct `the universal additive invariant' of dg categories, i.e. the
universal morphism of derivators $\mathcal{U}_a$ from
$\mathsf{HO}(\mathsf{dgcat})$ to a strong triangulated derivator
$\mathcal{M}_{dg}^{add}$ which satisfies the first two properties but
the third one only for split exact sequences. We call
$\mathcal{M}_{dg}^{add}$ the {\em additive motivator} of dg
categories. 

We prove that Waldhausen's $K$-theory spectrum appears as a spectrum of morphisms in the
base category $\mathcal{M}_{dg}^{add}(e)$ of the additive
motivator. This shows us that Waldhausen's $K$-theory is completely characterized by its additive property and `intuitively' it is {\em the} universal construction with values in a stable context which satisfies additivity.

To the best of the author's knowledge, this is the first
conceptual characterization of Quillen-Waldhausen's $K$-theory
\cite{Quillen1} \cite{Wald} since its definition in the early $70$'s. This result gives us a completely new way to think about algebraic $K$-theory and furnishes us for free the higher Chern characters from $K$-theory to cyclic homology~\cite{Loday}.

The co-representation of $K$-theory as a spectrum of morphisms extends our
results in \cite{addendum} \cite{IMRN}, where we co-represented $K_0$
using functors with values in {\em additive categories} rather than
morphisms of derivators with values in strong {\em triangulated derivators}.

For example, the mixed complex construction \cite{cyclic}, from which
all variants of cyclic homology can be deduced, and the
non-connective algebraic $\mbox{K}$-theory \cite{Marco} are localizing invariants and
factor through $\mathcal{U}_l$ and through $\mathcal{U}_a$. The
connective algebraic $K$-theory \cite{Wald} is an example of an additive
invariant which is not a localizing one. We prove that it becomes
co-representable in $\mathcal{M}^{add}_{dg}$, see theorem~\ref{corep}. 

Our construction is similar in spirit to those of Meyer-Nest
\cite{Nest}, Corti{\~n}as-Thom \cite{Cortinas} and Garkusha
\cite{Garkusha}. It splits into several general steps and also offers
some insight into the relationship between the theory of derivators
\cite{Heller} \cite{Grothendieck} \cite{KellerUniv} \cite{Malt}
\cite{Cis-Nee} and the classical theory
of Quillen model categories \cite{Quillen}. Derivators allow us to
state and prove precise universal properties and to dispense with many
of the technical problems one faces in using model categories.

In chapter~\ref{pre} we recall the notion of Grothendieck derivator and point out its connexion with that of small homotopy theory in the sense of Heller \cite{Heller}. In chapter~\ref{extension}, we recall Cisinski's
theory of derived Kan extensions \cite{Cisinski} and in chapter~\ref{localisation}, we
develop his ideas on the Bousfield localization of derivators
\cite{Letter}. In particular, we characterize the derivator associated
with a left Bousfield localization of a Quillen model category by a
universal property, see theorem~\ref{Cisinsk}. This is based on a
constructive description of the local weak equivalences.

In chapter~\ref{homotopy}, starting from a Quillen model category
$\mathcal{M}$ satisfying some compactness conditions, we construct a morphism of prederivators
$$ \mathsf{HO}(\mathcal{M})
\stackrel{\mathbb{R}\underline{h}}{\longrightarrow}
\mathsf{L}_{\Sigma}\mathsf{Hot}_{\mathcal{M}_f}$$
which commutes with filtered homotopy colimits, has a derivator as
target and is universal for these properties. In chapter~\ref{chappoint} we
study morphisms of pointed derivators and in chapter~\ref{small} we prove a
general result which garantees that small weak generators are
preserved under left Bousfield localizations. In chapter~\ref{spectra}, we
recall Heller's stabilization construction \cite{Heller} and we prove
that this construction takes `finitely generated' unstable theories to
compactly generated stable ones. We establish the connection between
Heller's stabilization and Hovey/Schwede's stabilization \cite{Spectra}
\cite{Schwede} by proving that if we start with a pointed Quillen
model category which satisfies some mild `generation' hypotheses, then
the two stabilization procedures yield equivalent results. This allows us to
characterize Hovey/Schwede's construction by a universal property and
in particular to give a very simple characterisation of the classical
category of spectra in the sense of Bousfield-Friedlander \cite{Bos-Fri}. In
chapter~\ref{chapquotient}, by applying the general arguments of the previous
chapters to the Morita homotopy theory of dg categories \cite{addendum} \cite{IMRN} \cite{cras}, we construct the universal morphism of derivators
$$ \mathcal{U}_t: \, \mathsf{HO}(\mathsf{dgcat}) \longrightarrow
\mathsf{St}(\mathsf{L}_{\Sigma,P}\mathsf{Hot}_{\mathsf{dgcat}_f})$$
which commutes with filtered homotopy colimits, preserves the
point and has a strong triangulated derivator as target. For every inclusion $\mathcal{A} \stackrel{K}{\hookrightarrow}
\mathcal{B}$ of a full dg subcategory, we have an induced morphism
$$ S_K : \, \mathsf{cone}(\mathcal{U}_t(\mathcal{A}
\stackrel{K}{\hookrightarrow} \mathcal{B})) \rightarrow
\mathcal{U}_t(\mathcal{B}/\mathcal{A})\,,$$
where $\mathcal{B}/\mathcal{A}$ denotes Drinfeld's dg quotient. By
applying the localization techniques of section~\ref{localisation} and using the fact
that the derivator
$\mathsf{St}(\mathsf{L}_{\Sigma,P}\mathsf{Hot}_{\mathsf{dgcat}_f})$
admits a stable Quillen model, we invert the morphisms $S_K$ and obtain finally the universal localizing invariant
of dg categories
$$ \mathcal{U}_l:\, \mathsf{HO}(\mathsf{dgcat}) \longrightarrow
\mathcal{M}_{dg}^{loc}\,.$$
We establish a connection between the triangulated category
$\mathcal{M}_{dg}^{loc}(e)$ and Waldhausen's $K$-theory by showing that
Waldhausen's $S_{\bullet}$-construction corresponds to the suspension
functor in $\mathcal{M}_{dg}^{loc}(e)$. 
In section~\ref{chapspectra}, we prove that the derivator $\mathcal{M}_{dg}^{loc}$
admits a stable Quillen model given by a left Bousfield localization
of a category of presheaves of spectra.
In section~\ref{trimat}, we introduce the concept of upper triangular
dg category and construct a Quillen model structure on this class of dg
categories, which satisfies strong compactness conditions. In
section~\ref{splitex}, we establish the connection between upper
triangular dg categories and split short exact sequences and use the
Quillen model structure of section~\ref{trimat} to prove an
`approximation result', see proposition~\ref{aproxsplit}. In
section~\ref{quasi}, by applying the techniques of
section~\ref{localisation}, we construct the universal morphism of
derivators
$$\mathcal{U}_u: \mathsf{HO}(\mathsf{dgcat}) \longrightarrow \mathcal{M}_{dg}^{unst}$$
which commutes with filtered homotopy colimits, preserves the point and
sends each split short exact sequence to a homotopy cofiber
sequence. We prove that Waldhausen's $K$-theory space construction
appears as a fibrant object in $\cm^{unst}_{dg}$. This allow us to obtain the weak equivalence of simplicial sets
$$ \mathsf{Map}(\mathcal{U}_u(k),S^1\wedge \mathcal{U}_u(\ca)) \stackrel{\sim}{\longrightarrow} | N.wS_{\bullet}\ca_f|$$
and the isomorphisms
$$ \pi_{i+1}\mathsf{Map}(\cu_u(k), S^1\wedge \cu_u(\ca)) \stackrel{\sim}{\longrightarrow} K_i(\ca), \,\,\, \forall i\geq 0\,,$$
see proposition~\ref{k-theoryunst}.

In section~\ref{univadit} we stabilize the derivator $\mathcal{M}_{dg}^{unst}$, using the fact that it
admits a Quillen model and obtain finally the universal additive
invariant of dg categories
$$ \mathcal{U}_a : \mathsf{HO}(\mathsf{dgcat}) \longrightarrow
\mathcal{M}^{add}_{dg} \,.$$
Connective algebraic $K$-theory is additive and so factors through
$\mathcal{U}_a$. We prove that for a small dg category $\ca$ its connective algebraic $K$-theory corresponds to a fibrant resolution of $\mathcal{U}_a(\mathcal{A})[1]$, see theorem~\ref{fibres}.
Using the fact that the Quillen model for $\cm^{add}_{dg}$ is enriched over spectra, we prove our main co-representability theorem.

Let $\ca$ and $\cb$ be small dg categories with $\ca \in \mathsf{dgcat}_f$.
\begin{theorem}[\ref{corep}]
We have the following weak equivalence of spectra
$$ \mathsf{Hom}^{\mathsf{Sp}^{\mathbb{N}}}(\cu_a(\ca), \cu_a(\cb)[1]) \stackrel{\sim}{\longrightarrow} K^c(\mathsf{rep}_{mor}(\ca, \cb))\,,$$
where $K^c(\mathsf{rep}_{mor}(\ca, \cb))$ denotes Waldhausen's connective $K$-theory spectrum of $\mathsf{rep}_{mor}(\ca, \cb)$.
\end{theorem}
In the triangulated base category $\cm^{add}_{dg}(e)$ of the additive motivator we have:
\begin{proposition}[\ref{true}]
We have the following isomorphisms of abelian groups
$$ \mathsf{Hom}_{\cm^{add}_{dg}(e)} (\cu_a(\ca), \cu_a(\cb)[-n]) \stackrel{\sim}{\longrightarrow} K_n(\mathsf{rep}_{mor}(\ca, \cb)),\,\, \forall n \geq 0\,.$$
\end{proposition}
\begin{remark}
Notice that if in the above theorem (resp. proposition), we consider $\ca=k$, we have
$$ \mathsf{Hom}^{\mathsf{Sp}^{\mathbb{N}}}(\cu_a(k), \cu_a(\cb)[1]) \stackrel{\sim}{\longrightarrow} K^c(\cb),\,\, \text{resp}.$$
$$ \mathsf{Hom}_{\cm^{add}_{dg}(e)} (\cu_a(k), \cu_a(\cb)[-n]) \stackrel{\sim}{\longrightarrow} K_n(\cb), \,\, \forall n\geq 0\,.$$
This shows that Waldhausen's connective $K$-theory spectrum (resp. groups) becomes co-representable in $\cm^{add}_{dg}$, resp. in $\cm^{add}_{dg}(e)$.
\end{remark}

In section~\ref{Chern}, we show that our co-representability theorem furnishes us for free the higher Chern characters from $K$-theory to cyclic homology. 
\begin{theorem}[\ref{Chern1}]
The co-representability theorem furnishes us the higher Chern characters
$$ ch_{n,r}: K_n(-) \longrightarrow HC_{n+2r}(-), \,\, n, r \geq 0\,.$$
\end{theorem}
In section~\ref{vistas}, we point out some questions that deserve
further investigation.

\section{Acknowledgments}

This article is part of my Ph.D. thesis under the supervision of
Prof. B.~Keller. It is a great pleasure to thank him for suggesting
this problem to me and for countless useful discussions. 

I have learned a lot from discussions with D.-C.~Cisinski. I deeply
thank him for sharing so generously his ideas on {\em Localisation des D{\'e}rivateurs}. I am also very grateful to B.~To{\"e}n for important remarks and for his interest. 

%I would like also to thank the anonymous referees for theirs comments and suggestions.

\section{Preliminaries}\label{pre}
In this section we recall the notion of Grothendieck derivator following~\cite{Cis-Nee}.

\begin{notation}\label{not1}
We denote by $CAT$, resp. $Cat$ the $2$-category of categories, resp. small categories. The empty category will be written $\emptyset$, and the $1$-point category (i.e. the category with one object and one identity morphism) will be written $e$. If $X$ is a small category, $X^{op}$ is the opposite category associated to $X$. If $u:X \rightarrow Y$ is a functor, and if $y$ is an object of $Y$, one defines the category $X/y$ as follows: the objects are the couples $(x,f)$, where $x$ is an object of $X$, and $f$ is a map in $Y$ from $u(x)$ to $y$; a map from $(x,f)$ to $(x',f')$ in $X/y$ is a map $\xi: x\rightarrow x'$ in $X$ such that $f'u(\xi)=f$. The composition law in $X/y$ is induced by the composition law in $X$. Dually, one defines $y\backslash X$ by the formula $y\backslash X=(X^{op}/y)^{op}$. We then have the canonical functors
$$ X/y \rightarrow X \,\,\,\, \text{and}\,\,\,\, y\backslash X \rightarrow X$$
defined by projection $(x,f) \mapsto x$. One can easily check that one gets the following pullback squares of categories
$$
\xymatrix{
X/y \ar[r] \ar[d]_{u/ y} & X \ar[d]^u & & y\backslash X \ar[d]_{y\backslash u} \ar[r] & X \ar[d]^u \\
Y/y \ar[r] & Y && y\backslash Y \ar[r] & Y \,.
}
$$
If $X$ is any category we let $p_X:X \rightarrow e$ be the canonical projection functor. Given any object $x$ of $X$ we will write $x:e \rightarrow X$ for the unique functor which sends the object of $e$ to $x$. The objects of $X$, or equivalently the functors $e \rightarrow X$, will be called the {\em points} of $X$.

If $X$ and $Y$ are two categories we denote by $\mathsf{Fun}(X,Y)$ the category of functors from $X$ to $Y$. If $\mathcal{C}$ is a $2$-category one writes $\mathcal{C}^{op}$ for its dual $2$-category: $\mathcal{C}^{op}$ has the same objects as $\mathcal{C}$ and, for any two objects $X$ and $Y$, the category $\mathsf{Fun}_{\mathcal{C}^{op}}(X,Y)$ of $1$-arrows from $X$ to $Y$ in $\mathcal{C}^{op}$ is $\mathsf{Fun}_{\mathcal{C}}(Y,X)^{op}$.
\end{notation}

\begin{definition}
A {\em prederivator} is a strict $2$-functor from $Cat^{op}$ to the $2$-category of categories
$$ \mathbb{D}: Cat^{op} \longrightarrow CAT\,.$$
More explicity: for any small category $X \in Cat$ one has a category $\mathbb{D}(X)$. For any functor $u:X \rightarrow Y$ in $Cat$ one gets a functor
$$ u^{\ast}=\mathbb{D}(u): \mathbb{D}(Y) \longrightarrow \mathbb{D}(X)\,.$$
For any morphism of functors
$$
\xymatrix{
X \ar@/^1pc/[rr]^u \ar@/_1pc/[rr]_v & \Downarrow \alpha & Y\,,
}
$$
one has a morphism of functors
$$
\xymatrix{
\mathbb{D}(X) & \alpha^{\ast} \Uparrow & \ar@/^1pc/[ll]^{v^{\ast}} \ar@/_1pc/[ll]_{u^{\ast}} \mathbb{D}(Y)\,.
}
$$
Of course, all these data have to verify some coherence conditions, namely:
\begin{itemize}
\item[(a)] For any composable maps in $Cat$, $X \stackrel{u}{\rightarrow} Y \stackrel{v}{\rightarrow} Z$,
$$ (vu)^{\ast} = u^{\ast} v^{\ast}\,\,\,\, \text{and}\,\,\,\, 1^{\ast}_X = 1_{\mathbb{D}(X)}\,.$$
\item[(b)] For any composable $2$-cells in $Cat$
$$
\xymatrix{
 X \ar[rr]^{\Downarrow \alpha}_{\Downarrow \beta} \ar@/^1pc/[rr]^u \ar@/_1pc/[rr]_w & & Y\,
}
$$
we have
$$ (\beta \alpha)^{\ast} =\alpha^{\ast} \beta^{\ast}\,\,\,\, \text{and}\,\,\,\, 1_u^{\ast}=1_{u^{\ast}}\,.$$
\item[(c)] For any $2$-diagram in $Cat$
$$ 
\xymatrix{
X \ar@/^1pc/[rr]^u \ar@/_1pc/[rr]_{u'} & \Downarrow \alpha & Y \ar@/^1pc/[rr]^v \ar@/_1pc/[rr]_{v'} & \Downarrow \beta & Z\,,
}
$$
we have $(\beta \alpha)^{\ast}= \alpha^{\ast} \beta^{\ast}\,.$
\end{itemize}
\end{definition}

\begin{example}\label{example1}
Let $M$ be a small category. The prederivator $\underline{M}$ naturally associated with $M$ is defined as
$$ X \mapsto \underline{M}(X)$$
where $\underline{M}(X)= \mathsf{Fun}(X^{op},M)$ is the category of presheaves over $X$ with values in $M$.
\end{example}

\begin{example}\label{example2}
Let $\cm$ be a category endowed with a class of maps called weak equivalences (for example $\cm$ can be the category of bounded complexes in a given abelian category, and the weak equivalences can be the quasi-isomorphisms). For any small category $X$, we define the weak equivalences in $\cm(X)$ to be the morphisms of presheaves which are termwise weak equivalences in $\cm$. We can then define $\mathbb{D}_{\cm}(X)$ as the localization of $\cm(X)$ by the weak equivalences. It is clear that, for any functor $u:X \rightarrow Y$ , the inverse image functor
$$
\begin{array}{rcl}
\cm(Y) & \rightarrow & \cm(X)\\
F & \mapsto & u^{\ast}(F)=F\circ u
\end{array}
$$
respects weak equivalences, so that it induces a well defined functor
$$ u^{\ast}: \mathbb{D}_{\cm}(X) \longrightarrow \mathbb{D}_{\cm}(X)\,.$$
The $2$-functoriality of localization implies that we have a prederivator $\mathbb{D}_{\cm}$.
\end{example}
Let $X$ be a small category and let $x$ be an object of $X$. Given an object $F \in \mathbb{D}(X)$ we will write $F_x=x^{\ast}(F)$. The object $F_x$ will be called the {\em fiber of $F$ at the point $x$}.

For a prederivator $\mathbb{D}$, define its {\em oposite} to be the prederivator $\mathbb{D}^{op}$ given by the formula $\mathbb{D}^{op}(X)=\mathbb{D}(X^{op})^{op}$ for all small categories $X$.

\begin{definition}\label{def1}
Let $\mathbb{D}$ be a prederivator. A map $u:X \rightarrow Y$ in $Cat$ has a {\em cohomological direct image functor} (resp. a {\em homological direct image functor}) in $\mathbb{D}$ if the inverse image functor
$$ u^{\ast}: \mathbb{D}(Y) \longrightarrow \mathbb{D}(X)$$
has a right adjoint (resp. a left adjoint)
$$ u_{\ast}: \mathbb{D}(X) \longrightarrow \mathbb{D}(Y) \,\,\,\,\,\,(\text{resp.}\,\,\, u_!: \mathbb{D}(X) \rightarrow \mathbb{D}(Y))\,,$$
called the {\em cohomological direct image functor} (resp. {\em homological direct image functor}) associated to $u$. 
\end{definition}

\begin{notation}\label{not2}
Let $X$ be a small category and $p=p_X:X \rightarrow e$. If $p$ has a cohomological direct image functor in $\mathbb{D}$, one defines for any object $F$ of $\mathbb{D}(X)$ the object of global sections of $F$ as
$$ \Gamma_{\ast}(X,F) = p_{\ast}(F)\,.$$
Dually, if $p$ has a homological direct image in $\mathbb{D}$ then, for any object $F$ of $\mathbb{D}(X)$, one sets
$$ \Gamma_!(X,F)=p_!(F)\,.$$
\end{notation}

\begin{notation}\label{not3}
Let $\mathbb{D}$ be a prederivator, and let
$$
\xymatrix{
X' \ar[r]^v \ar[d]_{u'} & X \ar[d]^u \ar@{}[dl]|{\Downarrow \alpha}\\
Y' \ar[r]_w & Y 
}
$$
be a $2$-diagram in $Cat$. By $2$-functoriality one obtains the following $2$-diagram
$$ 
\xymatrix{
\mathbb{D}(X') & \mathbb{D}(X) \ar[l]_{v^{\ast}} \\
\mathbb{D}(Y') \ar[u]^{u'^{\ast}} \ar@{}[ur]|{\Uparrow \alpha^{\ast}} & \mathbb{D}(Y) \ar[l]^{w^{\ast}} \ar[u]_{u^{\ast}} \,.
}
$$
If we assume that the functors $u$ and $u'$ both have cohomological direct images in $\mathbb{D}$ one can define the {\em base change morphism} induced by $\alpha$
$$ \beta: w^{\ast}u_{\ast} \rightarrow u'_{\ast}v^{\ast}$$
$$
\xymatrix{
\mathbb{D}(X') \ar[d]_{u_{\ast}'} & \mathbb{D}(X) \ar[l]_{v^{\ast}} \ar[d]^{u_{\ast}} \\
\mathbb{D}(Y') & \mathbb{D}(Y) \ar[l]^{w^{\ast}} \ar@{}[ul]|{\Uparrow \beta}
}
$$
as follows. The counit $u^{\ast}u_{\ast} \rightarrow 1_{\mathbb{D}(X)}$ induces a morphism $v^{\ast}u^{\ast}u_{\ast} \rightarrow v^{\ast}$, and by composition with $\alpha^{\ast}u_{\ast}$, a morphism $u'^{\ast}w^{\ast}u_{\ast} \rightarrow v^{\ast}$. This gives $\beta$ by adjunction.

This construction will be used in the following situation: let $u:X \rightarrow Y$ be a map in $Cat$ and let $y$ be a point of $Y$. According to notation \ref{not1} we have a functor $j: X/y \rightarrow X$, defined by the formula $j(x,f)=x$, where $f:u(x) \rightarrow y$ is a morphism in $Y$. If $p:X/y \rightarrow e$ is the canonical map one obtains the $2$-diagram below, where $\alpha$ denotes the $2$-cell defined by the formula $\alpha_{(x,f)}=f$
$$
\xymatrix{
X/y \ar[d]_p \ar[r]^j & X \ar[d]^u \ar@{}[dl]|{\Downarrow \alpha} \\
e \ar[r]_y & Y\,.
}
$$
\end{notation} 
Using notation~\ref{not2}, the associated base change morphism gives rise to a canonical morphism
$$ u_{\ast}(F)_y \rightarrow \Gamma_{\ast}(X/y, F/y)$$
for any object $F \in \mathbb{D}(X)$, where $F/y=j^{\ast}(F)$. Dually one has canonical morphisms
$$ \Gamma_!(y\backslash X, y\backslash F) \rightarrow u_!(F)_y$$
where $y\backslash F = k^{\ast}(F)$ and $k$ denotes the canonical functor from $y\backslash X$ to $X$.

\begin{notation}\label{not4}
Let $X$ and $Y$ be two small categories. Using the $2$-functoriality of $\mathbb{D}$, one defines a functor
$$d_{X,Y}: \mathbb{D}(X\times Y) \rightarrow \mathsf{Fun}(X^{op},\mathbb{D}(Y))$$
as follows. Setting $X'=X \times Y$, we have a canonical functor
$$ \mathsf{Fun}(Y,X')^{op} \longrightarrow \mathsf{Fun}(\mathbb{D}(X'), \mathbb{D}(Y))$$
which defines a functor
$$ \mathsf{Fun}(Y,X')^{op} \times \mathbb{D}(X') \longrightarrow \mathbb{D}(Y)$$
and then a functor
$$ \mathbb{D}(X') \longrightarrow \mathsf{Fun}(\mathsf{Fun}(Y,X')^{op}, \mathbb{D}(Y))\,.$$
Using the canonical functor
$$ X \rightarrow \mathsf{Fun}(Y,X \times Y), \,\,\,\, x \mapsto (y \mapsto (x,y))\,,$$
this gives the desired functor.

In particular, for any small category $X$, one gets a functor
$$ d_X=d_{X,e} : \mathbb{D}(X) \rightarrow \mathsf{Fun}(X^{op}, \mathbb{D}(e))\,.$$
If $F$ is an object of $\mathbb{D}(X)$, then $d_X(F)$ is the presheaf on $X$ with values in $\mathbb{D}(e)$ defined by 
$$ x \mapsto F_x\,.$$ 
\end{notation}

\begin{definition}
A {\em derivator} is a prederivator $\mathbb{D}$ with the following properties.
\begin{itemize}
\item[{\bf Der1}] (Non-triviality axiom). For any finite set $I$ and any family $\{X_i, i \in I\}$ of small categories, the canonical functor
$$ \mathbb{D}(\coprod_{i \in I} X_i) \longrightarrow \prod_{i \in I} \mathbb{D}(X_i)$$
is an equivalence of categories.
\item[{\bf Der2}] (Conservativity axiom). For any small category $X$, the family of functors
$$
\begin{array}{rcl}
 x^{\ast}: \mathbb{D}(X)&  \rightarrow &\mathbb{D}(e)\\ 
 F & \mapsto & x^{\ast}(F)=F_x
\end{array}
$$
corresponding to the points $x$ of $X$ is conservative. In other words: if $\varphi: F \rightarrow G$ is a morphism in $\mathbb{D}(X)$, such that for any point $x$ of $X$ the map $\varphi_x: F_x \rightarrow G_x$ is an isomophism in $\mathbb{D}(e)$, then $\varphi$ is an isomorphism in $\mathbb{D}(X)$.
\item[{\bf Der3}] (Direct image axiom). Any functor in $Cat$ has a cohomological direct image functor and a homological direct image functor in $\mathbb{D}$ (see definition~\ref{def1}).
\item[{\bf Der4}] (Base change axiom). For any functor $u: X \rightarrow Y$ in $Cat$, any point $y$ of $Y$ and any object $F$ in $\mathbb{D}(X)$, the canonical base change morphisms (see notation~\ref{not3})
$$ u_{\ast}(F)_y \rightarrow \Gamma_{\ast}(X/y, F/y) \,\,\,\, \text{and} \,\,\,\, \Gamma_!(y\backslash X, y\backslash F) \rightarrow u_!(F)_y$$
are isomorphisms in $\mathbb{D}(e)$.
\item[{\bf Der5}] (Essential surjectivity axiom). Let $I$ be the category corresponding to the graph
$$ 0 \leftarrow 1\,.$$
For any small category $X$, the functor
$$ d_{I,X}: \mathbb{D}(I\times X) \rightarrow \mathsf{Fun}(I^{op}, \mathbb{D}(X))$$
(see notation~\ref{not4}) is full and essentially surjective. 
\end{itemize}
\end{definition}
\begin{example}
By \cite{Cisinski1}, any Quillen model category $\cm$ gives rise to a derivator denoted $\mathsf{HO}(\cm)$. 
\end{example}
We denote by $\mathsf{Ho}(\mathcal{M})$
the homotopy category of $\mathcal{M}$. By definition, it equals $\mathsf{HO}(\mathcal{M})(e)$. 
\begin{definition}
A derivator $\mathbb{D}$ is {\em strong} if for every
finite free category $X$ and every small category $BY$, the natural
functor
$$ \mathbb{D}(X \times Y) \longrightarrow
\mathsf{Fun}(X^{op},\mathbb{D}(Y))\,,$$
see notation~\ref{not4}, is full and essentially surjective.
\end{definition}

Notice that a strong derivator is the same thing as a small homotopy
theory in the sense of Heller \cite{Heller}. Notice also that by proposition $2.15$ in
\cite{catder}, $\mathsf{HO}(\mathcal{M})$ is a strong derivator.

\begin{definition}
A derivator $\mathbb{D}$ is {\em regular} if in $\mathbb{D}$, sequential homotopy
colimits commute with finite products and homotopy pullbacks.
\end{definition}

\begin{notation}\label{not5}
Let $X$ be a category. Remember that a {\em sieve} (or a crible) in $X$ is a full subcategory $U$ of $X$ such that, for any object $x$ of $X$, if there exists a morphism $x \rightarrow u$ with $u$ in $U$ then $x$ is in $U$. Dually a {\em cosieve} (or a cocrible) in $X$ is a full subcategory $Z$ of $X$ such that, for any morphism $z \rightarrow x$ in $X$, if $z$ is in $Z$ then so is $x$.

A functor $j:U \rightarrow X$ is an {\em open immersion} if it is injective on objects, fully faithful, and if $j(U)$ is a sieve in $X$. Dually a functor $i:Z \rightarrow X$ is a {\em closed immersion} if it is injective on objects, fully faithful, and if $i(Z)$ is a cosieve in $X$. One can easily show that open immersions and closed immersions are stable by composition and pullback.
\end{notation}

\begin{definition}
A derivator $\mathbb{D}$ is {\em pointed} if it satisfies the following property.
\begin{itemize}
\item[{\bf Der6}] (Exceptional axiom). For any closed immersion $i:Z \rightarrow X$ in $Cat$ the cohomological direct image functor
$$i_{\ast}:\mathbb{D}(Z) \longrightarrow \mathbb{D}(X)$$
has a right adjoint
$$ i^!: \mathbb{D}(X) \longrightarrow \mathbb{D}(Z)$$
called the {\em exceptional inverse image functor} associated to $i$. Dually, for any open immersion $j:U \rightarrow X$ the homological direct image functor
$$ j_!: \mathbb{D}(U) \longrightarrow \mathbb{D}(X)$$
has a left adjoint
$$ j^{?}: \mathbb{D}(X) \longrightarrow \mathbb{D}(U)$$
called the {\em coexceptional inverse image functor} associated to $j$.
\end{itemize}
\end{definition}
Let $\square$ be the category given by the commutative square
$$
\xymatrix{
 (0,0) & (0,1) \ar[l] \\
 (1,0) \ar[u] & (1,1) \ar[l] \ar[u] \,.
}
$$
We are interested in two of its subcategories. The subcategory $\lr$ is
$$
\xymatrix{
 & (0,1)  \\
 (1,0)  & (1,1) \ar[l] \ar[u] \,.
}
$$
and $\ul$ is the subcategory
$$
\xymatrix{
 (0,0) & (0,1) \ar[l] \\
 (1,0) \ar[u] &  \,.
}
$$
We thus have two inclusion functors
$$ \sigma: \lr \rightarrow \square\,\,\,\, \text{and} \,\,\,\, \tau: \ul \rightarrow \square$$
($\sigma$ is an open immersion and $\tau$ a closed immersion). A {\em global commutative square} in $\mathbb{D}$ is an object of $\mathbb{D}(\square)$. A global commutative square $C$ in $\mathbb{D}$ is thus locally of shape
$$
\xymatrix{
C_{0,0} \ar[r] \ar[d] & C_{0,1} \ar[d] \\
C_{1,0} \ar[r] & C_{1,1}
}
$$
in $\mathbb{D}(e)$.

A global commutative square $C$ in $\mathbb{D}$ is {\em cartesian} (or a {\em homotopy pullback square}) if, for any global commutative square $B$ in $\mathbb{D}$, the canonical map
$$ \mathsf{Hom}_{\mathbb{D}(\square)}(B,C) \longrightarrow \mathsf{Hom}_{\mathbb{D}(\lrcorner)}(\sigma^{\ast}(B),\sigma^{\ast}(C))$$
is bijective. Dually, a global commutative square $B$ in $\mathbb{D}$ is {\em cocartesian} (or a {\em homotopy pushout square}) if, for any global commutative square $C$ in $\mathbb{D}$, the canonical map
$$ \mathsf{Hom}_{\mathbb{D}(\square)}(B,C) \longrightarrow \mathsf{Hom}_{\mathbb{D}(\ulcorner)}(\tau^{\ast}(B),\tau^{\ast}(C))$$ is bijective.

As $\square$ is isomorphic to its opposite $\square^{op}$, one can see that a global commutative square in $\mathbb{D}$ is cartesian (resp. cocartesian) if and only if is cocartesian (resp. cartesian) as a global commutative square in $\mathbb{D}^{op}$.

\begin{definition}
A derivator $\mathbb{D}$ is {\em triangulated} or {\em stable} if it is pointed and satisfies the following axiom:
\begin{itemize}
\item[{\bf Der7}] (Stability axiom). A global commutative square in $\mathbb{D}$ is cartesian if and only if it is cocartesian. 
\end{itemize}
\end{definition} 
\begin{theorem}[\cite{Malt1}]
For any triangulated derivator $\mathbb{D}$ and small category $X$ the category $\mathbb{D}(X)$ has a canonical triangulated structure.
\end{theorem}

Let $\mathbb{D}$ and $\mathbb{D}'$ be derivators. 
We denote by $\underline{\mathsf{Hom}}(\mathbb{D},\mathbb{D}')$ the
category of all morphisms of derivators, by
$\underline{\mathsf{Hom}}_{!}(\mathbb{D},\mathbb{D}')$ the
category of morphisms of derivators which commute with homotopy
colimits \cite[3.25]{Cisinski} and by
$\underline{\mathsf{Hom}}_{flt}(\mathbb{D},\mathbb{D}')$ the
category of morphisms of derivators which commute with filtered homotopy
colimits, see \cite[5]{Cis-Nee}.

\section{Derived Kan extensions}\label{extension}
Let $A$ be a small category and $\mathsf{Fun}(A^{op},Sset)$ the
Quillen model category of simplicial pre-sheaves on $A$, endowed with
the projective model structure, see \cite{HAG}.
We have at our disposal the functor
$$
\begin{array}{ccc}
A & \stackrel{h}{\longrightarrow} & \mathsf{Fun}(A^{op},Sset) \\
X & \longmapsto & \mathsf{Hom}_A(?,X)\,,
\end{array}
$$
where $\mathsf{Hom}_A(?,X)$ is considered as a constant simplicial set.

The functor $h$ gives rise to a morphism of pre-derivators
$$ \underline{A} \stackrel{h}{\longrightarrow}
\mathsf{HO}(\mathsf{Fun}(A^{op},Sset))\,.$$

Using the notation of \cite{Cisinski}, we denote by $\mathsf{Hot}_A$
the derivator $\mathsf{HO}(\mathsf{Fun}(A^{op}, Sset))$. The following
results are proven in \cite{Cisinski}.

Let $\mathbb{D}$ be a derivator.

\begin{theorem}\label{Cin}
The morphism $h$ induces an
equivalence of categories
$$
\xymatrix{
\underline{\mathsf{Hom}}(\underline{A},\mathbb{D})
\ar@<-1ex>[d]_{\varphi} \\
\underline{\mathsf{Hom}}_!(\mathsf{Hot}_A, \mathbb{D})
\ar@<-1ex>[u]_{h^{\ast}}\,.
}
$$
\end{theorem}

\begin{proof}
This theorem is equivalent to corollary $3.26$ in \cite{Cisinski},
since we have
$$ \underline{\mathsf{Hom}}(\underline{A},\mathbb{D}) \simeq \mathbb{D}(A^{op})\,.$$
\end{proof}

\begin{lemma}\label{ad}
We have an adjunction
$$
\xymatrix{
\underline{\mathsf{Hom}}(\mathsf{Hot}_A,\mathbb{D})
\ar@<2ex>[d]^{\Psi} \\
*+<1pc>{\underline{\mathsf{Hom}}_!(\mathsf{Hot}_A, \mathbb{D})}
\ar@{^{(}->}[u]^{inc} \,,
}
$$
where 
$$\Psi(F):=\varphi(F \circ h)\,.$$
\end{lemma}

\begin{proof}
We construct a universal $2$-morphism of functors
$$\epsilon : \, inc \circ \Psi \longrightarrow Id\,.$$
Let $F$ be a morphism of derivators belonging to
$\underline{\mathsf{Hom}}(\mathsf{Hot}_A,\mathbb{D})$. Let $L$ be a
small category and $X$ an object of $\mathsf{Hot}_A(L)$. Recall from
\cite{Cisinski} that we have the diagram
$$
\xymatrix{
 & \nabla \int X \ar[dl]_{\pi} \ar[dr]^{\varpi} & \\
L^{op} & & A\,.
}
$$
Now, let $p$ be the functor $\pi^{op}$ and $q$ the functor $\varpi^{op}$. By
the dual of proposition $1.15$ in \cite{Cisinski}, we have the
following functorial isomorphism
$$ p_! q^{\ast}(h) \stackrel{\sim}{\longrightarrow} X\,.$$
Finally let $\epsilon_L(X)$ be the composed morphism
$$ \epsilon_L(X) : \, \Psi(F)(X) = p_!q^{\ast}F(h) = p_!F(q^{\ast}h)
\rightarrow F(p_!q^{\ast}h) \stackrel{\sim}{\rightarrow}F(X)$$ and notice, using theorem~\ref{Cin}, that $\epsilon$ induces an adjunction.
\end{proof}

\section{Localization: model categories versus derivators}\label{localisation}
Let $\mathcal{M}$ be a left proper, cellular Quillen model
category, see \cite{Hirschhorn}. 

We start by fixing a frame on $\mathcal{M}$, see
definition $16.6.21$ in \cite{Hirschhorn}. Let $D$ be a
small category and $F$ a functor from $D$ to
$\mathcal{M}$. We denote by $\mbox{hocolim} \,F$ the object of
$\mathcal{M}$, as in definition $19.1.2$ of \cite{Hirschhorn}.
Let $S$ be a set of morphisms in $\mathcal{M}$ and denote by $\mathsf{L}_S\mathcal{M}$ the
  left Bousfield localization of $\mathcal{M}$ by $S$. 

Notice that the Quillen adjunction
$$
\xymatrix{
\mathcal{M} \ar@<-1ex>[d]_{Id} \\
\mathsf{L}_S \mathcal{M} \ar@<-1ex>[u]_{Id} \,,
}
$$
induces a morphism of derivators
$$ \gamma: \mathsf{HO}(\mathcal{M})
\stackrel{\mathbb{L}Id}{\longrightarrow}
\mathsf{HO}(\mathsf{L}_S \mathcal{M})$$
which commutes with homotopy colimits.

\begin{proposition}\label{Cisin}
Let $\mathcal{W}_S$ be the smallest class of morphisms in
$\mathcal{M}$ satisfying the following properties:
\begin{itemize}
\item[a)] Every element in $S$ belongs to $\mathcal{W}_S$.
\item[b)] Every weak equivalence of $\mathcal{M}$ belongs to
  $\mathcal{W}_S$.
\item[c)] If in a commutative triangle, two out of three morphisms
    belong to $\mathcal{W}_S$, then so does the third one. The class
    $\mathcal{W}_S$ is stable under retractions.
\item[d)] Let $D$ be a small category and $F$ and $G$ functors from
  $D$ to $\mathcal{M}$. If $\eta$ is a morphism of functors from $F$
  to $G$ such that for every object $d$ in $D$, $F(d)$ and $G(d)$ are cofibrant
  objects and the morphism $\eta(d)$
  belongs to $\mathcal{W}_S$, then so does the morphism
$$ \mbox{hocolim}\, F \longrightarrow \mbox{hocolim} \,G\,.$$
\end{itemize}
Then the class $\mathcal{W}_S$ equals the class of $S$-local
equivalences in $\mathcal{M}$, see \cite{Hirschhorn}.
\end{proposition}

\begin{proof}
The class of $S$-local equivalences satisfies properties $a)$,
$b)$, $c)$ and $d)$: Properties $a)$ and $b)$ are satisfied by
definition, propositions $3.2.3$ and $3.2.4$ in \cite{Hirschhorn} imply property
$c)$ and proposition $3.2.5$ in \cite{Hirschhorn} implies property $d)$.

Let us now show that conversely, each $S$-local equivalence is in $\mathcal{W}_S$.
Let 
$$X \stackrel{g}{\rightarrow} Y$$
 be an $S$-local equivalence in $\mathcal{M}$. Without loss of generality, we can suppose that $X$ is
 cofibrant. Indeed, let $Q(X)$ be a cofibrant resolution of $X$ and consider
 the diagram
$$
\xymatrix{
Q(X) \ar[d]_{\pi}^{\sim} \ar[dr]^{g\circ \pi} & \\
X \ar[r]_g & Y\,.
}
$$
Notice that since $\pi$ is a weak equivalence, $g$ is an $S$-local
equivalence if and only if $g \circ \pi$ is one.

By theorem $4.3.6$ in \cite{Hirschhorn}, $g$ is an $S$-local equivalence
if and only if the morphism $\mathsf{L}_S(g)$ appearing in the diagram
$$
\xymatrix{
X \ar[rr]^g \ar[d]_{j(X)} & &  Y \ar[d]^{j(Y)} \\
\mathsf{L}_SX \ar[rr]_{\mathsf{L}_S(g)} & &  \mathsf{L}_SY 
}
$$
is a weak equivalence in $\mathcal{M}$.
This shows that it is enough to prove that $j(X)$ and $j(Y)$ belong to
$\mathcal{W}_S$. Apply the small object argument to the morphism
$$ X \longrightarrow \ast $$
using the set $\widetilde{\Lambda(S)}$, see proposition $4.2.5$ in
\cite{Hirschhorn}. We have the factorization 
$$
\xymatrix{
X \ar[rr] \ar[dr]_{j(X)} & &  \ast \\
 & \mathsf{L}_S(X) \ar[ur] & ,
}
$$
where $j(X)$ is a relative $\widetilde{\Lambda(S)}$-cell complex.

We will now prove two stability conditions concerning the class
$\mathcal{W}_S$:
\begin{itemize}
\item[S1)] Consider the following push-out
$$
\xymatrix{
*+<1pc>{W_0} \ar[r] \ar@{>->}[d]_f  \ar@{}[dr]|{\lrcorner} & W_2 \ar[d]^{f_{\ast}} \\
W_1 \ar[r] & W_3 \,,
}
$$
where $W_0$, $W_1$ and $W_2$ are cofibrant objects in $\mathcal{M}$
and $f$ is a cofibration which belongs to $\mathcal{W}_S$. Observe that
$f_{\ast}$ corresponds to the colimit of the morphism of diagrams
$$
\xymatrix{
*+<1pc>{W_0} \ar@{=}[r] \ar@{>->}[d]_f & W_0 \ar@{=}[d] \ar[r] & W_2 \ar@{=}[d]
\\
W_1 & *+<1pc>{W_0} \ar@{>->}[l]^f \ar[r] & W_2 \,.
}
$$
Now, proposition $19.9.4$ in \cite{Hirschhorn} and property $d)$ 
imply that $f_{\ast}$ belongs to $\mathcal{W}_S$.

\item[S2)] Consider the following diagram
$$
\xymatrix{
\mathbf{X} :X_0 \,\,\ar@{>->}[r]^{f_0} & *+<1pc>{X_1} \ar@{>->}[r]^{f_1} &  *+<1pc>{X_2} \ar@{>->}[r]^{f_2} &  *+<1pc>{X_3} \ar@{>->}[r] & \ldots
}
$$
in $\mathcal{M}$, where the objects are cofibrant and the morphisms
are cofibrations which belong to the class $\mathcal{W}_S$. Observe that
the transfinite composition of $\mathbf{X}$ corresponds to the colimit of the
morphism of diagrams
$$
\xymatrix{
X_0 \ar@{=}[r] \ar@{=}[d] & *+<1pc>{X_0} \ar@{=}[r] \ar@{>->}[d]^{f_0} & *+<1pc>{X_0}
\ar@{=}[r] \ar@{>->}[d]^{f_1 \circ f_0} & \ldots \\
*+<1pc>{X_0} \ar@{>->}[r]_{f_0} & *+<1pc>{X_1} \ar@{>->}[r]_{f_1} & *+<1pc>{X_2} \ar@{>->}[r] &
\ldots \,\,. 
}
$$
Now, since $\mathbf{X}$ is a Reedy cofibrant diagram on category with
fibrant constants, see definition $15.10.1$ in \cite{Hirschhorn},
theorem $19.9.1$ from \cite{Hirschhorn} and property $d)$ imply that the transfinite
composition of $\mathbf{X}$ belongs to $\mathcal{W}_S$. Notice that
the above argument can be immediatly generalized to a transfinite
composition of a $\lambda$-sequence, where $\lambda$ denotes an
ordinal, see section $10.2$ in \cite{Hirschhorn}.
\end{itemize}

Now, the construction of the morphism $j(X)$ and the stability
conditions $S1)$ and $S2)$ shows us that it is enough to prove that the elements of
$\widetilde{\Lambda(S)}$ belong to $\mathcal{W}_S$.
By proposition $4.2.5$ in \cite{Hirschhorn}, it is sufficient to show
that the set 
$$ \Lambda(S)= \{ \tilde{{\bf A}}\otimes \Delta[n]
\underset{\tilde{{\bf A}}\otimes
  \partial \Delta[n]}{\amalg} \tilde{{\bf B}}\otimes \Delta[n]
\stackrel{\Lambda(g)}{\longrightarrow} \tilde{{\bf B}}\otimes \Delta[n] \,|\,
(A \stackrel{g}{\rightarrow} B) \in S, \, n \geq 0\} \,,$$
of horns in $S$ is contained in $\mathcal{W}_S$. Recall from definition $4.2.1$ in
\cite{Hirschhorn} that $\tilde{g}: \tilde{{\bf A}} \rightarrow
\tilde{{\bf B}}$
denotes a cosimplicial resolution of $g: A \rightarrow B$ and $\tilde{g}$
is a Reedy cofibration.
We have the diagram
$$
\xymatrix{
\tilde{{\bf A}} \otimes \partial \Delta[n] \ar[d]_{1 \otimes i}
\ar[rr]^{\tilde{g} \otimes 1}  & &   \tilde{{\bf B}} \otimes \partial
\Delta[n] \ar[d]^{1 \otimes i} \\
\tilde{{\bf A}} \otimes \Delta[n] \ar[rr]_{\tilde{g} \otimes 1}  & &
\tilde{{\bf B}} \otimes \Delta[n] \,.
}
$$
Observe that the morphism
$$ \tilde{{\bf A}} \otimes \Delta[n] \stackrel{\tilde{g} \otimes 1}{\longrightarrow}
\tilde{{\bf B}} \otimes \Delta[n]$$
identifies with 
$$ \tilde{{\bf A}}^n \stackrel{\tilde{g}^n}{\longrightarrow}
\tilde{{\bf B}}^n\,,$$
and so belongs to $\mathcal{W}_S$. Now, the morphism
$$ \tilde{{\bf A}}\otimes \partial \Delta[n] \stackrel{\tilde{g}\otimes \mathbf{1}_{\partial \Delta[n]}}{\longrightarrow}
\tilde{{\bf B}}\otimes \partial \Delta[n]$$
corresponds to the induced map of latching objects
$$\mathsf{L}_n \tilde{{\bf A}} \longrightarrow  \mathsf{L}_n \tilde{{\bf B}}\,,$$
which is a cofibration in $\mathcal{M}$ by proposition $15.3.11$ in
\cite{Hirschhorn}.

Now by propositions $15.10.4$, $16.3.12$ and theorem $19.9.1$ of
\cite{Hirschhorn}, we have the following commutative diagram
$$
\xymatrix{
\underset{\partial(\stackrel{\rightarrow}{\Delta}\downarrow
  [n])}{\mbox{hocolim}} \, \tilde{{\bf A}} \ar[r] \ar[d]_{\sim} & \underset{\partial(\stackrel{\rightarrow}{\Delta}\downarrow
  [n])}{\mbox{hocolim}} \, \tilde{{\bf B}} \ar[d]^{\sim} \\
\mathsf{L}_n \tilde{{\bf A}} \ar[r] & \mathsf{L}_n \tilde{{\bf B}}\,,
}
$$
where the vertival arrows are weak equivalences and $\partial(\stackrel{\rightarrow}{\Delta}\downarrow
  [n])$ denotes the category of strictly increasing maps with target $[n]$. By property $d)$ of
the class $\mathcal{W}_S$, we conclude that $\tilde{g} \otimes \mathbf{1}_{\partial\Delta[n]}$
belongs to $\mathcal{W}_S$.

We have the following diagram
$$
\xymatrix{
\tilde{{\bf A}}\otimes \Delta[n] \underset{\tilde{{\bf A}}\otimes
  \partial \Delta[n]}{\amalg} \tilde{{\bf B}}\otimes \Delta[n]
\ar[drr]^{\Lambda(g)} & & \\
\tilde{{\bf A}} \otimes \Delta[n] \ar[u]^I \ar[rr]_{\tilde{g}\otimes
  1} & & \tilde{{\bf B}} \otimes \Delta[n]\,.
}
$$
Notice that the morphism $I$ belongs to $\mathcal{W}_S$ by the
stability condition $S1)$ applied to the morphism
$$ \tilde{{\bf A}} \otimes \partial \Delta[n]  \stackrel{\tilde{g}
  \otimes 1}{\longrightarrow} \tilde{{\bf B}} \otimes \partial \Delta[n]\,,$$
which is a cofibration and belongs to $\mathcal{W}_S$. Since the
morphism $\tilde{g}\otimes 1$ belongs to $\mathcal{W}_S$ so does $\Lambda(g)$.

This proves the proposition.
\end{proof}
Let $\mathbb{D}$ be derivator and $S$ a class of
morphisms in $\mathbb{D}(e)$.

\begin{definition}[Cisinski \cite{Letter}]
The derivator $\mathbb{D}$ admits a {\em left Bousfield localization}
by $S$ if there exists a morphism of derivators
$$ \gamma : \mathbb{D} \rightarrow \mathsf{L}_S\mathbb{D}\,,$$
which commutes with homotopy colimits, sends the elements of $S$ to
isomorphisms in $\mathsf{L}_S\mathbb{D}(e)$ and satisfies the
universal property: for every derivator $\mathbb{D}'$ the morphism
$\gamma$ induces an equivalence of categories
$$\underline{\mathsf{Hom}}_!(\mathsf{L}_S\mathbb{D},\mathbb{D}')
  \stackrel{\gamma^{\ast}}{\longrightarrow}
  \underline{\mathsf{Hom}}_{!,S}(\mathbb{D},\mathbb{D}')\,,$$
where $\underline{\mathsf{Hom}}_{!,S}(\mathbb{D},\mathbb{D}')$ denotes
the category of morphisms of derivators which commute with homotopy
colimits and send the elements of $S$ to isomorphisms in $\mathbb{D}'(e)$.
\end{definition}

\begin{lemma}\label{lettri}
Suppose that $\mathbb{D}$ is a triangulated derivator, $S$ is stable
under the loop space functor $\Omega(e) : \mathbb{D}(e) \rightarrow
\mathbb{D}(e)$, see \cite{Cis-Nee}, and $\mathbb{D}$ admits a left
Bousfield localization $\mathsf{L}_S\mathbb{D}$ by $S$.

Then $\mathsf{L}_S\mathbb{D}$ is also a triangulated derivator. 
\end{lemma}

\begin{proof}
Recall from \cite{Cis-Nee} that since $\mathbb{D}$ is a triangulated
derivator, we have the following equivalence
$$
\xymatrix{
\mathbb{D} \ar@<1ex>[d]^{\Omega} \\
\mathbb{D} \ar@<1ex>[u]^{\Sigma}
}
$$
Notice that both morphisms of derivators, $\Sigma$ and $\Omega$,
commute with homotopy colimits. Since $S$ is stable under the functor
$\Omega(e): \mathbb{D}(e) \rightarrow \mathbb{D}(e)$ and $\mathbb{D}$
admits a left Bousfield localization $\mathsf{L}_S\mathbb{D}$ by $S$,
we have an induced morphism 
$$ \Omega: \mathsf{L}_S\mathbb{D} \rightarrow
\mathsf{L}_S\mathbb{D}\,.$$
Let $s$ be an element of $S$. We now show that the image of $s$ by the
functor $\gamma \circ \Sigma$ is an isomorphism in
$\mathsf{L}_S\mathbb{D}(e)$. For this consider the category $\ul$, see
section~\ref{pre}, and the functors
$$
\begin{array}{rcl}
(0,0): e \rightarrow \ul & \mbox{and} & p:\ul \rightarrow e\,.
\end{array}
$$
Now recall from section $7$ from \cite{Heller} that 
$$ \Omega(e) := p_! \circ (0,0)_{\ast}\,.$$
This description shows us that the image of $s$ under the functor
$\gamma \circ \Sigma$ is an isomorphism in $\mathsf{L}_S\mathbb{D}(e)$
because $\gamma$ commutes with homotopy colimits. In conclusion, we
have an induced adjunction
$$
\xymatrix{
\mathsf{L}_S\mathbb{D} \ar@<1ex>[d]^{\Omega}\\
\mathsf{L}_S\mathbb{D} \ar@<1ex>[u]^{\Sigma}
}
$$ 
which is clearly an equivalence. This proves the lemma.
\end{proof}

\begin{theorem}[Cisinski \cite{Letter}]\label{Cisinsk}
The morphism of derivators
$$\gamma: \mathsf{HO}(\mathcal{M})
\stackrel{\mathbb{L}Id}{\longrightarrow} \mathsf{HO}(\mathsf{L}_S
\mathcal{M})$$
is a left Bousfield localization of $\mathsf{HO}(\mathcal{M})$ by the
image of the set $S$ in $\mathsf{Ho}(\mathcal{M})$.
\end{theorem}

\begin{proof}
Let $\mathbb{D}$ be a derivator.

The morphism $\gamma $ admits a fully faithful right adjoint 
$$\sigma : \, \mathsf{HO}(\mathsf{L}_S \mathcal{M})
\longrightarrow \mathsf{HO}(\mathcal{M})\,.$$
Therefore, the induced functor
$$ \gamma^{\ast} : \underline{\mathsf{Hom}}_!(\mathsf{HO}(\mathsf{L}_S
\mathcal{M}),
\mathbb{D}) \longrightarrow
\underline{\mathsf{Hom}}_{!,S}(\mathsf{HO}(\mathcal{M}),
\mathbb{D})\,,$$
admits a left adjoint $\sigma^{\ast}$ and $\sigma^{\ast}\gamma^{\ast}=
(\gamma \sigma)^{\ast}$ is isomorphic to the identity. Therefore
$\gamma^{\ast}$ is fully faithful. We now show that
$\gamma^*$ is essentially surjective. Let $F$ be an object of
$\underline{\mathsf{Hom}}_{!,S}(\mathsf{HO}(\mathcal{M}),
\mathbb{D})$. Notice that since $\mathbb{D}$ satisfies the conservativity
axiom, it is sufficient to show that the functor
$$ F(e): \mathsf{Ho}(\mathcal{M}) \rightarrow \mathbb{D}(e)$$
sends the images in $\mathsf{Ho}(\mathcal{M})$ of $S$-local
equivalences of $\mathcal{M}$ to isomorphisms in $\mathbb{D}(e)$.
The morphism $F$ then becomes naturally a morphism of derivators
$$ \overline{F}: \mathsf{HO}(\mathsf{L}_S\mathcal{M}) \rightarrow
\mathbb{D}$$
such that $\gamma^*(\overline{F})=F$. Now, since $F$ commutes with
homotopy colimits, the functor
$$ \mathcal{M} \rightarrow \mathsf{Ho}(\mathcal{M}) \rightarrow
\mathbb{D}(e)$$
sends the elements of $\mathcal{W}_S$ to isomorphisms. This proves the
theorem since by proposition~\ref{Cisin} the class $\mathcal{W}_S$ equals the class of $S$-local equivalences in $\mathcal{M}$.
\end{proof}

\section{Filtered homotopy colimits}\label{homotopy}
Let $\mathcal{M}$ be a cellular Quillen model category, with $I$ the
set of generating cofibrations. Suppose that the domains and codomains
of the elements of $I$ are cofibrant, $\aleph_0$-compact,
$\aleph_0$-small and homotopically finitely presented, see definition
$2.1.1$ in \cite{Toen-Vaq}.

\begin{example}\label{mori}
Consider the quasi-equivalent, resp. quasi-equiconic, resp. Morita,
Quillen model structure on $\mathsf{dgcat}$ constructed in \cite{addendum} \cite{IMRN} \cite{cras}.

Recall that a dg functor $F:\cc \rightarrow \ce$ is a
quasi-equivalence, resp. quasi-equiconic, 
resp. a Morita dg functor,  if it satisfies one of the following conditions $\mathsf{C}1)$
or $\mathsf{C}2)$:
\begin{itemize}
\item[C1)] The dg category $\mathcal{C}$ is empty and all the objects
  of $\mathcal{E}$ are contractible.
\item[C2)] For every object $c_1, c_2 \in \cc$, the
  morphism of complexes from $\mathsf{Hom}_{\mathcal{C}}(c_1,c_2)$ to
  $\mathsf{Hom}_{\mathcal{E}}(F(c_1)),F(c_2))$ is a quasi-isomorphism
  and the functor $\mathsf{H}^0(F)$,
  resp. $\mathsf{H}^0(\mbox{pre-tr}(F))$, resp. $\mathsf{H}^0(\mbox{pre-tr}(F))^{\kar}$, is essentially surjective.
\end{itemize}
Observe that the domains and codomains of the set $I$ of generating
cofibrations in $\mathsf{dgcat}$ satisfy the conditions above for all
the Quillen model structures.
\end{example}

The following proposition is a simplification of proposition $2.2$ in
\cite{Toen-Vaq}.

\begin{proposition}\label{prop}
Let $\mathcal{M}$ be a Quillen model category which satisfies the
conditions above. Then 
\begin{itemize}
\item[1)] A filtered colimit of trivial fibrations is a trivial
  fibration.
\item[2)] For any filtered diagram $X_i$ in $\mathcal{M}$, the natural
  morphism
$$ \underset{i \in I}{\mbox{hocolim}}\, X_i \longrightarrow
  \underset{i \in I}{\mbox{colim}}\, X_i$$
is an isomorphism in $\mathsf{Ho}(\mathcal{M})$.
\item[3)] Any object $X$ in $\mathcal{M}$ is equivalent to a filtered
  colimit of strict finite $I$-cell objects.
\item[4)] An object $X$ in $\mathcal{M}$ is homotopically finitely
  presented if and only if it is equivalent to a rectract of a strict
  finite $I$-cell object.
\end{itemize}
\end{proposition}

\begin{proof}
The proof of $1)$, $2)$ and $3)$ is exactly the same as that of proposition
$2.2$ in \cite{Toen-Vaq}. The proof of $4)$ is also the same once we
observe that the domains and codomains of the elements of the set $I$
are already homotopically finitely presented by hypothesis.
\end{proof}

In everything that follows, we fix:
\begin{itemize}
\item[-] A co-simplicial resolution functor 
$$(\Gamma (-) : \mathcal{M} \rightarrow \mathcal{M}^{\Delta}, \, i)$$
in the model category
  $\mathcal{M}$, see definition $16.1.8$ in \cite{Hirschhorn}. This
  means that for every object $X$ in $\mathcal{M}$, $\Gamma(X)$ is
  cofibrant in the Reedy model structure on $\mathcal{M}^{\Delta}$ and 
$$ i(X) : \Gamma(X) \stackrel{\sim}{\longrightarrow} c^{\ast}(X)$$
is a weak equivalence on $\mathcal{M}^{\Delta}$, where $c^{\ast}(X)$
denotes the constant co-simplicial object associated with $X$. 
\item[-] A fibrant resolution functor 
$$( (-)_f : \mathcal{M} \rightarrow \mathcal{M}, \, \epsilon )$$
in the model category $\mathcal{M}$, see \cite{Hirschhorn}.
\end{itemize}

\begin{definition}
Let $\mathcal{M}_f$ be the smallest full subcategory of $\mathcal{M}$
such that
\begin{itemize}
\item[-]  $\mathcal{M}_f$ contains (a representative of the
  isomorphism class of) each strictly finite $I$-cell object of
  $\mathcal{M}$ and
\item[-] the category $\mathcal{M}_f$ is stable under the functors $(-)_f$ and $\Gamma(-)^n,\, n \geq 0$.
\end{itemize}
\end{definition}

\begin{remark}
Notice that $\mathcal{M}_f$ is a small category and that every object
in $\mathcal{M}_f$ is weakly equivalent to a strict finite $I$-cell.
\end{remark}

We have the inclusion
$$ \mathcal{M}_f \stackrel{I}{\hookrightarrow} \mathcal{M}\,. $$

\begin{definition}
Let $S$ be the set of pre-images of the weak equivalences in
$\mathcal{M}$ under the functor $i$.
\end{definition}

\begin{lemma}\label{fulfat}
The induced functor
$$ \mathcal{M}_f[S^{-1}] \stackrel{\mathsf{Ho}(I)}{\longrightarrow} \mathsf{Ho}(\mathcal{M})$$
is fully faithful, where $\mathcal{M}_f[S^{-1}]$ denotes the
localization of $\mathcal{M}$ by the set $S$.
\end{lemma}

\begin{proof}
Let $X$, $Y$ be objects of $\mathcal{M}_f$. Notice that $(Y)_f$ is a fibrant
resolution of $Y$ in $\mathcal{M}$ which belongs to $\mathcal{M}_f$ and 
$$ 
\xymatrix{
\Gamma(X)^0 \coprod \Gamma(X)^0 \ar[d]_{d^0 \coprod
  d^1} \ar[r] &  \Gamma(X)^0\\
\Gamma(X)^1 \ar[ur]_{s^0}
}
$$
is a cylinder object for $\Gamma(X)^0$, see proposition $16.1.6.$ from
\cite{Hirschhorn}. Since $\cm_f$ is also stable under the functors $\Gamma(-)^n, \, n\geq 0$, this cylinder object also belongs to $\mathcal{M}_f$. This implies that if in the construction of the homotopy category $\mathsf{Ho}(\cm)$, as in theorem $8.35$ of \cite{Hirschhorn}, we restrict ourselves to $\cm_f$ we recover $\cm_f[S^{-1}]$ as a full subcategory of $\mathsf{Ho}(\cm)$. This implies the lemma.
\end{proof}

We denote by $\mathsf{Fun}(\mathcal{M}_f^{op},Sset)$ the Quillen model
category of simplicial pre-sheaves on $\mathcal{M}_f$ endowed with the
projective model structure, see section~\ref{extension}.
Let $\Sigma$ be the image in $\mathsf{Fun}(\mathcal{M}_f^{op},Sset)$ by the functor $h$, see
section~\ref{extension}, of the set $S$ in $\mathcal{M}_f$. Since the
category $\mathsf{Fun}(\mathcal{M}_f^{op},Sset)$ is cellular and left
proper, its left Bousfield localization by the set $\Sigma$ exists,
see \cite{Hirschhorn}. We denote it by
$\mathsf{L}_{\Sigma}\mathsf{Fun}(\mathcal{M}_f^{op},Sset)$. We have a
composed functor that we still denote by $h$
$$h: \mathcal{M}_f \rightarrow \mathsf{Fun}(\mathcal{M}_f^{op},Sset)
\stackrel{Id}{\rightarrow} \mathsf{L}_{\Sigma}\mathsf{Fun}(\mathcal{M}_f^{op},Sset)\,.$$

Now, consider the functor
$$
\begin{array}{ccc}
\underline{h}: \mathcal{M} & \longrightarrow &
\mathsf{Fun}(\mathcal{M}_f^{op},Sset)\\
X & \longmapsto & \mathsf{Hom}(\Gamma(-),X)_{|\mathcal{M}_f}\,.
\end{array}
$$
We also have a composed functor that we still denote by $\underline{h}$

$$\underline{h}: \mathcal{M} \rightarrow \mathsf{Fun}(\mathcal{M}_f^{op},Sset)
\stackrel{Id}{\rightarrow} \mathsf{L}_{\Sigma}\mathsf{Fun}(\mathcal{M}_f^{op},Sset)\,.$$

Now, observe that the natural equivalence
$$ i(-) : \Gamma(-) \longrightarrow c^{\ast}(-)\,,$$
induces, for every object $X$ in $\mathcal{M}_f$, a morphism $\Psi(X)$
in $\mathsf{L}_{\Sigma}\mathsf{Fun}(\mathcal{M}_f^{op},Sset)$
$$\Psi(X): \, h(X) = \mathsf{Hom}(c^{\ast}(-),X) \longrightarrow
\mathsf{Hom}(\Gamma(-),X) =: (\underline{h} \circ I)(X)\,,$$
which is functorial in $X$.

\begin{lemma}
The functor $\underline{h}$ preserves weak equivalences between
fibrant objects.
\end{lemma}

\begin{proof}
Let $X$ be a fibrant object in $\mathcal{M}$. We have an equivalence
$$ \mathsf{Hom}(\Gamma(Y),X) \stackrel{\sim}{\longrightarrow}
\mathsf{Map}_{\mathcal{M}}(Y,X)\,,$$
see \cite{Hirschhorn}. This implies the lemma.
\end{proof}

\begin{remark}
The previous lemma implies that the functor $\underline{h}$ admits a right
derived functor
$$
\begin{array}{ccc}
\mathbb{R}\underline{h}: \mathsf{Ho}(\mathcal{M}) & \longrightarrow &
\mathsf{Ho}(\mathsf{L}_{\Sigma}\mathsf{Fun}(\mathcal{M}_f^{op},Sset))\\
X & \longmapsto & \mathsf{Hom}(\Gamma(-),X_f)_{|\mathcal{M}_f}\,.
\end{array}
$$
\end{remark}

Since the functor
$$h: \mathcal{M}_f \rightarrow \mathsf{L}_{\Sigma}\mathsf{Fun}(\mathcal{M}_f^{op},Sset)\,,$$ 
sends, by definition, the elements of $S$ to weak equivalences, we have an induced morphism
$$ \mathsf{Ho}(h): \mathcal{M}_f[S^{-1}] \rightarrow \mathsf{Ho}(\mathsf{L}_{\Sigma}\mathsf{Fun}(\mathcal{M}_f^{op},Sset))\,.$$

\begin{remark}\label{funcpont}
Notice that lemma $4.2.2$ from \cite{HAG} implies that for every $X$ in
$\mathcal{M}_f$, the morphism $\Psi(X)$
$$ \Psi(X):\, \mathsf{Ho}(h)(X) \longrightarrow
(\mathbb{R}\underline{h} \circ \mathsf{Ho}(I))(X)$$
is an isomorphism in
$\mathsf{Ho}(\mathsf{L}_{\Sigma}\mathsf{Fun}(\mathcal{M}_f^{op},Sset))$.
\end{remark}

This shows that the functors
$$ \mathsf{Ho}(h), \, \mathbb{R}\underline{h} \circ \mathsf{Ho}(I) : \,
\mathcal{M}_f[S^{-1}] \rightarrow \mathsf{Ho}(\mathsf{L}_{\Sigma}\mathsf{Fun}(\mathcal{M}_f^{op},Sset)$$
are canonically isomorphic and so we have the following diagram
$$
\xymatrix{
\mathcal{M}_f[S^{-1}] \ar[rr]^{\mathsf{Ho}(I)}
\ar[d]_{\mathsf{Ho}(h)} & &  \mathsf{Ho}(\mathcal{M})
\ar[dll]^{\mathbb{R}\underline{h}}\,, \\
\mathsf{Ho}(\mathsf{L}_{\Sigma}\mathsf{Fun}(\mathcal{M}_f^o,Sset)) & &  
}
$$
which is commutative up to isomorphism.

\begin{lemma}\label{filtered}
The functor $\mathbb{R}\underline{h}$ commutes with filtered homotopy colimits.
\end{lemma}

\begin{proof}
Let $\{Y_i\}_{i \in I}$ be a filtered diagram in $\mathcal{M}$. We
can suppose, without loss of generality, that $Y_i$ is fibrant in
$\mathcal{M}$. By proposition~\ref{prop}, the natural
morphism
$$ \underset{i \in I}{\mbox{hocolim}} \,Y_i \longrightarrow \underset{i \in
  I}{\mbox{colim}} \,Y_i$$
is an isomorphism in $\mathsf{Ho}(\mathcal{M})$ and $\underset{i \in I}{\mbox{colim}}
\,Y_i$ is also fibrant. Since the functor
$$ \mathsf{Ho}(\mathsf{Fun}(\mathcal{M}_f^{op},Sset))
\stackrel{\mathbb{L}Id}{\longrightarrow}
\mathsf{Ho}(\mathsf{L}_{\Sigma}\mathsf{Fun}(\mathcal{M}_f^{op},Sset))\,,$$
commutes with homotopy colimits and in 
$\mathsf{Ho}(\mathsf{Fun}(\mathcal{M}_f^{op},Sset))$ they are calculated objectwise, it is
sufficient to show that the morphism
$$ \underset{i \in I}{\mbox{hocolim}} \,\mathbb{R}\underline{h}(Y_i)(X)
\longrightarrow \mathbb{R}\underline{h}(\underset{i \in I}{\mbox{colim}}
\,Y_i)(X)$$
is an isomorphism in $\mathsf{Ho}(Sset)$, for every object $X$ in
$\mathcal{M}_f$. Now, since every object $X$ in $\mathcal{M}_f$ is homotopically finitely presented,
see proposition~\ref{prop}, we have the following equivalences:
$$
\begin{array}{rcl}
\mathbb{R}\underline{h}(\underset{i \in I}{\mbox{colim}} \,Y_i)(X) & =
& 
\mathsf{Hom}(\Gamma (X), \underset{i \in I}{\mbox{colim}} \,Y_i)\\
 & \simeq  & \mathsf{Map}(\Gamma (X), \underset{i \in I}{\mbox{colim}} \,Y_i)\\ 
& \simeq & \underset{i \in I}{\mbox{colim}}\, \mathsf{Map}(X,Y_i)\\
& \simeq & \underset{i \in I}{\mbox{hocolim}} \,\mathbb{R}\underline{h}(Y_i)(X)
\end{array}
$$
This proves the lemma.
\end{proof}

We now denote by $\mathsf{L}_{\Sigma}\mathsf{Hot}_{\mathcal{M}_f}$ the derivator associated with
$\mathsf{L}_{\Sigma}\mathsf{Fun}(\mathcal{M}_f^{op},Sset)$ and by
$\underline{\mathcal{M}_f}[S^{-1}]$ the pre-derivator
$\underline{\mathcal{M}_f}$ localized at the set $S$, see examples \ref{example1} and \ref{example2}.

Observe that the morphism of functors
$$ \Psi : h \longrightarrow \underline{h}\circ I$$
induces a $2$-morphism of derivators
$$ \overline{\Psi}: \, \mathsf{Ho}(h) \longrightarrow \mathbb{R}\underline{h}\circ
\mathsf{Ho}(I)\,.$$

\begin{lemma}\label{2-mor}
The $2$-morphism $\overline{\Psi}$ is an isomorphism.
\end{lemma}

\begin{proof}
For the terminal category $e$, the $2$-morphism $\overline{\Psi}$
coincides with the morphism of functors of
remark~\ref{funcpont}. Since this one is an isomorphism, so is
$\overline{\Psi}$ by conservativity. This proves the lemma.
\end{proof}

As before, we have the following diagram
$$
\xymatrix{
\underline{\mathcal{M}_f}[S^{-1}] \ar[rr]^{\mathsf{Ho}(I)}
\ar[d]_{\mathsf{Ho}(h)} & &  \mathsf{HO}(\mathcal{M})
\ar[dll]^{\mathbb{R}\underline{h}}\,, \\
\mathsf{L}_{\Sigma}\mathsf{Hot}_{\mathcal{M}_f} & &  
}
$$
which is commutative up to isomorphism in the $2$-category of
pre-derivators. Notice that by lemma~\ref{filtered},
$\mathbb{R}\underline{h}$ commutes with filtered homotopy colimits.

Let $\mathbb{D}$ be a derivator.

\begin{lemma}\label{colim}
The morphism of pre-derivators
$$ \underline{\mathcal{M}_f}[S^{-1}]
\stackrel{\mathsf{Ho}(h)}{\longrightarrow}
\mathsf{L}_{\Sigma}\mathsf{Hot}_{\mathcal{M}_f} \,,$$
induces an equivalence of categories
$$ \underline{\mathsf{Hom}}_!(\mathsf{L}_{\Sigma}\mathsf{Hot}_{\mathcal{M}_f},
\mathbb{D}) \stackrel{\mathsf{Ho}(h)^{\ast}}{\longrightarrow}
\underline{\mathsf{Hom}}(\underline{\mathcal{M}_f}[S^{-1}], \mathbb{D})\,.$$
\end{lemma}

\begin{proof}
The category
$\underline{\mathsf{Hom}}_!(\mathsf{L}_{\Sigma}\mathsf{Hot}_{\mathcal{M}_f},\mathbb{D})$
is equivalent, by theorem~\ref{Cisinsk}, to the category $\underline{\mathsf{Hom}}_{!,\Sigma}(\mathsf{Hot}_{\mathcal{M}_f},
\mathbb{D})$. This last
category identifies, under the equivalence
$$ \underline{\mathsf{Hom}}_!(\mathsf{Hot}_{\mathcal{M}_f}, \mathbb{D})
\rightarrow
\underline{\mathsf{Hom}}(\underline{\mathcal{M}_f},\mathbb{D})$$
given by theorem~\ref{Cin}, with the full subcategory of
$\underline{\mathsf{Hom}}(\underline{\mathcal{M}_f},\mathbb{D})$
consisting of the morphisms of pre-derivators which send the elements
of $S$ to isomorphisms in $\mathbb{D}(e)$. Now observe that this last
category identifies with
$\underline{\mathsf{Hom}}(\underline{\mathcal{M}_f}[S^{-1}],
\mathbb{D})$, by definition of the localized pre-derivator $\underline{\mathcal{M}_f}[S^{-1}]$.
This proves the lemma.
\end{proof}

Recall from section $9.5$ in \cite{Dugger} that the
co-simplicial resolution functor $\Gamma(-)$ that we have fixed in the
beginning of this section allows us to construct a Quillen adjunction:
$$
\xymatrix{
\mathcal{M} \ar@<1ex>[d]^{\underline{h}=sing} \\
\mathsf{Fun}(\mathcal{M}_f^{op},Sset) \ar@<1ex>[u]^{Re} \,.
}
$$
Since the functor $Re$ sends the elements of $\Sigma$ to weak
equivalences in $\mathcal{M}$, we have the following Quillen adjunction
$$
\xymatrix{
\mathcal{M} \ar@<1ex>[d]^{\underline{h}} \\
\mathsf{L}_{\Sigma}\mathsf{Fun}(\mathcal{M}_f^{op},Sset)
\ar@<1ex>[u]^{Re}\,,
}
$$
and a natural weak equivalence
$$ \eta : Re \circ h \stackrel{\sim}{\longrightarrow} I\,,$$
see \cite{Dugger}.

This implies that we have the following diagram
$$
\xymatrix{
\mathcal{M}_f[S^{-1}] \ar[rr]^{\mathsf{Ho}(I)}
\ar[d]_{\mathsf{Ho}(h)} & &  \mathsf{Ho}(\mathcal{M})
\ar@<1ex>[dll]^{\mathbb{R}\underline{h}}\,, \\
\mathsf{Ho}(\mathsf{L}_{\Sigma}\mathsf{Fun}(\mathcal{M}_f^{op},Sset))
\ar@<1ex>[urr]^{\mathbb{L}Re} & &  
}
$$
which is commutative up to isomorphism.

We now claim that $\mathbb{L}Re \circ \mathbb{R}\underline{h}$
is naturally isomorphic to the identity. Indeed, by
proposition~\ref{prop}, each object of $\mathcal{M}$ is isomorphic in
$\mathsf{Ho}(\mathcal{M})$, to a filtered colimit of strict finite
$I$-cell objects. Since
$\mathbb{R}\underline{h}$ and $\mathbb{L}Re$ commute with filtered
homotopy colimits and $\mathbb{L}Re \circ
\mathsf{Ho}(h) \simeq \mbox{Id}$, we conclude that
$\mathbb{L}Re \circ \mathbb{R}\underline{h}$ is naturally
isomorphic to the identity. This implies that the morphism
$\mathbb{R}\underline{h}$ is fully faithful.

Now, observe that the natural weak
equivalence $\eta$ induces a $2$-isomorphism and so we obtain the
following diagram
$$
\xymatrix{
\underline{\mathcal{M}_f}[S^{-1}] \ar[rr]^{\mathsf{Ho}(I)}
\ar[d]_{\mathsf{Ho}(h)} & &  \mathsf{HO}(\mathcal{M})
\ar@<1ex>[dll]^{\mathbb{R}\underline{h}}\,, \\
\mathsf{L}_{\Sigma}\mathsf{Hot}_{\mathcal{M}_f}
\ar@<1ex>[urr]^{\mathbb{L}Re} & &  
}
$$
which is commutative up to isomorphism in the $2$-category of pre-derivators. Notice that $\mathbb{L}Re \circ
\mathbb{R}\underline{h}$ is naturally isomorphic to the identity (by
conservativity) and
so the morphism of derivators $\mathbb{R}\underline{h}$ is fully faithful.

Let $\mathbb{D}$ be a derivator.

\begin{theorem}\label{flt}
The morphism of derivators
$$ \mathsf{HO}(\mathcal{M})
\stackrel{\mathbb{R}\underline{h}}{\longrightarrow}
\mathsf{L}_{\Sigma}\mathsf{Hot}_{\mathcal{M}_f}\,,$$
induces an equivalence of categories
$$ \underline{\mathsf{Hom}}_!(\mathsf{L}_{\Sigma}\mathsf{Hot}_{\mathcal{M}_f},
\mathbb{D}) \stackrel{\mathbb{R}\underline{h}^{\ast}}{\longrightarrow}
\underline{\mathsf{Hom}}_{flt}(\mathsf{HO}(\mathcal{M}), \mathbb{D})\,,$$
where $\underline{\mathsf{Hom}}_{flt}(\mathsf{HO}(\mathcal{M}),
\mathbb{D})$ denotes the category of morphisms of derivators which
commute with filtered homotopy colimits.
\end{theorem}

\begin{proof}
We have the following adjunction
$$
\xymatrix{
\underline{\mathsf{Hom}}(\mathsf{HO}(\mathcal{M}), \mathbb{D})
\ar@<1ex>[d]^{\mathbb{L}{Re}^{\ast}} \\
\underline{\mathsf{Hom}}(\mathsf{L}_{\Sigma}\mathsf{Hot}_{\mathcal{M}_f},
\mathbb{D}) \ar@<1ex>[u]^{\mathbb{R} \underline{h}^{\ast}} \,,
}
$$
with ${\mathbb{R}\underline{h}}^{\ast}$ a fully faithful functor.

Now notice that the adjunction of lemma~\ref{ad} induces naturally an
adjunction
$$
\xymatrix{
\underline{\mathsf{Hom}}(\mathsf{L}_{\Sigma}\mathsf{Hot}_{\mathcal{M}_f},\mathbb{D})
\ar@<2ex>[d]^{\Psi} \\
\underline{\mathsf{Hom}}_!(\mathsf{L}_{\Sigma}\mathsf{Hot}_{\mathcal{M}_f}, \mathbb{D})
\ar@{^{(}->}[u] \,.
}
$$
This implies that the composed functor
$$ {\mathbb{R}\underline{h}}^{\ast}: \, \underline{\mathsf{Hom}}_!(\mathsf{L}_{\Sigma}\mathsf{Hot}_{\mathcal{M}_f},
\mathbb{D}) \longrightarrow
\underline{\mathsf{Hom}}_{flt}(\mathsf{HO}(\mathcal{M}), \mathbb{D})$$
is fully faithful.

We now show that this functor is essentially surjective.

Let $F$ be an object of
$\underline{\mathsf{Hom}}_{flt}(\mathsf{HO}(\mathcal{M}),\mathbb{D})$.
Consider the morphism
$$ \mathbb{L}{Re}^{\ast}(F) := F \circ \mathbb{L}Re \,.$$
Notice that this morphism does not necessarily commute with homotopy
colimits. 
Now, by the above adjunction, we have a universal $2$-morphism
$$ \varphi : \, \Psi({\mathbb{L}Re}^{\ast}(F)) \longrightarrow
{\mathbb{L}Re}^{\ast}(F)\,.$$
Consider the $2$-morphism
$${\mathbb{R}\underline{h}}^{\ast}:\,
{\mathbb{R}\underline{h}}^{\ast}((\Psi \circ
{\mathbb{L}Re}^{\ast})(F)) \longrightarrow
({\mathbb{R}\underline{h}}^{\ast} \circ {\mathbb{L}Re}^{\ast})(F)
\simeq F\,.$$

Now, we will show that this $2$-morphism is a $2$-isomorphism. By
conservativity, it is
sufficient to show this for the case of the terminal
category $e$. For this, observe that ${\mathbb{R}\underline{h}}^{\ast}(\varphi)$ induces an isomorphism
$$ \Psi({\mathbb{L}Re}^{\ast}(F)) \circ \mathbb{R}\underline{h} \circ
\mathsf{Ho}(I) \longrightarrow F \circ \mathsf{Ho}(I)\,.$$
Now each object of $\mathcal{M}$ is isomorphic, in
$\mathsf{Ho}(\mathcal{M})$, to a filtered colimit of strict finite
$I$-cell objects. Since $F$ and
$\Psi({\mathbb{L}Re}^{\ast}(F))$ commute with filtered homotopy
colimits, ${\mathbb{R}\underline{h}}^{\ast}(\varphi)$ induces an
isomorphism.
This shows that the functor $\mathbb{R}\underline{h}^{\ast}$ is
essentially surjective. 

This proves the theorem.
\end{proof}

\section{Pointed derivators}\label{chappoint}
Recall from the previous section that we have constructed a derivator
$\mathsf{L}_{\Sigma}\mathsf{Hot}_{\mathcal{M}_f}$ associated with a
Quillen model category $\mathcal{M}$ satisfying suitable compactness assumptions. 

Now suppose that $\mathsf{Ho}(\mathcal{M})$ is pointed, i.e. that the
morphism
$$ \emptyset \longrightarrow \ast \,,$$
in $\mathcal{M}$, where $\emptyset$ denotes the initial object and $\ast$
the terminal one, is a weak equivalence.
Consider the morphism
$$ P: \widetilde{\emptyset} \longrightarrow h(\emptyset)\,,$$
where $\widetilde{\emptyset}$ denotes the initial object in $\mathsf{L}_{\Sigma}\mathsf{Fun}(\mathcal{M}_f^{op},Sset)$.

Observe that, since $\mathbb{R}\underline{h}$ admits a left adjoint,
$h(\emptyset)$ identifies with the terminal object in 
$$ \mathsf{Ho}(\mathsf{L}_{\Sigma}\mathsf{Fun}(\mathcal{M}_f^{op},Sset))\,,$$
because
$$h(\emptyset)=  \mathsf{Ho}(h)(\emptyset) \stackrel{\sim}{\rightarrow}
\mathbb{R}\underline{h}\circ \mathsf{Ho}(I)(\emptyset) \stackrel{\sim}{\rightarrow} \mathbb{R}\underline{h}(\ast)\,.$$
We denote by 
$$ \mathsf{L}_{\Sigma,P}\mathsf{Fun}(\mathcal{M}_f^{op},Sset)\,,$$
the left Bousfield localization of
$\mathsf{L}_{\Sigma}\mathsf{Fun}(\mathcal{M}_f^{op},Sset)$ at the
morphism $P$.

Notice that the category
$$
\mathsf{Ho}(\mathsf{L}_{\Sigma,P}\mathsf{Fun}(\mathcal{M}_f^{op},Sset))\,,$$
is now a pointed one.

We have the following morphisms of derivators
$$
\xymatrix{
\mathsf{Ho}(\mathcal{M}) \ar@<1ex>[d]^{\mathbb{R}\underline{h}} \\
\mathsf{L}_{\Sigma}\mathsf{Hot}_{\mathcal{M}_f}
\ar@<1ex>^{\mathbb{L}Re}[u] \ar[d]_{\Phi} \\
\mathsf{L}_{\Sigma,P}\mathsf{Hot}_{\mathcal{M}_f}. 
}
$$
By construction, we have a pointed morphism of derivators
$$ \mathsf{HO}(\mathcal{M}) \stackrel{\Phi \circ
  \mathbb{R}\underline{h}}{\longrightarrow}
\mathsf{L}_{\Sigma,P}\mathsf{Hot}_{\mathcal{M}_f}\,,$$
which commutes with filtered homotopy colimits and preserves the point.

Let $\mathbb{D}$ be a pointed derivator.

\begin{proposition}\label{ext}
The morphism of derivators $\Phi \circ \mathbb{R}\underline{h}$
induces an equivalence of categories
$$\underline{\mathsf{Hom}}_!(\mathsf{L}_{\Sigma,P}\mathsf{Hot}_{\mathcal{M}_f},\mathbb{D})
\stackrel{(\Phi \circ
  \mathbb{R}\underline{h})^{\ast}}{\longrightarrow}
\underline{\mathsf{Hom}}_{flt,p}(\mathsf{HO}(\mathcal{M}),\mathbb{D})\,,$$
where
$\underline{\mathsf{Hom}}_{flt,p}(\mathsf{HO}(\mathcal{M}),\mathbb{D})$
denotes the category of morphisms of derivators which commute with
filtered homotopy colimits and preserve the point.
\end{proposition}

\begin{proof}
By theorem~\ref{Cisinsk}, we have an equivalence of categories
$$\underline{\mathsf{Hom}}_!(\mathsf{L}_{\Sigma,P}\mathsf{Hot}_{\mathcal{M}_f},\mathbb{D})
\stackrel{\Phi^{\ast}}{\longrightarrow}
\underline{\mathsf{Hom}}_{!,P}(\mathsf{L}_{\Sigma}\mathsf{Hot}_{\mathcal{M}_f},\mathbb{D})\,.$$
By theorem~\ref{flt}, we have an equivalence of categories
$$\underline{\mathsf{Hom}}_!(\mathsf{L}_{\Sigma}\mathsf{Hot}_{\mathcal{M}_f},\mathbb{D})
\stackrel{\mathbb{R}\underline{h}^{\ast}}{\longrightarrow}
\underline{\mathsf{Hom}}_{flt}(\mathsf{HO}(\mathcal{M}),\mathbb{D})\,.$$
We now show that under this last equivalence, the category 
$\underline{\mathsf{Hom}}_{!,P}(\mathsf{L}_{\Sigma}\mathsf{Hot}_{\mathcal{M}_f},\mathbb{D})$
identifies with
$\underline{\mathsf{Hom}}_{flt,p}(\mathsf{HO}(\mathcal{M}),\mathbb{D})$.
Let $F$ be an object of
$\underline{\mathsf{Hom}}_{!,P}(\mathsf{L}_{\Sigma}\mathsf{Hot}_{\mathcal{M}_f},\mathbb{D})$.
Since $F$ commutes with homotopy colimits, it preserves the initial
object. This implies that $F \circ \mathbb{R}\underline{h}$ belongs to $\underline{\mathsf{Hom}}_{flt,p}(\mathsf{HO}(\mathcal{M},\mathbb{D})\,.$

Let now $G$ be an object of
$\underline{\mathsf{Hom}}_{flt,p}(\mathsf{HO}(\mathcal{M}),\mathbb{D})$.
Consider, as in the proof of theorem~\ref{flt}, the morphism
$$ \Psi({\mathbb{L}Re}^{\ast}(G)):
\mathsf{L}_{\Sigma}\mathsf{Hot}_{\mathcal{M}_f} \longrightarrow
\mathbb{D}\,.$$

Since $\Psi({\mathbb{L}Re}^{\ast}(G))$ commutes with homotopy
colimits, by construction, it sends $\widetilde{\emptyset}$ to the
point of $\mathbb{D}$. Observe also that $h(\emptyset)$ is also sent to
the point of $\mathbb{D}$ because
$$\Psi({\mathbb{L}Re}^{\ast}(G))(h(\emptyset)) \simeq
G(\emptyset)\,.$$

This proves the proposition.

\end{proof}

\section{Small weak generators}\label{small}

Let $\mathcal{N}$ be a pointed, left proper, compactly generated
Quillen model category as in definition $2.1$ of \cite{Toen-Vaq}. Observe that in particular this implies that
$\mathcal{N}$ is finitely generated, as in section $7.4$ in \cite{Hovey}.
We denote by $\mathcal{G}$ the set of cofibers of the generating
cofibrations $I$ in $\mathcal{N}$. By corollary $7.4.4$ in
\cite{Hovey}, the set $\mathcal{G}$ is a set of small weak generators
for $\mathsf{Ho}(\mathcal{N})$, see definitions $7.2.1$ and $7.2.2$ in
\cite{Hovey}.
Let $S$ be a set of morphisms in
$\mathcal{N}$ between objects which are homotopically finitely
presented, see \cite{Toen-Vaq}, and $\mathsf{L}_S\mathcal{N}$ the left Bousfield localization
of $\mathcal{N}$ by $S$. We have an adjunction
$$
\xymatrix{
\mathsf{Ho}(\mathcal{N}) \ar@<-1ex>[d]_{\mathbb{L}Id} \\
\mathsf{Ho}(\mathsf{L}_S\mathcal{N}) \ar@<-1ex>[u]_{\mathbb{R}Id}\,.
}
$$
\begin{lemma}\label{gener}
The image of the set $\mathcal{G}$ under the functor $\mathbb{L}Id$ is a set of small weak generators in
$\mathsf{Ho}(\mathsf{L}_S\mathcal{N})$.
\end{lemma}

\begin{proof}
The previous adjunction is equivalent to
$$
\xymatrix{
\mathsf{Ho}(\mathcal{N}) \ar@<-2ex>[d]_{(-)_f} \\
*+<1pc>{\mathsf{Ho}(\mathcal{N})_S} \ar@{_{(}->}[u]\,,
}
$$
where $\mathsf{Ho}(\mathcal{N})_S$ denotes the full subcategory of
$\mathsf{Ho}(\mathcal{M})$ formed by the $S$-local objects of
$\mathcal{N}$ and $(-)_f$ denotes a fibrant resolution functor in
$\mathsf{L}_S\mathcal{N}$, see \cite{Hirschhorn}. Clearly, this
implies that the image of the set $\mathcal{G}$ under the functor $(-)_f$ is a set of weak generators in
$\mathsf{Ho}(\mathsf{L}_S\mathcal{N})$.

We now show that the $S$-local
objects in $\mathcal{N}$ are stable under filtered homotopy colimits.
Let ${\{ X_i \}}_{i \in I}$ be a filtered diagram of $S$-local
objects.
By proposition~\ref{prop}, we have an
isomorphism
$$ \underset{i \in I}{\mbox{hocolim}}\,X_i \stackrel{\sim}{\longrightarrow}
\underset{i \in I}{\mbox{colim}}\,X_i$$
in $\mathsf{Ho}(\mathcal{N})$.
We now show that $\underset{i \in I}{\mbox{colim}}\,X_i$ is an $S$-local
object.
Let $g:A \rightarrow B$ be an element of $S$. We have at our disposal the
following commutative diagram
$$
\xymatrix{
\mathsf{Map}(B,\underset{i \in I}{\mbox{colim}}\,X_i) \ar[rr]^{g^*} & & 
\mathsf{Map}(A,\underset{i \in I}{\mbox{colim}}\,X_i) \\
\underset{i \in I}{\mbox{colim}}\,\mathsf{Map}(B,X_i) \ar[u]^{\sim}   \ar[rr]_{\underset{i
    \in I}{\mbox{\mbox{colim}}} \, g_i^*} & & \underset{i \in
  I}{\mbox{colim}}\,\mathsf{Map}(A,X_i) \ar[u]_{\sim}\,.
}
$$
Now observe that since $A$ and $B$ are homotopically finitely
presented objects, the vertical arrows in the diagram are isomorphisms
in $\mathsf{Ho}(Sset)$. Since each object $X_i$ is $S$-local, the
morphism $g_i^*$ is an isomorphism in $\mathsf{Ho}(Sset)$ and so is $\underset{i
    \in I}{\mbox{colim}} \, g_i^*$. This implies that $\underset{i \in
    I}{\mbox{colim}}\,X_i$ is an $S$-local object. This shows that the
  inclusion
$$ \mathsf{Ho}(\mathcal{N})_S \hookrightarrow
\mathsf{Ho}(\mathcal{N})\,,$$
commutes with filtered homotopy colimits and so the image of the
set $\mathcal{G}$ under the functor $(-)_f$ consists of small objects in
$\mathsf{Ho}(\mathsf{L}_S\mathcal{N})$.

This proves the lemma.
\end{proof}

Recall from the previous chapter that we have constructed a pointed
derivator $\mathsf{L}_{\Sigma,P}\mathsf{Hot}_{\mathcal{M}_f}$. We will
now construct a strictly pointed Quillen model category whose
associated derivator is equivalent to $\mathsf{L}_{\Sigma,P}\mathsf{Hot}_{\mathcal{M}_f}$.
Consider the pointed Quillen model category
$$ \ast \downarrow \mathsf{Fun}(\mathcal{M}^{op}_f,Sset) =
\mathsf{Fun}(\mathcal{M}_f^{op}, Sset_{\bullet})\,.$$
We have the following Quillen adjunction
$$
\xymatrix{
\mathsf{Fun}(\mathcal{M}^{op}_f,Sset_{\bullet}) \ar@<1ex>[d]^U \\
\mathsf{Fun}(\mathcal{M}_f^{op},Sset) \ar@<1ex>[u]^{(-)_{+}}\,,
}
$$
where $U$ denotes the forgetful functor.

We denote by 
$\mathsf{L}_{\Sigma,P}\mathsf{Fun}(\mathcal{M}_f^{op},Sset_{\bullet})$
the left Bousfield localization of
$\mathsf{Fun}(\mathcal{M}_f^{op},Sset_{\bullet})$ by the image of the set
$\Sigma \cup \{ P\}$ under the functor $(-)_{+}$.
We denote by ${\mathsf{L}_{\Sigma,P}\mathsf{Hot}_{\mathcal{M}_f}}_{\bullet}$ the
derivator associated with $\mathsf{L}_{\Sigma,P}\mathsf{Fun}(\mathcal{M}_f^{op},Sset_{\bullet})$.

\begin{remark}\label{remar}
Since the derivators associated with
$\mathsf{L}_{\Sigma,P}\mathsf{Fun}(\mathcal{M}_f^{op},Sset)$ and
$\mathsf{L}_{\Sigma,P}\mathsf{Fun}(\mathcal{M}_f^{op},Sset_{\bullet})$
are caracterized by the same universal property we have a canonical equivalence of pointed derivators
$$ \mathsf{L}_{\Sigma,P}\mathsf{Hot}_{\mathcal{M}_f}
\stackrel{\sim}{\longrightarrow}
{\mathsf{L}_{\Sigma,P}\mathsf{Hot}_{\mathcal{M}_f}}_{\bullet}$$
Notice also that the category
$\mathsf{Fun}(\mathcal{M}_f^{op},Sset_{\bullet})$ endowed with the
projective model structure is pointed, left proper, compactly
generated and that the domains and codomains of the elements of the
set $(\Sigma \cup \{P\})_+$ are homotopically finitely presented
objects.
Therefore by lemma~\ref{gener}, the set 
$$ \mathcal{G}= \{
\mathbf{F}^X_{\Delta[n]_+ / \partial \Delta[n]_+ } |
\, X \in
\mathcal{M}_f\,, n\geq 0 \} \,,$$
of cofibers of the generating cofibrations in
$\mathsf{Fun}(\mathcal{M}_f^{op},Sset_{\bullet})$ is a set of small weak generators in $\mathsf{Ho}(\mathsf{L}_{\Sigma,P}\mathsf{Fun}(\mathcal{M}_f^{op},Sset_{\bullet}))$.
\end{remark}

\section{Stabilization}\label{spectra}

Let $\mathbb{D}$ be a regular pointed strong derivator.

In \cite{Heller}, Heller constructs a universal morphism to a triangulated
strong derivator
$$ \mathbb{D} \stackrel{stab}{\longrightarrow}
\mathsf{St}(\mathbb{D})\,,$$
which commutes with homotopy colimits.

This construction is done in two steps.
First consider the following ordered set
$$ \mathbf{V} := \{ (i,j) |\, |i -j|\leq 1\} \subset \mathbb{Z} \times \mathbb{Z}$$
naturally as a small category. We denote by
$$ \overset{\cdot}{\mathbf{V}}:= \{ (i,j) |\, |i-j|=1\} \subset
\mathbb{V}\,,$$
the full subcategory of `points on the boundary'.

Now, let $\mathsf{Spec}(\mathbb{D})$ be the full subderivator of
$\mathbb{D}_{\mathbf{V}}$, see definition $3.4$ in \cite{Cis-Nee}, formed by the objects $X$ in
$\mathbb{D}_{\mathbf{V}}(L)$, whose image under the functor
$$ \mathbb{D}_{\mathbf{V}}(L)=\mathbb{D}(\mathbf{V} \times L)
\longrightarrow \mathsf{Fun}(\mathbf{V}^{op},\mathbb{D}(L))$$
is of the form
$$
\xymatrix{
 & \ast \ar[r] & X_{(1,1)}\, \cdots \\
\ast \ar[r] & X_{(0,0)} \ar[r] \ar[u] & \ast \ar[u] \\
\cdots\,X_{(-1,-1)} \ar[u] \ar[r] & \ast \ar[u] & \,,
}
$$ 
see section $8$ in \cite{Heller}.
We have an evaluation functor $ev_{(0,0)}:\mathsf{Spec}(\mathbb{D})
\rightarrow \mathbb{D}$, which admits a left adjoint $L[0,0]$.

Finally, let $\mathsf{St}(\mathbb{D})$ be the full reflexive subderivator of
$\mathsf{Spec}(\mathbb{D})$ consisting of the $\Omega$-spectra, as
defined in \cite{Heller}.

We have the following adjunctions
$$
\xymatrix{
\mathbb{D} \ar@<-1ex>[d]_{L[0,0]} \ar@/_4pc/[dd]_{stab} \\
\mathsf{Spec}(\mathbb{D}) \ar@<-1ex>[u]_{ev_{(0,0)}}
\ar@<-1ex>[d]_{loc} \\
*+<1pc>{\mathsf{St}(\mathbb{D})} \ar@{_{(}->}[u] \,,
}
$$
in the $2$-category of derivators.

Let $\mathbb{T}$ be a triangulated strong derivator.
The following theorem is proved in \cite{Heller}.

\begin{theorem}\label{HellerT}
The morphism $stab$ induces an equivalence of categories
$$ \underline{\mathsf{Hom}}_!(\mathsf{St}(\mathbb{D}),\mathbb{T})
\stackrel{stab^*}{\longrightarrow}
\underline{\mathsf{Hom}}_!(\mathbb{D},\mathbb{T})\,.$$
\end{theorem}

\begin{lemma}\label{gener1}
Let $\mathcal{G}$ be a set of objects in $\mathbb{D}(e)$
which satisfies the following conditions:
\begin{itemize}
\item[A1)] If, for each $g$ in $\mathcal{G}$, we have
$$ \mathsf{Hom}_{\mathbb{D}(e)}(g,X)=\{*\}\,,$$
then $X$ is isomorphic to $*$, where $*$ denotes
 the terminal and initial object in $\mathbb{D}(e)$.

\item[A2)] For every set $K$ and each $g$ in $\mathcal{G}$  the
  canonical map
$$\underset{\stackrel{S \subseteq K}{S finite}}{\mbox{colim}}\,
\mathsf{Hom}_{\mathbb{D}(e)}(g, \underset{\alpha \in S}{\coprod}
X_{\alpha}) \stackrel{\sim}{\rightarrow} \mathsf{Hom}_{\mathbb{D}(e)}(g, \underset{\alpha \in S}{\coprod}
X_{\alpha})$$
is bijective.
\end{itemize}
Then the set 
$$ {\{ \Sigma^n stab(g) \,|\, g \in \mathcal{G}, \, n \in \mathbb{Z} \}}$$
of objects in $\mathsf{St}(\mathbb{D})(e)$, where $\Sigma$ denotes the
suspension functor in $\mathsf{St}(\mathbb{D})(e)$, satisfies conditions $A1)$
and $A2)$.
\end{lemma}

\begin{proof}
Let $\underline{X}$ be an object of $\mathsf{St}(\mathbb{D})(e)$.
Suppose that for each $g$ in $\mathcal{G}$ and $n$ in $\mathbb{Z}$, we have 
$$ \mathsf{Hom}_{\mathsf{St}(\mathbb{D})(e)}(\Sigma^n stab(g),
\underline{X}) = \{\ast\}\,.$$
Then by the following isomorphisms
$$
\begin{array}{rcl}
 \mathsf{Hom}_{\mathsf{St}(\mathbb{D})(e)}(\Sigma^n stab(g),
\underline{X}) & \simeq  & \mathsf{Hom}_{\mathsf{St}(\mathbb{D})(e)}(stab(g),
\Omega^n \underline{X})\\
& \simeq & \mathsf{Hom}_{\mathbb{D}(e)}(g,
ev_{(0,0)} \Omega^n \underline{X})\\
& \simeq & \mathsf{Hom}_{\mathbb{D}(e)}(g,
ev_{(n,n)}\underline{X})\,,
\end{array}
$$
we conclude that, for all $n$ in $\mathbb{Z}$, we have 
$$ ev_{(n,n)} \underline{X} = \ast\,.$$
By the conservativity axiom, $\underline{X}$ is isomorphic to $\ast$ in $\mathsf{St}(\mathbb{D})(e)$. This shows condition $A1)$. Now
observe that condition $A2)$ follows from the following isomorphisms
$$
\begin{array}{rcl}
\mathsf{Hom}_{\mathsf{St}(\mathbb{D})(e)}(\Sigma^n stab(g),
\underset{\alpha \in K}{\bigoplus} \underline{X_{\alpha}})&
\simeq & \mathsf{Hom}_{\mathsf{St}(\mathbb{D})(e)}(stab(g),
 \Omega^n \underset{\alpha \in K}{\bigoplus} \underline{X_{\alpha}})\\

& \simeq & \mathsf{Hom}_{\mathsf{St}(\mathbb{D})(e)}(stab(g), \underset{\alpha
  \in K}{\bigoplus} \Omega^n \underline{X_{\alpha}})\\

& \simeq & \mathsf{Hom}_{\mathbb{D}(e)}(g, ev_{(0,0)} \underset{\alpha
  \in K}{\coprod} \Omega^n \underline{X_{\alpha}})\\

&\simeq &\mathsf{Hom}_{\mathbb{D}(e)}(g, \underset{\alpha
  \in K}{\coprod} ev_{(0,0)} \Omega^n \underline{X_{\alpha}})\\

&\simeq &\underset{\stackrel{S \subseteq K}{S finite}}{\mbox{colim}} \, \mathsf{Hom}_{\mathbb{D}(e)}(g, \underset{\alpha
  \in S}{\coprod} ev_{(0,0)} \Omega^n \underline{X_{\alpha}})\\

&\simeq &\underset{\stackrel{S \subseteq K}{S finite}}{\mbox{colim}}
\,\mathsf{Hom}_{\mathbb{D}(e)}(g, ev_{(0,0)} \underset{\alpha
  \in S}{\coprod} \Omega^n \underline{X_{\alpha}})\\

&\simeq &\underset{\stackrel{S \subseteq K}{S finite}}{\mbox{colim}}
\,\mathsf{Hom}_{\mathsf{St}(\mathbb{D})(e)}(stab(g),\underset{\alpha
  \in S}{\bigoplus} \Omega^n \underline{X_{\alpha}})\\

&\simeq &\underset{\stackrel{S \subseteq K}{S finite}}{\mbox{colim}}
\,\mathsf{Hom}_{\mathsf{St}(\mathbb{D})(e)}(\Sigma^n stab(g),\underset{\alpha
  \in S}{\bigoplus} \underline{X_{\alpha}})

\end{array}
$$

\end{proof}

\begin{lemma}\label{gen1}
Let $\mathbb{T}$ be a triangulated derivator and $\mathcal{G}$ a set
of objects in $\mathbb{T}(e)$ which satisfies conditions $A1)$ and $A2)$ of lemma~\ref{gener1}.

Then for every small category $L$ and every point $x: e \rightarrow L$
in $L$, the set 
$$ \{x_!(g)\,|\,g \in \mathcal{G}, \, x:e\rightarrow L\}$$ satisfies conditions
$A1)$ and $A2)$ in the category $\mathbb{T}(L)$.
\end{lemma}

\begin{proof}
Suppose that 
$$ \mathsf{Hom}_{\mathbb{T}(L)}(x_!(g),M) = \{\ast \}\,,$$
for every $g \in \mathcal{G}$ and every point $x$ in $L$. Then by adjunction
$x^{\ast}M$ is isomorphic to $\ast$ in $\mathbb{T}(e)$ and so by the
conservativity axiom, $M$ is isomorphic to $\ast$ in
$\mathbb{T}(L)$. This shows condition $A1)$. Condition $A2)$ follows
from the following isomorphisms
$$
\begin{array}{rcl}
\mathsf{Hom}_{\mathbb{T}(L)}(x_!(g),\underset{\alpha \in
  K}{\bigoplus} M_{\alpha})&
\simeq &\mathsf{Hom}_{\mathbb{T}(e)}(g,x^{\ast} \underset{\alpha \in
  K}{\bigoplus} M_{\alpha})\\

&\simeq &\mathsf{Hom}_{\mathbb{T}(e)}(g, \underset{\alpha \in
  K}{\bigoplus} x^{\ast} M_{\alpha})\\

&\simeq & \underset{\alpha \in K}{\bigoplus}   \mathsf{Hom}_{\mathbb{T}(e)}(g, x^{\ast} M_{\alpha})\\

&\simeq & \underset{\alpha \in K}{\bigoplus}
\mathsf{Hom}_{\mathbb{T}(L)}(x_!(g), M_{\alpha})\,.

\end{array}
$$
\end{proof}

\begin{remark}\label{important}
Notice that if $\mathbb{D}$ is a regular pointed strong derivator and we
have at our disposal of a set $\mathcal{G}$ of objects in $\mathbb{D}(e)$
which satisfies conditions $A1)$ and $A2)$, then lemma~\ref{gener} and lemma~\ref{gen1} imply
that $\mathsf{St}(\mathbb{D})(L)$ is a compactly generated triangulated
category, for every small category $L$.
\end{remark}

\subsection*{Relation with Hovey/Schwede's stabilization}

We will now relate Heller's construction with the construction of
spectra as it is done by Hovey in \cite{Spectra} and Schwede in \cite{Schwede}.

Let $\mathcal{M}$ be a pointed, simplicial, left proper, cellular,
almost finitely generated Quillen model category, see definition $4.1$ in
\cite{Spectra}, where sequential colimits commute with finite products
and homotopy pullbacks. This implies in particular that the associated derivator
$\mathsf{HO}(\mathcal{M})$ will be regular.

\begin{example}\label{exem}
Consider the category
$\mathsf{L}_{\Sigma,P}\mathsf{Fun}(\mathcal{M}_f^{op},Sset_{\bullet})$
defined in section~\ref{small}.
Notice that the category of pointed simplicial pre-sheaves
$\mathsf{Fun}(\mathcal{M}_f^{op},Sset_{\bullet})$ is pointed, simplicial,
left proper, cellular and even finitely generated, see definition
$4.1$ in \cite{Spectra}. Since limits and colimits in
$\mathsf{Fun}(\mathcal{M}_f^{op},Sset_{\bullet})$ are calculated
objectwise, we conclude that sequential colimits commute with finite
products. 
Now, by theorem $4.1.1$ in \cite{Hirschhorn} the category
$\mathsf{L}_{\Sigma,P}\mathsf{Fun}(\mathcal{M}_f^{op},Sset_{\bullet})$ is
also pointed, simplicial, left proper and cellular.

Now observe that the domains and codomains of each morphism in
$\Lambda((\Sigma \cup \{P\})_{+})$, see definition $4.2.1$ in
\cite{Hirschhorn}, are finitely presented, since the forgetful functor 
$$ \mathsf{Fun}(\mathcal{M}_f^{op}, Sset_{\bullet}) \rightarrow \mathsf{Fun}(\mathcal{M}^{op},Sset)$$
commutes with filtered colimits and homotopy pullbacks. Now, by proposition $4.2.4$
in \cite{Hirschhorn}, we conclude that a morphism $A
\stackrel{f}{\rightarrow} B$ in
$\mathsf{L}_{\Sigma,P}\mathsf{Fun}(\mathcal{M}_f^{op},Sset_{\bullet})$,
with $B$ a local object, is a local fibration if and
only if it has the right lifting property with respect to the set 
$$ J \cup \Lambda((\Sigma \cup \{P\})_{+})\,,$$
where $J$ denotes the set of generating acyclic cofibrations in
$\mathsf{Fun}(\mathcal{M}_f^{op},Sset_{\bullet})$. This shows that
$\mathsf{L}_{\Sigma,P}\mathsf{Fun}(\mathcal{M}_f^{op},Sset_{\bullet})$ is
almost finitely generated.

\end{example}

Recall from section $1.2$ in \cite{Schwede} that since $\mathcal{M}$
is a pointed, simplicial model category, we have a Quillen adjunction
$$
\xymatrix{
\mathcal{M} \ar@<-1ex>[d]_{\Sigma(-)} \\
\mathcal{M} \ar@<-1ex>[u]_{\Omega(-)}\,,
}
$$
where $\Sigma(X)$ denotes the suspension of an object $X$ , i.e. the
pushout of the diagram
$$
\xymatrix{
X \otimes \partial \Delta^1 \ar[r] \ar[d] & X \otimes \Delta^1 \\
\ast & .
}
$$
Recall also that in \cite{Spectra} and \cite{Schwede} the authors
construct a stable Quillen model category
$\mathsf{Sp}^{\mathbb{N}}(\mathcal{M})$ of spectra associated with
$\mathcal{M}$ and with the left Quillen functor $\Sigma(-)$.
We have the following Quillen adjunction, see~\cite{Spectra},
$$
\xymatrix{
\mathcal{M} \ar@<-1ex>[d]_{\Sigma^{\infty}} \\
\mathsf{Sp}^{\mathbb{N}}(\mathcal{M}) \ar@<-1ex>[u]_{ev_0}
}
$$
and thus a morphism to a strong triangulated derivator
$$ \mathsf{HO}(\mathcal{M})
\stackrel{\mathbb{L}\Sigma^{\infty}}{\longrightarrow}
  \mathsf{HO}(\mathsf{Sp}^{\mathbb{N}}(\mathcal{M}))$$
which commutes with homotopy colimits.

By theorem~\ref{HellerT}, we have at our disposal a diagram
$$
\xymatrix{
\mathsf{HO}(\mathcal{M}) \ar[d]_{stab}
\ar[dr]^{\mathbb{L}\Sigma^{\infty}} & \\
\mathsf{St}(\mathsf{HO}(\mathcal{M})) \ar[r]_{\varphi} & \mathsf{HO}(\mathsf{Sp}^{\mathbb{N}}(\mathcal{M}))\,, 
}
$$
which is commutative up to isomorphism in the $2$-category of derivators.

Now suppose also that we have a set $\mathcal{G}$ of small weak generators in $\mathsf{Ho}(\mathcal{M})$, as in
definitions $7.2.1$ and $7.2.2$ in \cite{Spectra}.
Suppose also that each object of $\mathcal{G}$ considered in $\mathcal{M}$ is cofibrant,
finitely presented, homotopy finitely presented and has a finitely
presented cylinder object.

\begin{example}\label{ex2}
Observe that the category
$\mathsf{Fun}(\mathcal{M}_f^{op},Sset_{\bullet})$ is pointed and finitely
generated. By corollary $7.4.4$ in \cite{Hovey}, the set 
$$ \mathcal{G}= \{
\mathbf{F}^X_{\Delta[n]_{+} / \partial \Delta[n]_{+} } |
\, X \in
\mathcal{M}_f\,, n\geq 0 \} \,,$$
is a set of small weak generators in
$\mathsf{Ho}(\mathsf{Fun}(\mathcal{M}^{op}_f,Sset_{\bullet}))$. Since the
domains and codomains of the set 
$$ (\Sigma \cup \{ P\} )_{+}$$
are homotopically finitely presented objects, lemma~\ref{gener}
implies that $\mathcal{G}$ is a set of small weak generators
in $\mathsf{Ho}(\mathsf{L}_{\Sigma,P}
\mathsf{Fun}(\mathcal{M}_f^{op},Sset_{\bullet}))$. Clearly the elements
of $\mathcal{G}$ are cofibrant, finitely presented and have a finitely
presented cylinder object. They are also homotopically finitely presented.
\end{example}

Under the hypotheses above on the category $\mathcal{M}$, we have the following comparison theorem

\begin{theorem}\label{repre}
The induced morphism of triangulated derivators
$$ \varphi: \mathsf{St}(\mathsf{HO}(\mathcal{M})) \longrightarrow
\mathsf{HO}(\mathsf{Sp}^{\mathbb{N}}(\mathcal{M}))$$
is an equivalence.
\end{theorem}

The proof of theorem~\ref{repre} will consist in verifying the
conditions of the following general proposition.

\begin{proposition}\label{aux}
Let $F: \mathbb{T}_1 \rightarrow \mathbb{T}_2$ be a morphism of strong
triangulated derivators. Suppose that the triangulated categories
$\mathbb{T}_1(e)$ and $\mathbb{T}_2(e)$ are compactly generated and
that there is a set $\mathcal{G} \subset \mathbb{T}_1(e)$ of compact generators,
which is stable under suspensions and satisfies the following conditions~: 
\begin{itemize}
\item[a)] $F(e)$ induces bijections
$$
  \mathsf{Hom}_{\mathbb{T}_1(e)}(g_1,g_2) \rightarrow
  \mathsf{Hom}_{\mathbb{T}_2(e)}(Fg_1, Fg_2), \, \forall g_1, g_2 \in
  \mathcal{G}$$ 
and
\item[b)] the set of objects $\{ Fg \, | \, g \in \mathcal{G} \}$ is a
  set of compact generators in $\mathbb{T}_2(e)$.
\end{itemize}
Then the morphism $F$ is an equivalence of derivators.
\end{proposition}

\begin{proof}
Conditions $a)$ and $b)$ imply that $F(e)$ is an equivalence of
triangulated categories, see~\cite{Neeman}.

Now, let $L$ be a small category. We show that conditions $a)$ and
$b)$ are also verified by $F(L)$, $\mathbb{T}_1(L)$ and $\mathbb{T}_2(L)$. By lemma~\ref{gen1} the sets
$$
\begin{array}{rcl}
\{ x_!(g) |\, g \in \mathcal{G}, \, x:e \rightarrow L\} & \mbox{and} & \{ x_!(Fg) |\, g \in \mathcal{G}, \, x:e \rightarrow L\}
\end{array}
$$
consist of compact generators for $\mathbb{T}_1(L)$, resp. $\mathbb{T}_2(L)$, which are stable under suspensions. Since $F$
commutes with homotopy colimits $F(x_!(g)) = x_!(Fg)$ and so the following isomorphisms
$$
\begin{array}{rcl}
\mathsf{Hom}_{\mathbb{T}_1(L)}(x_!(g_1), x_!(g_2)) & \simeq &
\mathsf{Hom}_{\mathbb{T}_1(e)}(g_1, x^{\ast} x_!(g_2)) \\
 & \simeq & \mathsf{Hom}_{\mathbb{T}_2(e)}(F(g_1), F(x^{\ast} x_!(g_2)))
 \\
 & \simeq & \mathsf{Hom}_{\mathbb{T}_2(e)}(Fg_1, x^{\ast} F(x_!(g_2)))
 \\
 & \simeq &  \mathsf{Hom}_{\mathbb{T}_2(L)}(x_! F(g_1), x_!F(g_2))
\end{array}
$$
imply the proposition.
\end{proof}

Let us now prove theorem~\ref{repre}:
\begin{proof}

Let us first prove condition $b)$ of proposition~\ref{aux}. Since the
set $\mathcal{G}$ of small generators in $\mathsf{Ho}(\mathcal{M})$
satisfies the conditions of lemma~\ref{gener1}, we have a set 
$$ {\{ \Sigma^n stab(g) \,|\, g \in \mathcal{G}, \, n \in \mathbb{Z} \}}$$
of compact generators in $\mathsf{St}(\mathsf{HO}(\mathcal{M}))(e)$,
which is stable under suspensions. We now show that the set 
$$ {\{ \Sigma^n \mathbb{L} \Sigma^{\infty}(g) \,|\, g \in \mathcal{G}, \, n \in \mathbb{Z} \}}$$
is a set of compact generators in
$\mathsf{Ho}(\mathsf{Sp}^{\mathbb{N}}(\mathcal{M}))$. These objects
are compact because the functor $\mathbb{R}ev_0$ in the adjunction
$$
\xymatrix{
\mathsf{Ho}(\mathcal{M}) \ar@<-1ex>[d]_{\mathbb{L}\Sigma^{\infty}} \\
\mathsf{Ho}(\mathsf{Sp}^{\mathbb{N}}(\mathcal{M})) \ar@<-1ex>[u]_{\mathbb{R}ev_0}
}
$$
commutes with filtered homotopy colimits.
We now show that they form a set of generators. Let $Y$ be an object in
$\mathsf{Ho}(\mathsf{Sp}^{\mathbb{N}}(\mathcal{M}))$, that we can
suppose, without loss of generality, to be an $\Omega$-spectrum,
see \cite{Spectra}.
Suppose that 
$$ 
\mathsf{Hom}(\Sigma^n \mathbb{L}\Sigma^{\infty}(g_i),Y)
\simeq \underset{m}{\mbox{colim}}\, \mathsf{Hom}(g_i, \Omega^m Y_{m+p}) =
\{ \ast \},\,\, n \geq 0 \,.
$$ 
Since $Y$ is an $\Omega$-spectrum, we have
$$ Y_p = \ast,\,\,\,\forall p\geq 0\,.$$
This implies that $Y$ is isomorphic to $\ast$ in
$\mathsf{Ho}(\mathsf{Sp}^{\mathbb{N}}(\mathcal{M}))$.

We now show condition $a)$.
Let $g_1$ and $g_2$ be objects in $\mathcal{G}$. Observe that we have the following isomorphisms, see \cite{Heller}
$$
\begin{array}{l}
\mathsf{Hom}_{\mathsf{St}(\mathsf{HO}(\mathcal{M}))(e)}(stab(g_1),stab(g_2))
\\
\simeq \mathsf{Hom}_{\mathsf{Ho}(\mathcal{M})}(g_1, (ev_{(0,0)} \circ loc
\circ L[0,0])(g_2))\\
\simeq \mathsf{Hom}_{\mathsf{Ho}(\mathcal{M})}(g_1,
ev_{(0,0)}(\mbox{hocolim}\,( L[0,0](g_2) \rightarrow
\Omega\sigma L[0,0](g_2) \rightarrow \ldots )))\\
\simeq \mathsf{Hom}_{\mathsf{Ho}(\mathcal{M})}(g_1,
\mbox{hocolim}\, ev_{(0,0)} (L[0,0](g_2) \rightarrow
\Omega\sigma L[0,0](g_2) \rightarrow \cdots ))\\
\simeq \underset{j}{\mbox{colim}}\,
\mathsf{Hom}_{\mathsf{Ho}(\mathcal{M})}(g_1, \Omega^j \Sigma^j(g_2))\,.
\end{array}
$$
Now, by corollary $4.13$ in \cite{Spectra}, we have
$$
\begin{array}{ccc}
\mathsf{Hom}_{\mathsf{Ho}(\mathsf{Sp}^{\mathbb{N}}(\mathcal{M}))}(\mathbb{L}\Sigma^{\infty}(g_1),
\mathbb{L}\Sigma^{\infty}(g_2)) & \simeq & \mathsf{Hom}_{\mathsf{Ho}(\mathsf{Sp}^{\mathbb{N}}(\mathcal{M}))}(\Sigma^{\infty}(g_1),
(\Sigma^{\infty}(g_2))_f) \\
& \simeq & \underset{m}{\mbox{colim}}\,
\mathsf{Hom}_{\mathsf{Ho}(\mathcal{M})}(g_1, \Omega^m (\Sigma^m(g_2))_f)\,,
\end{array}
$$
where $(\Sigma^{\infty}(g_2))_f$ denotes a levelwise fibrant
resolution of $\Sigma^{\infty}(g_2)$ in the category $\mathsf{Sp}^{\mathbb{N}}(\mathcal{M})$.

Now, notice that since $g_2$ is cofibrant, so is $\Sigma^m(g_2)$ and
so we have the following isomorphism
$$ \Omega^m (\Sigma^m(g_2))_f \stackrel{\sim}{\longrightarrow}
(\mathbb{R}\Omega)^m \circ (\mathbb{L}\Sigma)^m(g_2)$$
in  $\mathsf{Ho}(\mathsf{SP}^{\mathbb{N}}(\mathcal{M}))$.
This implies that for $j\geq 0$, we have an isomorphism
$$ \Omega^j \Sigma^j(g_2) \stackrel{\sim}{\longrightarrow} \Omega^j
(\Sigma^j(g_2))_f$$
in $\mathsf{Ho}(\mathsf{Sp}^{\mathbb{N}}(\mathcal{M}))$ and so
$$ \mathsf{Hom}_{\mathsf{St}(\mathsf{HO}(\mathcal{M}))(e)}(stab(g_1), stab(g_2)) =
\mathsf{Hom}_{\mathsf{Ho}(\mathsf{Sp}^{\mathbb{N}}(\mathcal{M}))}(\mathbb{L}\Sigma^{\infty}(g_1),
\mathbb{L}\Sigma^{\infty}(g_2))\,.$$
Let now $p$ be an integer.
Notice that 
$$ \mathsf{Hom}_{\mathsf{St}(\mathsf{HO}(\mathcal{M}))(e)}(stab(g_1),
\Sigma^p stab(g_2)) = \underset{j}{\mbox{colim}}\,
\mathsf{Hom}_{\mathsf{Ho}(\mathcal{M})}(g_1, \Omega^j
\Sigma^{j+p}(g_2))$$
and that 
$$
\mathsf{Hom}_{\mathsf{Ho}(\mathsf{Sp}^{\mathbb{N}}(\mathcal{M}))}(\mathbb{L}
\Sigma^{\infty} (g_1),
\Sigma^p\mathbb{L}\Sigma^{\infty}(g_2))= \underset{m}{\mbox{colim}}\,
\mathsf{Hom}_{\mathsf{Ho}(\mathcal{M})}(g_1,
\Omega^m(\Sigma^{m+p}(g_2))_f)\,.$$

This proves condition $a)$ and so the theorem is proven.
\end{proof}

\begin{remark}\label{Tr}
If we consider for $\mathcal{M}$ the category
$\mathsf{L}_{\Sigma,P}\mathsf{Fun}(\mathcal{M}_f^{op}, Sset_{\bullet})$, we
have equivalences of derivators
$$ \varphi :
\mathsf{St}({\mathsf{L}_{\Sigma,P}\mathsf{Hot}_{\mathcal{M}_f}}_{\bullet})
\stackrel{\sim}{\longrightarrow}
\mathsf{HO}(\mathsf{Sp}^{\mathbb{N}}(\mathsf{L}_{\Sigma,P}\mathsf{Fun}(\mathcal{M}_f^{op},
Sset_{\bullet}))) \stackrel{\sim}{\leftarrow} \mathsf{St}(\mathsf{L}_{\Sigma,P}\mathsf{Hot}_{\mathcal{M}_f}) \,.$$
\end{remark}
Let $\mathbb{D}$ be a strong triangulated derivator.

Now, by theorem~\ref{HellerT} and proposition~\ref{ext}, we have the
following proposition

\begin{proposition}\label{Trin}
We have an equivalence of categories
$$
\underline{\mathsf{Hom}}_!(\mathsf{St}(\mathsf{L}_{\Sigma,P} \mathsf{Hot}_{\mathsf{dgcat}_f}),\mathbb{D})
\stackrel{(stab\circ \Phi \circ
  \mathbb{R}\underline{h})^{\ast}}{\longrightarrow} \underline{\mathsf{Hom}}_{flt,p}(\mathsf{HO}(\mathsf{dgcat}),\mathbb{D})\,.$$
\end{proposition}

Since the category $Sset_{\bullet}$ satisfies all the conditions of
theorem~\ref{repre}, we have the following characterization of the
classical category of spectra, after Bousfield-Friedlander~\cite{Bos-Fri}, by a
universal property.

\begin{proposition}
We have an equivalence of categories
$$\underline{\mathsf{Hom}}_!(\mathsf{HO}(\mathsf{Sp}^{\mathbb{N}}(Sset_{\bullet})),\mathbb{D})
\stackrel{\sim}{\longrightarrow} \mathbb{D}(e) \,.$$
\end{proposition}

\begin{proof}
By theorems~\ref{repre} and \ref{Cin}, we have the following
equivalences
$$
\begin{array}{rcl}
\underline{\mathsf{Hom}}_!(\mathsf{HO}(\mathsf{Sp}^{\mathbb{N}}(Sset_{\bullet})),\mathbb{D})
& 
\simeq & 
\underline{\mathsf{Hom}}_!(\mathsf{HO}(Sset_{\bullet}),\mathbb{D})\\
& = & \underline{\mathsf{Hom}}_!(\mathsf{Hot}_{\bullet},\mathbb{D}) \\
& \simeq & \underline{\mathsf{Hom}}_!(\mathsf{Hot},\mathbb{D}) \\
& \simeq & \mathbb{D}(e) \,.
\end{array}
$$
This proves the proposition.
\end{proof}

\begin{remark}
An analoguos caracterization of the category of spectra, but in the
context of stable $\infty$-categories is proved in \cite[17.6]{Lurie}.
\end{remark}

\section{DG quotients}\label{chapquotient}
Recall from \cite{addendum} \cite{IMRN} that we have at our disposal a Morita Quillen model structure on the category of small dg
categories $\mathsf{dgcat}$, see example~\ref{mori}. As shown in
\cite{addendum} \cite{IMRN} the homotopy category $\mathsf{Ho}(\mathsf{dgcat})$ is pointed.
In the following, we will be considering this Quillen model structure. We
denote by $I$ its set of generating cofibrations.

\begin{notation}
We denote by $\mathcal{E}$ the set of inclusions of full dg subcategories
$$ \mathcal{G} \hookrightarrow \mathcal{H}\,,$$
where $\mathcal{H}$ is a strict finite $I$-cell.
\end{notation}

Recall that we have a morphism of derivators
$$ \mathcal{U}_t:= stab \circ \Phi \circ \mathbb{R}\underline{h} : \mathsf{HO}(\mathsf{dgcat})
\longrightarrow
\mathsf{St}(\mathsf{L}_{\Sigma,P}\mathsf{Hot}_{\mathsf{dgcat}_f})$$
which commutes with filtered homotopy colimits and preserves the point.

Let us now make some general arguments.

Let $\mathbb{D}$ be a pointed derivator. We denote by $M$ the category
associated to the graph
$$ 0 \leftarrow 1 \,.$$
Consider the functor $t=1 : e \rightarrow M$. Since the functor
$t$ is an open immersion, see notation~\ref{not4} and the derivator $\mathbb{D}$ is pointed,
the functor
$$ t_! : \mathbb{D}(e) \rightarrow \mathbb{D}(M)$$
has a left adjoint
$$ t^? : \mathbb{D}(M) \rightarrow \mathbb{D}(e)\,,$$
see \cite{Cis-Nee}.
We denote it by 
$$ \mathsf{cone}: \mathbb{D}(M) \rightarrow \mathbb{D}(e)\,.$$
Let $F:\mathbb{D} \rightarrow \mathbb{D}'$ be a morphism of pointed
derivators. Notice that we have a natural transformation of functors.
$$S: \mathsf{cone} \circ F(M) \rightarrow F(e) \circ
\mathsf{cone}\,.$$

\begin{proposition}\label{cons}
Let $\mathcal{A} \stackrel{R}{\hookrightarrow} \mathcal{B}$ be an
inclusion of a full dg subcategory and $\ul_R$
$$
\xymatrix{
\mathcal{A} \ar@{^{(}->}[r]^R \ar[d] & \mathcal{B} \,,\\
0 & 
}
$$ 
the associated object in
$\mathsf{HO}(\mathsf{dgcat})(\ul)$, where $0$ denotes the
terminal object in $\mathsf{Ho}(\mathsf{dgcat})$. Then there exists a filtered
category $J$ and an object $D_R$ in
$\mathsf{HO}(\mathsf{dgcat})(\ul \times J)$, such that
$$ p_!(D_R) \stackrel{\sim}{\longrightarrow} \ul_R\,,$$
where $p:\ul \times J \rightarrow \ul$ denotes the
projection functor. Moreover, for every point $j:e \rightarrow J$ in
$J$ the object $(1 \times j)^{\ast}$ in
$\mathsf{HO}(\mathsf{dgcat})(\ul)$ is of the form
$$ 0 \leftarrow Y_j \stackrel{L_j}{\rightarrow} X_j \,,$$
where $Y_j \stackrel{L_j}{\rightarrow} X_j$, belongs to
the set $\mathcal{E}$.
\end{proposition}

\begin{proof}
Apply the small object argument to the morphism
$$ \emptyset \longrightarrow \mathcal{B}$$
using the set of generating cofibrations $I$ and obtain the
factorization
$$
\xymatrix{
*+<1pc>{\emptyset} \ar[rr] \ar@{>->}[dr]_i & & \mathcal{B} \\
 & Q(\mathcal{B}) \ar@{->>}[ur]^{\sim}_p & \,, 
}
$$
where $i$ is an $I$-cell.
Now consider the following fiber product
$$
\xymatrix{
p^{-1}(\mathcal{A}) \ar@{->>}[d]_{\sim} \ar@{^{(}->}[r]
\ar@{}[dr]|{\ulcorner} &
  Q(\mathcal{B}) \ar@{->>}[d]^p_{\sim} \\
\mathcal{A} \ar@{^{(}->}[r]^J & \mathcal{B}\,.
}
$$
Notice that $p^{-1}(\mathcal{A})$ is a full dg subcategory of
  $Q(\mathcal{B})$. 

Now, by proposition~\ref{prop}, we have an
isomorphism
$$ \underset{j \in J}{\mbox{colim}} \, X_j \stackrel{\sim}{\longrightarrow}
Q(\mathcal{B})\,,$$
where $J$ is the filtered category of inclusions of strict finite
sub-$I$-cells $X_j$ into $Q(\mathcal{B})$.

For each $j \in J$, consider the fiber product
$$
\xymatrix{
Y_j \ar[d] \ar@{^{(}->}[r] \ar@{}[dr]|{\ulcorner} & X_j \ar[d] \\
p^{-1}(\mathcal{A}) \ar[r] \ar@{^{(}->}[r] & Q(\mathcal{B})\,.
}
$$
In this way, we obtain a morphism of diagrams
$$ \{ Y_j \}_{j \in J} \hookrightarrow \{ X_j \}_{j \in
  J} \,,$$
such that for each $j$ in $J$, the inclusion
$$ Y_j \hookrightarrow X_j $$
belongs to the set $\mathcal{E}$ and $J$ is filtered.

Consider now the diagram $D_I$
$$ {\{0 \leftarrow Y_j \hookrightarrow X_j \}}_{j \in
  J}$$
in the category $\mathsf{Fun}(\ul \times J, \mathsf{dgcat})$.
Now, notice that we have the isomorphism
$$ \underset{j \in J}{\mbox{colim}} \, \{0 \leftarrow
Y_j \hookrightarrow X_j \} \stackrel{\sim}{\longrightarrow}
\{ 0
 \leftarrow p^{-1}(\mathcal{A}) \hookrightarrow Q(\mathcal{B}) \}$$
in $\mathsf{Fun}(\ul, \mathsf{dgcat})$ and the weak equivalence
$$ 
\xymatrix{
\{ 0 \ar@{=}[d] & p^{-1}(\mathcal{A}) \ar[l] \ar[d]^{\sim}
\ar@{^{(}->}[r] & Q(\mathcal{B})\} \ar[d]^{\sim} \\
\{ 0 & \mathcal{A} \ar[l] \ar@{^{(}->}[r] & \mathcal{B} \}
}
$$
in $\mathsf{Fun}(\ul, \mathsf{dgcat})$, when endowed with the
projective model structure, see \cite{Hirschhorn}. Since
$\mathsf{Fun}(\ul, \mathsf{dgcat})$ is clearly also compactly
generated, we have an isomorphism
$$ \underset{j \in J}{\mbox{hocolim}} \, ( 0 \leftarrow
Y_j \rightarrow X_j) \stackrel{\sim}{\longrightarrow} \underset{j \in J}{\mbox{colim}} \, ( 0 \leftarrow Y_j \rightarrow X_j)\,.$$

Finally, notice that $D_R$ is an object of
$\mathsf{HO}(\mathsf{dgcat})(\ul \times J)$ and that $p_!(D_R)$,
where $p:\ul \times J \rightarrow J$ denotes the projection
functor, identifies with
$$ \underset{i \in J}{\mbox{hocolim}}\, ( 0 \leftarrow Y_j \rightarrow X_j )\,.$$
This proves the proposition.
\end{proof}

\begin{notation}
We denote by $\mathcal{E}_{st}$ the set of morhisms $S_L$, where $L$
belongs to the set $\mathcal{E}$.
\end{notation}

Let $\mathbb{D}$ be a strong triangulated derivator.
\begin{theorem}\label{invert}
If $$G:
\mathsf{St}(\mathsf{L}_{\Sigma,P}\mathsf{Hot}_{\mathsf{dgcat}_f})
\rightarrow \mathbb{D}$$ is a morphism of triangulated derivators
commuting with arbitrary homotopy colimits and such that $G(e)(S_L)$
is invertible for each $L$ in $\mathcal{E}$, then $G(e)(S_K)$ is
invertible for each inclusion $K:\mathcal{A} \hookrightarrow
\mathcal{B}$ of a full dg subcategory.
\end{theorem}

\begin{proof}
Let $\mathcal{A} \stackrel{K}{\hookrightarrow} \mathcal{B}$ be an
inclusion of a full dg subcategory. Consider the morphism
$$ \varphi_K := \varphi(\ul_K):(i_! \circ
\mathcal{U}_T)(\ul_K) \longrightarrow (\mathcal{U}_T \circ
i_!)(\ul_K)$$
in
$\mathsf{St}(\mathsf{L}_{\Sigma,P}\mathsf{Hot}_{\mathsf{dgcat}_f})(\square)$.

Let $D_K$ be the object of
$\mathsf{HO}(\mathsf{dgcat})(\ul \times J)$ constructed in
proposition~\ref{cons}. In particular $p'_!(D_K)
\stackrel{\sim}{\rightarrow} \ul_K$, where $p':\ul \times
J \rightarrow \ul$ denotes the projection functor.

The inclusion $i : \ul \hookrightarrow \square$, induces a
commutative square
$$
\xymatrix{
\mathsf{HO}(\mathsf{dgcat})(\square \times J) \ar[d]^{(i\times
  1)^{\ast}} \ar[rr]^{\mathcal{U}_T(\square \times J)} & & 
\mathsf{St}(\mathsf{L}_{\Sigma,P}\mathsf{Hot}_{\mathsf{dgcat}_f})(\square
\times J) \ar[d]^{(i\times 1)^{\ast}} \\
\mathsf{HO}(\mathsf{dgcat})(\ul \times J)
\ar[rr]_{\mathcal{U}_T(\square \times J)} & & \mathsf{St}(\mathsf{L}_{\Sigma,P}\mathsf{Hot}_{\mathsf{dgcat}_f})(\ul
\times J)
}
$$
and a morphism
$$ \Psi : ((i\times 1)_! \circ \mathcal{U}_T(\ul \times J))(D_K)
\longrightarrow (\mathcal{U}_T(\square \times J) \circ (i \times
1)_!)(D_K)\,.$$
We will now show that
$$ p_! \Psi \stackrel{\sim}{\longrightarrow} \varphi_K \,,$$
where $p:\square \times J \rightarrow \square$, denotes the projection
functor.

The fact that we have the following commutative square
$$ 
\xymatrix{
\square & \square \times J \ar[l]_p \\
\ul \ar[u]^i & \ul \times J \ar[u]_{i\times 1} \ar[l]^{p'}
}
$$
and that the morphism of derivators $\mathcal{U}_T$ commutes with
filtered homotopy colimits implies the following equivalences
$$
\begin{array}{lcl}
p_!\Psi & & \\
= p_! \circ (i\times 1)_! \circ \mathcal{U}_T(\ul \times j)(D_K)
& \rightarrow & p_! \circ \mathcal{U}_T(\square\times J)\circ (i\times
1)_! (D_k) \\
\simeq i_! \circ p'_! \circ \mathcal{U}_T(\ul \times J)(D_k) &
\rightarrow & \mathcal{U}_T(\square \times J) \circ p_! \circ (i\times
1)_! (D_K)\\
\simeq i_! \circ \mathcal{U}_T(\ul) \circ p'_!(D_K) &
\rightarrow & \mathcal{U}_T(\square) \circ i_! \circ p'_!(D_K) \\
\simeq (i_! \circ \mathcal{U}_T(\ul))(\ul_K) & \rightarrow
& (\mathcal{U}_T(\square) \circ i_!)(\ul_K)\\
= \varphi_J
\end{array}
$$
This shows that
$$ p_!(\Psi) \stackrel{\sim}{\longrightarrow} \varphi_K\,.$$
We now show that $\Psi$ is an isomorphism. For this, by conservativity, it
is enough to show that for every object $j:e \rightarrow J$ in $J$,
the morphism
$$ (1\times j)^{\ast}(\Psi)\,,$$
is an isomorphism in
$\mathsf{St}(\mathsf{L}_{\Sigma,P}\mathsf{Hot}_{\mathsf{dgcat}_f})(\square)$.
Recall from proposition~\ref{cons} that $(1\times j)^{\ast} (D_K)$
identifies with
$$ \{ \, 0 \leftarrow Y_j
\stackrel{L_j}{\hookrightarrow} X_j \}\,,$$
where $L_j$ belongs to $\mathcal{E}$. We now show that $(1\times
j)^{\ast}(\Psi)$ identifies with $\varphi_{L_j}$, which by hypotheses,
is an isomorphism.

Now, the following commutative diagram
$$
\xymatrix{
\square \ar[r]^{1\times j} & \square \times J \\
\ul \ar[u]^i \ar[r]_{1\times j} & \ul \times J
\ar[u]_{i\times i} 
}
$$
and the dual of proposition $2.8$ in \cite{Cisinski} imply that we
have the following equivalences
$$
\begin{array}{lcl}
(1\times j)^{\ast} \Psi & & \\
= ((1\times j)^{\ast} \circ (i \times 1)_! \circ
\mathcal{U}_T(\ul \times J))(D_K) & \rightarrow &
((1\times j)^{\ast} \circ \mathcal{U}_T(\square \times J) \circ (i
\times 1)_!)(D_K)\\
\simeq (i_! \circ (1\times j)^{\ast} \circ \mathcal{U}_T(\ul
\times J))(D_K) & \rightarrow & (\mathcal{U}_T(\square \times J) \circ
(1\times j)^{\ast} \circ (i \times 1)_!)(D_K)\\
\simeq (i_! \circ \mathcal{U}_T(\ul) \circ (1\times
j)^{\ast})(D_K) & \rightarrow & (\mathcal{U}_T(\square) \circ i_!
\circ (1\times j)^{\ast})(D_K)\\
\simeq i_! \circ \mathcal{U}_T(\ul)(\ul_{L_j}) &
\rightarrow & \mathcal{U}_T(\square) \circ i_!(\ul_{L_j})\\
= \varphi_{L_j} & &
\end{array}
$$
Since by hypotheses $\varphi_{L_j}$ is an isomorphism and the morphism
$G$ commutes with homotopy colimits the theorem is proven.
\end{proof}

\section{The universal localizing invariant}\label{labuniv}

Recall from theorem~\ref{repre} and remark~\ref{Tr} that if we
consider for the category $\mathcal{M}$ the category
$\mathsf{L}_{\Sigma,P}\mathsf{Fun}(\mathsf{dgcat}_f^{op},Sset_{\bullet})$,
see example~\ref{ex2}, we have an equivalence of triangulated
derivators
$$\varphi:
\mathsf{St}(\mathsf{L}_{\Sigma,P}\mathsf{Hot}_{\mathsf{dgcat}_f}) \stackrel{\sim}{\longrightarrow}  \mathsf{HO}(\mathsf{Sp}^{\mathbb{N}}(\mathsf{L}_{\Sigma,P}\mathsf{Fun}(\mathsf{dgcat}_f^{op},
Sset_{\bullet})))\,.$$
Now, stabilize the set $\mathcal{E}_{st}$ defined in the previous
section under the functor loop space and choose for each
element of this stabilized set a representative in the category $\mathsf{Sp}^{\mathbb{N}}(\mathsf{L}_{\Sigma,P}\mathsf{Fun}(\mathsf{dgcat}_f^{op},
Sset_{\bullet}))$. We denote the set of these representatives by
$\widetilde{\mathcal{E}_{st}}$. Since $\mathsf{Sp}^{\mathbb{N}}(\mathsf{L}_{\Sigma,P}\mathsf{Fun}(\mathsf{dgcat}_f^{op},
Sset_{\bullet}))$ is a left proper, cellular Quillen model category,
see \cite{Spectra}, its left Bousfield localization by
$\widetilde{\mathcal{E}_{st}}$ exists. We denote it by $\mathsf{L}_{\widetilde{\mathcal{E}_{st}}}\mathsf{Sp}^{\mathbb{N}}(\mathsf{L}_{\Sigma,P}\mathsf{Fun}(\mathsf{dgcat}_f^{op},
Sset_{\bullet}))$. By lemma~\ref{lettri} it is a stable Quillen
model category.

\begin{remark}
Since the localization morphism
$$ \gamma:\,
\mathsf{St}(\mathsf{L}_{\Sigma,P}\mathsf{Hot}_{\mathsf{dgcat}_f}) 
\stackrel{\mathbb{L}Id}{\longrightarrow}
\mathsf{HO}(\mathsf{L}_{\widetilde{\mathcal{E}_{st}}} \mathsf{Sp}^{\mathbb{N}}(\mathsf{L}_{\Sigma,P}\mathsf{Fun}(\mathsf{dgcat}_f^{op},
Sset_{\bullet})))$$
commutes with homotopy colimits and inverts the set of morphisms
$\mathcal{E}_{st}$, theorem~\ref{invert} allows us to conclude that it
inverts all morphisms $S_K$ for each inclusion $\mathcal{A}
\hookrightarrow \mathcal{B}$ of a full dg subcategory.
\end{remark}

\begin{definition}
\begin{itemize}
\item[-] The {\em Localizing Motivator} of dg categories $\mathcal{M}_{dg}^{loc}$ is the triangulated
derivator associated with the stable Quillen model category 
$$ \mathsf{L}_{\widetilde{\mathcal{E}_{st}}} \mathsf{Sp}^{\mathbb{N}}(\mathsf{L}_{\Sigma,P}\mathsf{Fun}(\mathsf{dgcat}_f^{op},
Sset_{\bullet}))\,.$$
\item[-] The {\em Universal localizing invariant} of dg categories is the canonical morphism of
  derivators $$ \mathcal{U}_l : \mathsf{HO}(\mathsf{dgcat}) \rightarrow \mathcal{M}_{dg}^{loc}\,.$$
\end{itemize}
\end{definition}

We sum up the construction of $\mathcal{M}_{dg}^{loc}$ in the following diagram

$$
\xymatrix{
\underline{\mathsf{dgcat}_f}[S^{-1}] \ar[r] \ar[d]_{\mathsf{Ho}(h)} &
\mathsf{HO}(\mathsf{dgcat}) \ar[dl]^{\mathbb{R}\underline{h}} \ar@/^2pc/[ddddl]^{\mathcal{U}_l}
 \\
\mathsf{L}_{\Sigma}\mathsf{Hot}_{\mathsf{dgcat}_f}
\ar[d]_{\Phi}  \ar@<1ex>[ur]^{\mathbb{L}Re}  & \\
\mathsf{L}_{\Sigma,P}\mathsf{Hot}_{\mathsf{dgcat}_f}
\ar[d]_{stab} & \\
\mathsf{St}({\mathsf{L}_{\Sigma,P}\mathsf{Hot}_{\mathsf{dgcat}_f}})
\ar[d]_{\gamma} & \\
\mathcal{M}_{dg}^{loc} & 
}
$$
Observe that the morphism of derivators $\mathcal{U}_l$ is pointed, commutes with filtered
homotopy colimits and satisfies the following condition:

\begin{itemize}
\item[Dr)] For every inclusion $\mathcal{A} \stackrel{K}{\hookrightarrow}
  \mathcal{B}$ of a full dg subcategory the canonical morphism
$$ S_K:\, \mathsf{cone}(\mathcal{U}_l(\mathcal{A} \stackrel{K}{\hookrightarrow} \mathcal{B}))
\rightarrow \mathcal{U}_l(\mathcal{B}/\mathcal{A})$$
is invertible in $\mathcal{M}_{dg}^{loc}(e)$.
\end{itemize}

We now give a conceptual characterization of condition $Dr)$. Let us now denote by $I$
be the category associated with the graph $0 \leftarrow 1$.

\begin{lemma}\label{mono}
The isomorphism classes in $\mathsf{HO}(\mathsf{dgcat})(I)$ associated
with the inclusions $\mathcal{A} \stackrel{K}{\hookrightarrow}
\mathcal{B}$ of full dg subcategories coincide with the classe of
homotopy monomorphims in $\mathsf{dgcat}$, see section $2$ in \cite{Toen}.
\end{lemma}

\begin{proof}
Recall from section $2$ in \cite{Toen} that in a model category
$\mathcal{M}$ a morphism $X \stackrel{f}{\rightarrow} Y$ is a homotopy
monomorphism if for every object $Z$ in $\mathcal{M}$, the induced
morphism of simplicial sets
$$ \mathsf{Map}(Z,X) \stackrel{f_{\ast}}{\rightarrow}
\mathsf{Map}(Z,Y)$$
induces an injection on $\pi_0$ and isomorphisms on all $\pi_i$ for
$i>0$ (for all base points).

Now, by lemma $2.4$ of \cite{Toen} a dg functor $\mathcal{A}
\stackrel{F}{\rightarrow} \mathcal{B}$ is an homotopy monomorphism on
the quasi-equivalent Quillen model category in $\mathsf{dgcat}$ if and
only if it is {\em quasi-fully faithful}, i.e. for any two objects $X$ and
$Y$ in $\mathcal{A}$ the morphism of complexes
$\mathsf{Hom}_{\mathcal{A}}(X,Y) \rightarrow
\mathsf{Hom}_{\mathcal{B}}(FX,FY)$ is a quasi-isomorphism.

Recall that by corollary $5.10$ of \cite{IMRN} the
mapping space functor $\mathsf{Map}(\mathcal{A},\mathcal{B})$ in the
Morita Quillen model category identifies with the mapping space
$\mathsf{Map}(\mathcal{A},\mathcal{B}_{f})$ in the quasi-equivalent
Quillen model category, where $\mathcal{B}_f$ denotes a Morita fibrant
resolution of $\mathcal{B}$.
This implies that a dg functor $\mathcal{A} \stackrel{F}{\rightarrow}
\mathcal{B}$ is a homotopy monomorphism if and only if $\mathcal{A}_f
\stackrel{F_f}{\rightarrow} \mathcal{B}_f$ is a quasi-fully faithful
dg functor.

Now, notice that an inclusion $\mathcal{A} \hookrightarrow
\mathcal{B}$ of a full dg subcategory is a homotopy
monomorphism. Conversely, let $\mathcal{A} \stackrel{F}{\rightarrow}
\mathcal{B}$ be a homotopy monomorphism. Consider the diagram
$$
\xymatrix{
\widetilde{\mathcal{A}_f} \ar@{^{(}->}[r] & \mathcal{B}_f \ar@{=}[d]  \\
\mathcal{A}_f \ar[u]^{\pi} \ar[r]^{F_f} & \mathcal{B}_f \\
\mathcal{A} \ar[u]^{\sim} \ar[r]_F & \mathcal{B} \ar[u]_{\sim} \,,
}
$$
where $\widetilde{\mathcal{A}_f}$ denotes the full dg subcategory of
$\mathcal{B}_f$ whose objects are those in the image by the dg
functor $F_f$. Since $F_f$ is a quasi-fully faithful dg functor, the dg
functor $\pi$ is a quasi-equivalence.
This proves the lemma.
\end{proof}

\begin{remark}
Lemma~\ref{mono} shows that condition $Dr)$ is equivalent to
\begin{itemize}
\item[Dr')] For every homotopy monomorphism $\mathcal{A}
  \stackrel{F}{\rightarrow} \mathcal{B}$ in
  $\mathsf{HO}(\mathsf{dgcat})(I)$ the canonical morphism 
$$ \mathsf{cone}(\mathcal{U}_l (\mathcal{A} \stackrel{F}{\rightarrow}
\mathcal{B})) \rightarrow \mathcal{U}_l (\mathsf{cone}(F))$$
is invertible in $\mathcal{M}_{dg}^{loc}(e)$.
\end{itemize}
\end{remark}
Let $\mathbb{D}$ be a strong triangulated derivator.
\begin{theorem}\label{principal}
The morphism $\mathcal{U}_l$ induces an equivalence of categories
$$
\underline{\mathsf{Hom}}_!(\mathcal{M}_{dg}^{loc},
\mathbb{D}) \stackrel{\mathcal{U}_l^{\ast}}{\longrightarrow}
\underline{\mathsf{Hom}}_{flt,\,Dr,\,p}(\mathsf{HO}(\mathsf{dgcat}),\mathbb{D})\,,$$
where
$\underline{\mathsf{Hom}}_{flt,\,Dr\,, p}(\mathsf{HO}(\mathsf{dgcat}),\mathbb{D})$
denotes the category of morphisms of derivators which commute with filtered
homotopy colimits, satisfy condition Dr) and preserve the point.
\end{theorem}

\begin{proof}
By theorem~\ref{Cisinsk}, we have the following equivalence of
categories
$$ \underline{\mathsf{Hom}}_!(\mathcal{M}_{dg}^{loc}
\mathbb{D}) \stackrel{\gamma^{\ast}}{\longrightarrow}
\underline{\mathsf{Hom}}_{!,\widetilde{\mathcal{E}_{st}}}( \mathsf{St}(\mathsf{L}_{\Sigma,P}\mathsf{Fun}(\mathsf{dgcat}_f^{op},
Sset_{\bullet})), \mathbb{D})\,.$$
We now show that we have the following equivalence of categories
$$ \underline{\mathsf{Hom}}_{!,\widetilde{\mathcal{E}_{st}}}(\mathsf{St}(\mathsf{L}_{\Sigma,P}\mathsf{Fun}(\mathsf{dgcat}_f^{op},
Sset_{\bullet})), \mathbb{D}) \stackrel{\sim}{\rightarrow} \underline{\mathsf{Hom}}_{!,\mathcal{E}_{st}}(\mathsf{St}(\mathsf{L}_{\Sigma,P}\mathsf{Fun}(\mathsf{dgcat}_f^{op},
Sset_{\bullet})), \mathbb{D})\,.$$
Let $G$ be an element of  $\underline{\mathsf{Hom}}_{!,\mathcal{E}_{st}}(\mathsf{St}(\mathsf{L}_{\Sigma,P}\mathsf{Fun}(\mathsf{dgcat}_f^{op},
Sset_{\bullet})), \mathbb{D})$ and $s$ an element of
$\mathcal{E}_{st}$. We show that the image of $s$ under the functor
$G(e) \circ \Omega(e)$ is an isomorphism in $\mathbb{D}(e)$. Recall from the proof of lemma~\ref{lettri} that the functor $G(e)$
commutes with $\Sigma(e)$. Since the suspension and loop space
functors in $\mathbb{D}(e)$ are inverse of each other we conclude that
the image of $s$ under the functor $G(e) \circ \Omega(e)$ is an
isomorphism in $\mathbb{D}(e)$.
Now, simply observe that the category on the right hand side of the
above equivalence identifies with
$\underline{\mathsf{Hom}}_{flt,\,Dr\,,
  p}(\mathsf{HO}(\mathsf{dgcat}),\mathbb{D})$ under the equivalence
$$
\underline{\mathsf{Hom}}_!(\mathsf{St}(\mathsf{L}_{\Sigma,P} \mathsf{Hot}_{\mathsf{dgcat}_f}),\mathbb{D})
\stackrel{(stab\circ \Phi \circ
  \mathbb{R}\underline{h})^{\ast}}{\longrightarrow} \underline{\mathsf{Hom}}_{flt,p}(\mathsf{HO}(\mathsf{dgcat}),\mathbb{D})\,,$$ 
of proposition~\ref{Trin}. 

This proves the theorem.
\end{proof}
\begin{notation}
We call an object of the right hand side category of
theorem~\ref{principal} a {\em localizing invariant} of dg categories.
\end{notation}
We now present some examples.

\subsection*{Hochschild and cyclic homology}
Let $\mathcal{A}$ be a small $k$-flat $k$-category. The {\it
  Hochschild chain complex} of $\mathcal{A}$ is the complex
concentrated in homological degrees $p\geq 0$ whose $p$th component is
the sum of the 
$$ \mathcal{A}(X_p,X_0) \otimes \mathcal{A}(X_p,X_{p-1})\otimes
\mathcal{A}(X_{p-1}, X_{p-2})\otimes \cdots \otimes
\mathcal{A}(X_0,X_1)\,,$$
where $X_0, \ldots, X_p$ range through the objects of $\mathcal{A}$,
endowed with the differential
$$ d(f_p \otimes \ldots \otimes f_0) = f_{p-1} \otimes \cdots \otimes
f_0 f_p + \sum_{i=1}^p (-1)^i f_p\otimes \cdots \otimes f_i f_{i-1} \otimes
\cdots \otimes f_0\,.$$
Via the cyclic permutations
$$ t_p(f_{p-1} \otimes \cdots \otimes f_0) = (-1)^pf_0\otimes f_{p-1}
\otimes \cdots \otimes f_1$$
this complex becomes a precyclic chain complex and thus gives rise to
a {\it mixed complex} $C(\mathcal{A})$, i.e. a dg module over the dg
algebra $\Lambda = k[B]/(B^2)$, where $B$ is of degree $-1$ and
$dB=0$. All variants of cyclic homology only depend on
$C(\mathcal{A})$ considered in $\mathcal{D}(\Lambda)$. For example,
the cyclic homology of $\mathcal{A}$ is the homology of the complex
$C(\mathcal{A})\overset{\mathbb{L}}{\otimes}_{\Lambda}k$, \emph{cf.}~\cite{Kassel}.

If $\mathcal{A}$ is a $k$-flat differential graded category, its mixed
complex is the sum-total complex of the bicomplex obtained as the
natural re-interpretation of the above complex. If $\mathcal{A}$ is an
arbitrary dg $k$-category, its Hochschild chain complex is defined as
the one of a $k$-flat (e.g. a cofibrant) resolution of $\mathcal{A}$.
The following theorem is proved in \cite{cyclic}.

\begin{theorem}\label{thmC}
 The map $\mathcal{A} \mapsto C(\mathcal{A})$ yields a
  morphism of derivators
 $$\mathsf{HO}(\mathsf{dgcat}) \rightarrow
  \mathsf{HO}(\Lambda-Mod)\,,$$
 which commutes with filtered homotopy colimits, preserves the point
 and satisfies condition Dr).
\end{theorem}

\begin{remark}
By theorem~\ref{principal} the morphism of derivators $C$ factors
through $\mathcal{U}_l$ and so gives rise to a morphism of derivators
$$ C: \mathcal{M}_{dg}^{loc} \rightarrow \mathsf{HO}(\Lambda-Mod)\,.$$
\end{remark}

\subsection*{Non-connective $K$-theory}
Let $\mathcal{A}$ be a small dg category. Its non-connective $K$-theory
spectrum $K(A)$ is defined by applying Schlichting's construction
\cite{Marco} to the Frobenius pair associated with the category of
cofibrant perfect $\mathcal{A}$-modules (to the empty dg category we
associate $0$).
Recall that the conflations in the Frobenius category of cofibrant
perfect $\mathcal{A}$-modules are the short exact sequences which
split in the category of graded $\mathcal{A}$-modules.

\begin{theorem}\label{thmK}
The map $\mathcal{A} \mapsto K(\mathcal{A})$ yields a
  morphism of derivators
 $$\mathsf{HO}(\mathsf{dgcat}) \rightarrow \mathsf{HO}(Spt)\,,$$
  to the derivator associated with the category of spectra, which
  commutes with filtered homotopy
  colimits, preserves the point and satisfies condition Dr).
\end{theorem}

\begin{proof}
Proposition $11.15$ in \cite{Marco}, which is an adaption of theorem
$1.9.8$ in \cite{Thomason}, implies that we have a well defined
morphism of derivators
$$ \mathsf{HO}(\mathsf{dgcat}) \rightarrow \mathsf{HO}(Spt)\,.$$
Lemma $6.3$ in \cite{Marco} implies that this morphism commutes with
filtered homotopy colimits and theorem $11.10$ in \cite{Marco} implies that condition
Dr) is satisfied. 
\end{proof}

\begin{remark}
By theorem~\ref{principal}, the morphism of derivators $K$ factors
through $\mathcal{U}_l$ and so gives rise to a morphism of derivators
$$ K :\mathcal{M}_{dg}^{loc} \rightarrow \mathsf{HO}(Spt)\,.$$
\end{remark}

We now establish a connection between Waldhausen's
$S_{\bullet}$-construction, see~\cite{Wald} and the suspension functor
in the triangulated category $\mathcal{M}_{dg}^{loc}(e)$. Let $\mathcal{A}$
be a Morita fibrant dg category, see \cite{addendum}
\cite{IMRN}. Notice that $\mathsf{Z}^0(\mathcal{A})$ carries a natural
exact category structure obtained by pulling back the garded-split
structure on $\mathcal{C}_{dg}(\mathcal{A})$ along the Yoneda functor
$$
\begin{array}{rcl}
h: \mathsf{Z}^0(\mathcal{A}) & \longrightarrow &
\mathcal{C}_{dg}(\mathcal{A})\\
A & \mapsto & \mathsf{Hom}^{\bullet}(?,A)\,.
\end{array}
$$
\begin{notation}\label{simplicialK}
Remark that the simplicial category $S_{\bullet}\mathcal{A}$, obtained by applying Waldhausen's
$S_{\bullet}$-construction to $\mathsf{Z}^0(\ca)$, admits a natural enrichissement over the complexes. We denote by $S_{\bullet}\ca$ this simplicial Morita
fibrant dg category obtained.
\end{notation} 
Recall that $\Delta$ denotes the simplicial category and $p:\Delta \rightarrow e$ the projection functor.
\begin{proposition}\label{real}
There is a canonical isomorphism in $\mathcal{M}_{dg}^{loc}(e)$
$$ p_!\mathcal{U}_l(S_{\bullet}\mathcal{A})
\stackrel{\sim}{\rightarrow} \mathcal{U}_l(\mathcal{A})[1]\,.$$
\end{proposition}

\begin{proof}
As in \cite[3.3]{MacCarthy}, we consider the sequence in $\mathsf{HO}(\mathsf{dgcat})(\Delta)$
$$ 0 \rightarrow \mathcal{A}_{\bullet} \rightarrow
PS_{\bullet}\mathcal{A} \rightarrow S_{\bullet}\mathcal{A} \rightarrow
0\,,$$
where $\mathcal{A}_{\bullet}$ denotes the constant simplicial dg
category with value $\mathcal{A}$ and $PS_{\bullet}\mathcal{A}$ the
path object of $S_{\bullet}\mathcal{A}$. For each
point $n: e \rightarrow \Delta$, the $n$th component of the above
sequence is the following short exact sequence in
$\mathsf{Ho}(\mathsf{dgcat})$
$$ 0 \rightarrow \mathcal{A} \stackrel{I}{\rightarrow}
PS_n\mathcal{A}=S_{n+1}\mathcal{A} \stackrel{Q}{\rightarrow}
S_n\mathcal{A} \rightarrow 0\,,$$
where $I$ maps $A \in \mathcal{A}$ to the constant sequence
$$ 
0 \rightarrow A \stackrel{Id}{\rightarrow} A \stackrel{Id}{\rightarrow}
\cdots \stackrel{Id}{\rightarrow} A $$
and $Q$ maps a sequence 
$$ 0 \rightarrow A_0 \rightarrow A_1 \rightarrow \cdots \rightarrow
A_n$$
to
$$ A_1/A_0 \rightarrow \cdots \rightarrow A_n/A_0\,.$$
Since the morphism of derivators $\mathcal{U}_l$ satisfies condition
Dr), the conservativity axiom implies that we obtain a triangle
$$ \mathcal{U}_l(\mathcal{A}_{\bullet}) \rightarrow
\mathcal{U}_l(PS_{\bullet}) \rightarrow \mathcal{U}_l(S_{\bullet})
\rightarrow \mathcal{U}_l(\mathcal{A}_{\bullet})[1]$$
in $\mathcal{M}_{dg}^{loc}(\Delta)$. By applying the functor $p_!$, we obtain
the following triangle
$$ p_! \mathcal{U}_l(\mathcal{A}_{\bullet}) \rightarrow
p_!\mathcal{U}_l(PS_{\bullet}\mathcal{A}) \rightarrow p_!\mathcal{U}_l(S_{\bullet}\mathcal{A}) \rightarrow p_!\mathcal{U}_l(\mathcal{A}_{\bullet})[1]$$
in $\mathcal{M}_{dg}^{loc}(e)$.
We now show that we have natural isomorphisms
$$ p_!\mathcal{U}_l(\mathcal{A}_{\bullet})
\stackrel{\sim}{\rightarrow} \mathcal{U}_l(\mathcal{A})$$
and
$$ p_! \mathcal{U}_l(PS_{\bullet}\mathcal{A})
\stackrel{\sim}{\rightarrow} 0\,,$$
in $\mathcal{M}_{dg}^{loc}(e)$, where $0$ denotes the zero object in the
  triangulated category $\mathcal{M}_{dg}^{loc}(e)$. This clearly implies
  the proposition. 
Since the morphisms of derivators $\Phi$, $stab$ and $\gamma$ commute
with homotopy colimits it is enough to show that we have isomorphisms
$$ p_! \mathbb{R}\underline{h}(\mathcal{A}_{\bullet})
\stackrel{\sim}{\rightarrow} \mathbb{R}\underline{h}(\mathcal{A})$$
and 
$$ p_!\mathbb{R}\underline{h}(PS_{\bullet}\mathcal{A})
\stackrel{\sim}{\rightarrow} \ast$$
in $\mathsf{Hot}_{\mathsf{dgcat}_f}(e)$, where $\ast$ denotes the terminal
object in $\mathsf{Hot}_{\mathsf{dgcat}_f}(e)$. Notice that since
$\mathcal{A}$ and $PS_n\mathcal{A}$, $n \geq 0$ are Morita fibrant dg
categories, we have natural isomorphisms
$$ \underline{h}(\mathcal{A}_{\bullet})
\stackrel{\sim}{\rightarrow}\mathbb{R}\underline{h}(\mathcal{A}_{\bullet})$$
and
$$ \underline{h}(PS_{\bullet}\mathcal{A}) \stackrel{\sim}{\rightarrow}
\mathbb{R}\underline{h}(PS_{\bullet}\mathcal{A})$$
in $\mathsf{Hot}_{\mathsf{dgcat}_f}(\Delta)$.

Now, since homotopy colimits in $\mathsf{Fun}(\mathsf{dgcat}^{op}_f,Sset)$
are calculated objectwise and since $h(\mathcal{A}_{\bullet})$ is a
constant simplicial object in $\mathsf{Fun}(\mathsf{dgcat}^{op}_f,Sset)$, corollary $18.7.7$ in \cite{Hirschhorn} implies that we have an
isomorphism
$$ p_!\mathbb{R}\underline{h}(\mathcal{A}_{\bullet})
\stackrel{\sim}{\rightarrow} \mathbb{R}\underline{h}(\mathcal{A})$$
in $\mathcal{M}_{dg}^{loc}(e)$.

Notice also that since $PS_{\bullet}\mathcal{A}$ is a contractible
simplicial object, see \cite{MacCarthy}, so is
$\underline{h}(PS_{\bullet}\mathcal{A})$. Since homotopy colimits in
$\mathsf{Fun}(\mathsf{dgcat}^{op}_f,Sset)$ are calculated objectwise, we
have an isomorphism
$$ p_!\mathbb{R}\underline{h}(PS_{\bullet}\mathcal{A})
\stackrel{\sim}{\rightarrow} \ast$$
in $\mathsf{Hot}_{\mathsf{dgcat}_f}(e)$.

This proves the proposition.
 
\end{proof}

\section{A Quillen model in terms of presheaves of spectra}\label{chapspectra}
In this section, we construct another Quillen model category whose
associated derivator is the localizing motivator of dg categories $\mathcal{M}_{dg}^{loc}$.

Consider the Quillen adjunction
$$
\xymatrix{
\mathsf{Fun}(\mathsf{dgcat}_f^{op},Sset_{\bullet})\ar@<-1ex>[d]_{\Sigma^{\infty}}\\
\mathsf{Sp}^{\mathbb{N}}(\mathsf{Fun}(\mathsf{dgcat}_f^{op},
Sset_{\bullet})) \ar@<-1ex>[u]_{ev_0}\,.
}
$$
Recall from section~\ref{small} that we have a set of morphisms
$(\Sigma \cup \{P \})_+$ in the category
$\mathsf{Fun}(\mathsf{dgcat}_f^{op}, Sset_{\bullet})$. Now stabilize the
image of this set by the derived functor $\mathbb{L}\Sigma^{\infty}$, under the functor loop space in
$\mathsf{Ho}(\mathsf{Sp}^{\mathbb{N}}(\mathsf{Fun}(\mathsf{dgcat}^{op}_f,Sset_{\bullet})))$.
For each one of the morphims thus obtained, choose a representative
in the model category
$\mathsf{Sp}^{\mathbb{N}}(\mathsf{Fun}(\mathsf{dgcat}^{op}_f,Sset_{\bullet}))$.
\begin{notation}
Let us denote this set by $G$ and by
$\mathsf{L}_G\mathsf{Sp}^{\mathbb{N}}(\mathsf{Fun}(\mathsf{dgcat}_f^{op},Sset_{\bullet}))$
the associated left Bousfield localization.
\end{notation}

\begin{proposition}
We have an equivalence of triangulated strong derivators
$$
\mathsf{HO}(\mathsf{Sp}^{\mathbb{N}}(\mathsf{L}_{\Sigma,P}\mathsf{Fun}(\mathcal{M}_f^{op},
Sset_{\bullet}))) \stackrel{\sim}{\longrightarrow}
\mathsf{HO}(\mathsf{L}_G \mathsf{Sp}^{\mathbb{N}}(\mathsf{Fun}(\mathcal{M}_f^{op},Sset_{\bullet})))\,.$$
\end{proposition}

\begin{proof}

Observe that theorems \ref{Cisinsk} and \ref{repre} imply that both
derivators have the same universal property. This proves the proposition.
\end{proof}

\begin{remark}

Notice that the stable Quillen model category
$$
\mathsf{Sp}^{\mathbb{N}}(\mathsf{Fun}(\mathcal{M}_f^{op},Sset_{\bullet}))$$
identifies with 
$$\mathsf{Fun}(\mathcal{M}_f^{op},\mathsf{Sp}^{\mathbb{N}}(Sset_{\bullet}))$$
endowed with the projective model structure.
\end{remark}
The above considerations imply the following proposition.

\begin{proposition}
We have an equivalence of derivators
$$ \mathsf{HO}(\mathsf{L}_{\widetilde{\mathcal{E}_{st}},G}\mathsf{Fun}(\mathsf{dgcat}_f^{op},\mathsf{Sp}^{\mathbb{N}}(Sset_{\bullet}))) \stackrel{\sim}{\longrightarrow} \mathcal{M}_{dg}^{loc} \,.$$
\end{proposition}

\section{Upper  triangular  DG categories}\label{trimat}

In this section we study upper triangular dg categories using the
formalism of Quillen's homotopical algebra. In the next section, we
will relate this important class of dg categories with split short
exact sequences in $\mathsf{Ho}(\mathsf{dgcat})$.

\begin{definition}
An {\em upper triangular} dg category $\underline{\mathcal{B}}$ is
given by an upper triangular matrix
$$
\begin{array}{rcl}
\underline{\mathcal{B}} & := & \begin{pmatrix} \mathcal{A} & X \\ 0 &
  \mathcal{C} \end{pmatrix}\,,
\end{array}
$$
where $\mathcal{A}$ and $\mathcal{C}$ are small dg categories and $X$
is a $\mathcal{A}$-$\mathcal{C}$-bimodule.

A morphism $\underline{F} : \underline{\mathcal{B}} \rightarrow
\underline{\mathcal{B}'}$ of upper triangular dg categories is given by a triple
$\underline{F}:=(F_{\mathcal{A}}, F_{\mathcal{C}}, F_X)$, where
$F_{\mathcal{A}}$, resp. $F_{\mathcal{C}}$, is a dg functor from
  $\mathcal{A}$ to $\mathcal{A}'$, resp. from $\mathcal{C}$ to
  $\mathcal{C}'$, and $F_X$ is a morphism of
  $\mathcal{A}$-$\mathcal{C}$-bimodules from $X$ to $X'$ (we consider
  $X'$ endowed with
  the action induced by $F_{\mathcal{A}}$ and $F_{\mathcal{C}}$). The
  composition is the natural one.
\end{definition} 

\begin{notation}
We denote by $\mathsf{dgcat}^{tr}$ the category of upper triangular dg categories.
\end{notation}

Let $\underline{\mathcal{B}} \in \mathsf{dgcat}^{tr}$.
\begin{definition}
Let $|\underline{\mathcal{B}}|$ be the {\em totalization} of
$\underline{\mathcal{B}}$, i.e. the small dg category whose set of
objects is the disjoint union of the set of objects of $\mathcal{A}$
and $\mathcal{C}$ and whose morphisms are given by
$$
\begin{array}{rcl}
\mathsf{Hom}_{|\underline{\mathcal{B}}|}(x,x') & := & 
\left\{
\begin{array}{ccl}
\mathsf{Hom}_{\mathcal{A}}(x,x') & \mbox{if} & x,\, x' \in \mathcal{A} \\
\mathsf{Hom}_{\mathcal{C}}(x,x') & \mbox{if} & x,\, x' \in \mathcal{C} \\
X(x,x') & \mbox{if} & x \in \mathcal{A},\, x' \in \mathcal{C} \\
0 & \mbox{if} &  x \in \mathcal{C},\, x' \in \mathcal{A}
\end{array}\right.\,.
\end{array}
$$
\end{definition}
We have the following adjunction
$$
\xymatrix{
\mathsf{dgcat}^{tr} \ar@<-1ex>[d]_{|-|} \\
\mathsf{dgcat} \ar@<-1ex>[u]_I \,,
}
$$
where 
$$
\begin{array}{rcl}
I(\mathcal{B}') & := & \begin{pmatrix} \mathcal{B}' & \mathsf{Hom}_{\mathcal{B}'}(-,-) \\
0 & \mathcal{B}' \end{pmatrix}\,.
\end{array}
$$
\begin{lemma}\label{triancomp}
The category $\mathsf{dgcat}^{tr}$ is complete and cocomplete.
\end{lemma}
\begin{proof}
Let $\{\underline{\mathcal{B}_j}\}_{j \in J}$ be a diagram in
$\mathsf{dgcat}^{tr}$. Observe that the upper triangular dg category
$$
\begin{pmatrix}
\underset{j \in J}{\mbox{colim}\,\mathcal{A}_j} & \underset{j \in
  J}{\mbox{colim}}\,|\underline{\mathcal{B}_j}|(-,-)  \\
0 & \underset{j \in J}{\mbox{colim}}\,\mathcal{C}_j
\end{pmatrix}\,,
$$
where $\underset{j \in
  J}{\mbox{colim}}\,|\underline{\mathcal{B}_j}|(-,-)$ is the $\underset{j \in
  J}{\mbox{colim}}\,\mathcal{A}_j$-$\underset{j \in
  J}{\mbox{colim}}\,\mathcal{C}_j$-bimodule naturally associated
with the dg category $\underset{j \in
  J}{\mbox{colim}}\,|\underline{\mathcal{B}_j}|$, corresponds to $\underset{j \in
  J}{\mbox{colim}}\,\underline{\mathcal{B}_j}$. Observe also that the
upper triangular dg category
$$
\begin{pmatrix}
\underset{j \in J}{\mathsf{lim}}\, \mathcal{A}_j & \underset{j \in J}{\mathsf{lim}}\,X_j \\
0 & \underset{j \in J}{\mathsf{lim}}\,\mathcal{C}_j
\end{pmatrix}\,,
$$
corresponds to $\underset{j \in
  J}{\mathsf{lim}}\,\underline{\mathcal{B}_j}$.
This proves the lemma.
\end{proof}

\begin{notation}
Let $p_1(\underline{\mathcal{B}}) := \mathcal{A}$ and $p_2(\underline{\mathcal{B}}) := \mathcal{C}$.
\end{notation}
We have at our disposal the following adjunction
$$
\xymatrix{
\mathsf{dgcat}^{tr} \ar@<1ex>[d]^{p_1 \times p_2} \\
\mathsf{dgcat}\times\mathsf{dgcat} \ar@<1ex>[u]^E \,,
}
$$
where 
$$
\begin{array}{rcl}
E(\mathcal{B}', \mathcal{B}'') & := & \begin{pmatrix} \mathcal{B}' & 0 \\
0 & \mathcal{B}'' \end{pmatrix}\,.
\end{array}
$$
Recall from \cite{addendum}\cite{IMRN}\cite{cras} that $\mathsf{dgcat}$ admits a structure of cofibrantly
generated Quillen model category whose weak equivalences are the
Morita dg functors. This structure clearly induces a componentwise
model structure on $\mathsf{dgcat} \times \mathsf{dgcat}$ which is
also cofibrantly generated.

\begin{proposition}\label{prop1}
The category $\mathsf{dgcat}^{tr}$ admits a structure of cofibrantly generated
Quillen model category whose weak equivalences, resp. fibration, are
the morphisms $\underline{F}: \underline{\mathcal{B}} \rightarrow
\underline{\mathcal{B}'}$ such that $(p_1 \times p_2) (\underline{F})$
are quasi-equivalences, resp. fibrations, in $\mathsf{dgcat}\times \mathsf{dgcat}$.
\end{proposition}

\begin{proof}
We show that
the previous adjunction $(E,p_1 \times p_2)$ verifies conditions $(1)$ and $(2)$ of theorem
$11.3.2$ from \cite{Hirschhorn}. 
\begin{itemize}
\item[(1)] Since the functor $E$ is also a right adjoint to $p_1
  \times p_2$, the functor $p_1 \times p_2$ commutes with colimits and
  so condition $(1)$ is verified.
\item[(2)] Let $J$, resp. $J \times J$, be the set of generating trivial
cofibrations in $\mathsf{dgcat}$, resp. in $\mathsf{dgcat} \times \mathsf{dgcat}$. Since the
functor $p_1 \times p_2$ commutes with filtered colimits it is enough
to prove the following: let $\underline{G}: \underline{\mathcal{B}'}
\rightarrow \underline{\mathcal{B}''}$ be an element of the set
$E(J \times J)$, $\underline{\mathcal{B}}$ an object in
$\mathsf{dgcat}^{tr}$ and $\underline{\mathcal{B}'} \rightarrow
\underline{\mathcal{B}}$ a morphism in $\mathsf{dgcat}^{tr}$. Consider
the following push-out in $\mathsf{dgcat}^{tr}$~:
$$
\xymatrix{
\underline{\mathcal{B}'} \ar[r] \ar[d]_{\underline{G}}
\ar@{}[dr]|{\lrcorner} &
\underline{\mathcal{B}} \ar[d]^{\underline{G}_{\ast}} \\
\underline{\mathcal{B}''} \ar[r] &
\underline{\mathcal{B}''}\underset{\underline{\mathcal{B}'}}{\coprod}\underline{\mathcal{B}}\,.
}
$$
We now prove that $(p_1 \times p_2)(\underline{G}_{\ast})$ is a
weak-equivalence in $\mathsf{dgcat}\times \mathsf{dgcat}$. Observe
that the image of the previous push-out under the functors $p_1$ and
$p_2$ correspond to the following two push-outs in $\mathsf{dgcat}$~:
$$
\xymatrix{
*+<1pc>{\mathcal{A}'} \ar[r] \ar@{>->}[d]_{G_{\mathcal{A}'}}^{\sim} \ar@{}[dr]|{\lrcorner}  & \mathcal{A}
  \ar[d]^{G_{\mathcal{A}'_{\ast}}}  &  *+<1pc>{\mathcal{C}'} \ar[r]
  \ar@{>->}[d]_{G_{\mathcal{C}'}}^{\sim}  \ar@{}[dr]|{\lrcorner}  & \mathcal{C}
  \ar[d]^{G_{\mathcal{C}'_{\ast}}} \\
\mathcal{A}'' \ar[r] & \mathcal{A}'' \underset{\mathcal{A}'}{\coprod}
\mathcal{A} & \mathcal{C}'' \ar[r] & \mathcal{C}'' \underset{\mathcal{C}'}{\coprod}
\mathcal{C}\,.
}
$$
Since $G_{\mathcal{A}'_{\ast}}$ and $G_{\mathcal{C}'_{\ast}}$
belong to $J$ the morphism
$$(p_1\times p_2)(\underline{G}_{\ast})=(G_{\mathcal{A}'_{\ast}},
G_{\mathcal{C}'_{\ast}})$$
is a weak-equivalence in $\mathsf{dgcat} \times \mathsf{dgcat}$. This proves condition $(2)$. 
\end{itemize} 
The proposition is then proven.
\end{proof}

Let $\underline{\mathcal{B}}$, $\underline{\mathcal{B}'} \in \mathsf{dgcat}^{tr}$.

\begin{definition}
A morphism $\underline{F}: \underline{\mathcal{B}} \rightarrow
\underline{\mathcal{B}'}$ is a {\em total Morita dg functor} if
$F_{\mathcal{A}}$ and $F_{\mathcal{C}}$ are Morita dg functors, see
\cite{addendum}\cite{IMRN}\cite{cras}, and $F_X$
    is a quasi-isomorphism of $\mathcal{A}$-$\mathcal{C}$-bimodules.
\end{definition}

\begin{remark}
Notice that if $\underline{F}$ is a total Morita dg functor then
$|\underline{F}|$ is a Morita dg functor in $\mathsf{dgcat}$ but the converse is not true.
\end{remark}

\begin{theorem}\label{theo1}
The category $\mathsf{dgcat}^{tr}$ admits a structure of cofibrantly generated
Quillen model category whose weak equivalences $\mathcal{W}$ are the
total Morita dg functors and whose fibrations are the morphisms
$\underline{F}:\underline{\mathcal{B}} \rightarrow
\underline{\mathcal{B}'}$ such that $F_{\mathcal{A}}$ and
  $F_{\mathcal{C}}$ are Morita fibrations, see 
  \cite{addendum} \cite{IMRN}, and $F_X$ is a componentwise surjective morphism of bimodules.
\end{theorem}

\begin{proof}
The proof is based on enlarging the set $E(I \times I)$,
resp. $E(J \times J)$, of
generating cofibrations, resp. generating trivial cofibrations, of the
Quillen model structure of proposition~\ref{prop1}.

Let $\tilde{I}$ be the set of morphisms in $\mathsf{dgcat}^{tr}$
$$
\begin{array}{rcl}
\begin{pmatrix} k & S^{n-1} \\ 0 & k \end{pmatrix} & \hookrightarrow &  
\begin{pmatrix} k & D^{n} \\ 0 & k \end{pmatrix}\,, n \in \mathbb{Z}\,,  
\end{array}
$$
where $S^{n-1}$ is the complex $k[n-1]$ and $D^n$ the mapping cone on
the identity of $S^{n-1}$. The $k$-$k$-bimodule $S^{n-1}$ is sent to
$D^n$ by the identity on $k$ in degree $n-1$.

Consider also the set $\tilde{J}$ of morphisms in $\mathsf{dgcat}^{tr}$
$$
\begin{array}{rcl}
\begin{pmatrix} k & 0 \\ 0 & k \end{pmatrix} & \hookrightarrow &  
\begin{pmatrix} k & D^{n} \\ 0 & k \end{pmatrix}\,, n \in \mathbb{Z}\,.  
\end{array}
$$
Observe that a morphism $\underline{F} :\underline{\mathcal{B}}
\rightarrow \underline{\mathcal{B}'}$ in $\mathsf{dgcat}^{tr}$ has the
right lifting property (=R.L.P.) with respect to the set $\tilde{J}$,
resp. $\tilde{I}$, if and only if $F_X$ is a componentwise surjective morphism,
resp. surjective quasi-isomorphism, of $\mathcal{A}$-$\mathcal{C}$-bimodules.

Define $I:= E(I\times I) \cup \tilde{I}$ as the set of {\em generating
cofibrations} in $\mathsf{dgcat}^{tr}$ and $J:=E(J\times J)\cup
\tilde{J}$ as the set of {\em generating trivial cofibrations}. We now prove
that conditions $(1)$-$(6)$ of theorem $2.1.19$ from \cite{Hovey} are
satisfied. This is clearly the case for conditions $(1)$-$(3)$. 
\begin{itemize}

\item[(4)] We now prove that $J\text{-}cell \subset \mathcal{W}$, see
\cite{Hirschhorn}. Since by proposition~\ref{prop1} we have $E(J\times J)\text{-}cell
\subset \mathcal{W}'$, where $\mathcal{W}'$ denotes the weak equivalences of
proposition~\ref{prop1}, it is enough to prove that pushouts with respect
to any morphism in $\tilde{J}$ belong to $\mathcal{W}$.
Let $n$ be an integer and $\underline{\mathcal{B}}$ an object in
$\mathsf{dgcat}^{tr}$. Consider the following push-out in $\mathsf{dgcat}^{tr}$~:
$$
\xymatrix{
{\begin{pmatrix} k & 0 \\ 0 & k \end{pmatrix}} \ar[r]^{\underline{T}}
\ar@{^{(}->}[d] \ar@{}[dr]|{\lrcorner}
& \underline{\mathcal{B}} \ar[d]^{\underline{R}} \\
{\begin{pmatrix} k & D^{n} \\ 0 & k \end{pmatrix}} \ar[r] & \underline{\mathcal{B}'}\,.
}
$$
Notice that the morphism $\underline{T}$ corresponds to specifying
an object $A$ in $\mathcal{A}$ and an object $C$ in $\mathcal{C}$. The
upper triangular dg category $\underline{\mathcal{B}'}$ is then
obtained from $\underline{\mathcal{B}}$ by gluing a new morphism of
degree $n$ from $A$ to $C$. Observe that $R_{\mathcal{A}}$ and
  $R_{\mathcal{C}}$ are the identity dg functors and that $R_X$ is a
  quasi-isomorphism of bimodules. This shows that $\underline{R}$
  belongs to $\mathcal{W}$ and so condition $(4)$ is proved.

\item[(5-6)] We now show that R.L.P.($I$)=R.L.P.($J$) $\cap$ $\mathcal{W}$.
The proof of proposition~\ref{prop1} implies that
R.L.P.($E(I\times I)$)=R.L.P.($E(J \times J)$) $\cap$ $\mathcal{W}'$. Let $\underline{F}:\underline{\mathcal{B}} \rightarrow
  \underline{\mathcal{B}'}$ be a morphism in
  R.L.P.($\tilde{I}$). Clearly $\underline{F}$ belongs to
  R.L.P.($\tilde{J}$) and $F_X$ is a quasi-isomorphism of
  bimodules. This shows that R.L.P.($I$) $\subset$ R.L.P.($J$) $\cap$
  $\mathcal{W}$. 
Let now $\underline{F}:\underline{\mathcal{B}} \rightarrow
  \underline{\mathcal{B}'}$ be a morphism in
  R.L.P.($\tilde{J}$) $\cap$ $\mathcal{W}$. Clearly $\underline{F}$
  belongs to R.L.P.($\tilde{I}$) and so R.L.P.($J$) $\cap$
  $\mathcal{W}$ $\subset$ R.L.P($I$). This proves conditions $(5)$ and $(6)$.

\end{itemize}
This proves the theorem.
\end{proof}

\begin{remark}\label{filtri}
Notice that the Quillen model structure of theorem~\ref{theo1} is
cellular, see \cite{Hirschhorn} and that the domains and codomains of
$I$ (the set of generating cofibrations) are cofibrant,
$\aleph_0$-compact, $\aleph_0$-small and homotopically finitely
presented, see definition $2.1.1.$ from \cite{Toen-Vaq}. This implies
that we are in the conditions of proposition~\ref{prop} and so any object
$\underline{\mathcal{B}}$ in $\mathsf{dgcat}^{tr}$ is weakly equivalent
to a filtered colimit of strict finite $I$-cell objects.
\end{remark}

\begin{proposition}\label{comspl}
If $\underline{\mathcal{B}}$ be a strict finite $I$-cell object in
$\mathsf{dgcat}^{tr}$, then $p_1(\underline{\mathcal{B}})$,
$p_2(\underline{\mathcal{B}})$ and $|\underline{\mathcal{B}}|$ are
strict finite $I$-cell objects in $\mathsf{dgcat}$.
\end{proposition}

\begin{proof}
We consider the following inductive argument:
\begin{itemize}
\item[-] notice that the initial object in $\mathsf{dgcat}^{tr}$ is 
$$ \begin{pmatrix} \emptyset & 0 \\
0 & \emptyset \end{pmatrix}$$
and it is sent to $\emptyset$ (the initial object in
$\mathsf{dgcat}$) by the functors $p_1$, $p_2$ and $|-|$.
\item[-] suppose that $\underline{\mathcal{B}}$ is an upper triangular
    dg category such that $p_1(\underline{\mathcal{B}})$,
    $p_2(\underline{\mathcal{B}})$ and $|\underline{\mathcal{B}}|$ are
    strict finite $I$-cell objects in $\mathsf{dgcat}$. Let
    $\underline{G}: \underline{\mathcal{B}'} \rightarrow
    \underline{\mathcal{B}''}$ be an element of the set $I$ in
    $\mathsf{dgcat}^{tr}$, see the proof of theorem~\ref{theo1}, and $\underline{\mathcal{B}'} \rightarrow
    \underline{\mathcal{B}}$ a morphism. Consider the following
    push-out in $\mathsf{dgcat}^{tr}$:
$$
\xymatrix{
\underline{\mathcal{B}'} \ar[r] \ar[d]_{\underline{G}}
\ar@{}[dr]|{\lrcorner} &
\underline{\mathcal{B}} \ar[d]^{\underline{G}_{\ast}} \\
\underline{\mathcal{B}''} \ar[r] &
\mathsf{PO}\,.
}
$$ 
We now prove that
$p_1(\mathsf{PO})$,
$p_2(\mathsf{PO})$ and
$|\mathsf{PO}|$
are strict finite $I$-cell objects in $\mathsf{dgcat}$. We consider the following two cases~:
\begin{itemize}
\item[1)] $\underline{G}$ belongs to
  $E(I \times I)$: observe that $p_1(\mathsf{PO})$,
$p_2(\mathsf{PO})$ and
$|\mathsf{PO}|$
correspond exactly to the following push-outs in $\mathsf{dgcat}$:
$$
\xymatrix{
*+<1pc>{\mathcal{A}'} \ar[r] \ar@{>->}[d]_{G_{\mathcal{A}'}}^{\sim} \ar@{}[dr]|{\lrcorner}  & \mathcal{A}
  \ar[d]  &   *+<1pc>{\mathcal{C}'} \ar[r]
  \ar@{>->}[d]_{G_{\mathcal{C}'}}^{\sim}  \ar@{}[dr]|{\lrcorner}  & \mathcal{C}
  \ar[d]  & \mathcal{A}' \coprod \mathcal{C}' \ar[r]
  \ar@{}[dr]|{\lrcorner} \ar[d]_{G_{\mathcal{A}'}\amalg
    G_{\mathcal{C}'}}  & |\underline{\mathcal{B}}| \ar[d] \\
\mathcal{A}'' \ar[r] & p_1(\mathsf{PO}) & \mathcal{C}'' \ar[r] & p_2(\mathsf{PO}) & \mathcal{A}'' \coprod \mathcal{C}'' \ar[r] & |\mathsf{PO}| \,.
}
$$
Since $G_{\mathcal{A}'}$ and $G_{\mathcal{C}'}$ belong to $I$ this
case is proved.
\end{itemize}
\item[2)] $\underline{G}$ belongs to $\tilde{I}$: observe
  that
  $p_1(\mathsf{PO})$ identifies with $\mathcal{A}$, 
$p_2(\mathsf{PO})$ identifies with $\mathcal{C}$ and
$|\mathsf{PO}|$
corresponds to the following push-out in $\mathsf{dgcat}$~:
$$
\xymatrix{
\mathcal{C}(n) \ar[r] \ar[d]_{S(n)} \ar@{}[dr]|{\lrcorner} &
|\underline{\mathcal{B}}| \ar[d] \\
\mathcal{P}(n) \ar[r] &
|\mathsf{PO}|,
}
$$
where $S(n)$ is a generating cofibration in $\mathsf{dgcat}$, see
section $2$ in \cite{cras}. This proves this case.
\end{itemize}
The proposition is proven.
\end{proof}

\section{Split short exact sequences}\label{splitex}
In this section, we establish the connexion between split short exact
sequences of dg categories and upper triangular dg categories.

\begin{definition}
A split short exact sequence of dg categories is a short exact
sequence of dg categories, see \cite{ICM}, which is Morita equivalent
to one of the form
$$ 
\xymatrix{
0 \ar[r] & \mathcal{A} \ar[r]_{i_{\mathcal{A}}} & \mathcal{B} \ar@<-1ex>[l]_R
\ar[r]_P & \mathcal{C} \ar@<-1ex>[l]_{i_{\mathcal{C}}} \ar[r] & 0 \,,
}
$$
where we have $P\circ i_{\mathcal{A}} =0$, $R$ is a dg functor right adjoint to $i_{\ca}$, $i_{\cc}$ is a dg functor right adjoint to $P$ and we have $P\circ i_{\mathcal{C}}=
Id_{\mathcal{C}}$ and $R\circ i_{\mathcal{A}}=Id_{\mathcal{A}}$ via the adjunction morphisms.
\end{definition}
To a split short exact sequence, we can naturally associate the upper triangular
dg category
$$ 
\begin{array}{rcl}
\underline{\mathcal{B}} & := & {\begin{pmatrix} \mathcal{A} &
    \mathsf{Hom}_{\mathcal{B}}(i_{\mathcal{C}}(-), i_{\mathcal{A}}(-)) \\
0 & \mathcal{C} \end{pmatrix}}\,.
\end{array}
$$

Conversely to an upper triangular dg category
$\underline{\mathcal{B}}$ such that $\mathcal{C}$
admits a zero object (for instance if $\mathcal{C}$ is Morita fibrant), we can
associate a split short exact sequence
$$ 
\xymatrix{
0 \ar[r] & \mathcal{A} \ar[r]_{i_{\mathcal{A}}} & |\underline{\mathcal{B}}| \ar@<-1ex>[l]_R
\ar[r]_P & \mathcal{C} \ar@<-1ex>[l]_{i_{\mathcal{C}}} \ar[r] & 0 \,,
}
$$
where $P$ and $R$ are the projection dg functors.
Moreover, this construction is functorial in $\underline{\mathcal{B}}$
and sends total Morita equivalences to Morita equivalent split short
exact sequences. Notice also that by lemma~\ref{triancomp} this
functor preserves colimits.

\begin{proposition}\label{aproxsplit}
Every split short exact sequence of dg categories is weakly equivalent
to a filtered homotopy colimit of split short exact sequences whose
components are strict finite $I$-cell objects in $\mathsf{dgcat}$.
\end{proposition}

\begin{proof}
Let 
$$ 
\xymatrix{
0 \ar[r] & \mathcal{A} \ar[r]_{i_{\mathcal{A}}} & \mathcal{B} \ar@<-1ex>[l]_R
\ar[r]_P & \mathcal{C} \ar@<-1ex>[l]_{i_{\mathcal{C}}} \ar[r] & 0 \,,
}
$$
be a split short exact sequence of dg categories. We can supose that
$\mathcal{A}$, $\mathcal{B}$ and $\mathcal{C}$ are Morita fibrant dg
categories, see \cite{addendum} \cite{IMRN}. Consider the upper
triangular dg category
$$ 
\begin{array}{rcl}
\underline{\mathcal{B}} & := & {\begin{pmatrix} \mathcal{A} & \mathsf{Hom}_{\mathcal{B}}(i_{\mathcal{C}}(-),i_{\mathcal{A}}(-)) \\
0 & \mathcal{C} \end{pmatrix}}\,.
\end{array}
$$
Now by remark~\ref{filtri}, $\underline{\mathcal{B}}$ is
equivalent to a filtered colimit of strict finite $I$-cell objects in
$\mathsf{dgcat}^{tr}$. Consider the image of this diagram by the
functor, described above,  which sends an upper triangular dg category to a split
short exact sequence. By proposition~\ref{comspl} the
components of each split short exact sequence of this diagram are
strict finite $I$-cell objects in $\mathsf{dgcat}$. Since the category
$\mathsf{dgcat}$ satisfies the conditions of proposition~\ref{prop}, filtered homotopy colimits are equivalent to filtered colimits and so the proposition is proven.
\end{proof}

\section{Quasi-Additivity}\label{quasi}
Recall from section~\ref{labuniv} that we have at our disposal the
Quillen model category
$\mathsf{L}_{\Sigma,P} \mathsf{Fun}(\mathsf{dgcat}_f^{op},Sset)$ which is
  {\em homotopically pointed}, i.e. the morphism $\emptyset
  \rightarrow \ast$, from the initial object $\emptyset$ to the
  terminal one $\ast$, is a weak equivalence. We now consider a
  strictly pointed Quillen model.

\begin{proposition}\label{point}
We have a Quillen equivalence
$$
\xymatrix{
\ast \downarrow \mathsf{L}_{\Sigma,P}\mathsf{Fun}(\mathcal{M}^{op}_f,Sset) \ar@<1ex>[d]^U \\
\mathsf{L}_{\Sigma,P} \mathsf{Fun}(\mathcal{M}_f^{op},Sset) \ar@<1ex>[u]^{(-)_{+}}\,,
}
$$
where $U$ denotes the forgetful functor.
\end{proposition}
This follows from the fact that the category $\mathsf{L}_{\Sigma,P} \mathsf{Fun}(\mathsf{dgcat}_f^{op},Sset)$ is
  homotopically pointed and from the following general argument.

\begin{proposition}
Let $\mathcal{M}$ be a homotopically pointed Quillen model
category. We have a Quillen equivalence.
$$
\xymatrix{
\ast \downarrow \mathcal{M} \ar@<1ex>[d]^U \\
\mathcal{M} \ar@<1ex>[u]^{(-)_{+}}\,,
}
$$
where $U$ denotes the forgetful functor.
\end{proposition}

\begin{proof}
Clearly the functor $U$ preserves cofibrations, fibrations and weak
equivalences, by construction. Let now $N\in \mathcal{M}$ and $M \in
\ast \downarrow \mathcal{M}$. Consider the following commutative
diagram in $\mathcal{M}$
$$
\xymatrix{
N \simeq \emptyset\coprod N \ar[rr]^f \ar[dr]^{\sim}_{i\amalg \mathbf{1}} & & U(M)
\\
 & \ast \coprod N \ar[ur]_{f^{\sharp}} & \,,
}
$$
where $f^{\sharp}$ is the morphism which corresponds to $f$, considered
as a morphism in $\mathcal{M}$, under the adjunction and $i:\emptyset
\stackrel{\sim}{\rightarrow} \ast$. Since the morphism $i\amalg \mathbf{1}$
corresponds to the homotopy colimit of $i$ and $1$, which are both
weak equivalences, proposition~\ref{Cisin} implies that  $i\amalg \mathbf{1}$  is a
weak equivalence. Now, by the `$2$ out of $3$' property, the morphism
$f$ is a weak equivalence if and only if $f^{\sharp}$ is one. This
proves the proposition.
\end{proof}

\begin{notation}
Let $\ca$ and $\cb$ be small dg categories. We denote by $\rep_{mor}(\ca,\cb)$ the full dg subcategory of $\cc_{dg}(\ca_c^{op}\otimes \cb)$, where $\ca_c$ denotes a cofibrant resolution of $\ca$, whose objects are the bimodules $X$ such that $X(?,A)$ is a compact object in $\cd(\cb)$ for all $A \in \ca_c$ and which are cofibrant as bimodules.
We denote by $w \mathcal{A}$ the category of homotopy equivalences of $\mathcal{A}$ and by $N.w\mathcal{A}$ its nerve.
\end{notation}

Now, consider the morphism
$$
\begin{array}{ccl}
\mathsf{Ho}(\mathsf{dgcat}) & \rightarrow & \mathsf{Ho}(\ast
\downarrow \mathsf{L}_{\Sigma,P}\mathsf{Fun}(\mathcal{M}^{op}_f,Sset)) \\
\mathcal{A} & \mapsto & \left\{\begin{array}{l}
    \mathsf{Hom}_{\mathsf{dgcat}}(\Gamma(?),\mathcal{A}_f)_+\\
\simeq \mathsf{Map}_{\mathsf{dgcat}}(?,\mathcal{A})_+\\
\simeq N.w\mathsf{rep}_{mor}(?,\mathcal{A})_+ \end{array} \right.
\end{array}
$$
which by sections~\ref{homotopy}, \ref{chappoint} and
proposition~\ref{point} corresponds to the component $(\Phi\circ
\mathbb{R}\underline{h})(e)$ of the morphism of derivators
$$ \Phi \circ \mathbb{R}\underline{h}: \mathsf{HO}(\mathsf{dgcat})\longrightarrow
\mathsf{L}_{\Sigma,P}\mathsf{Hot}_{\mathsf{dgcat}_f}\,,$$
see proposition~\ref{ext}. Observe that the simplicial presheaf
$N.w\mathsf{rep}_{mor}(?,\mathcal{A})$ is already canonically pointed.

\begin{proposition}
The canonical morphism
$$\Psi : N.w\mathsf{rep}_{mor}(?,\mathcal{A})_+ \rightarrow N.w \mathsf{rep}_{mor}(?,\mathcal{A})$$
is a weak equivalence in $\ast \downarrow \mathsf{L}_{\Sigma,P}\mathsf{Fun}(\mathcal{M}^{op}_f,Sset)$.
\end{proposition}

\begin{proof}
Observe that $N.w\mathsf{rep}_{mor}(?,\mathcal{A})$ is a fibrant object
in $\ast
  \downarrow \mathsf{L}_{\Sigma,P}\mathsf{Fun}(\mathcal{M}^{op}_f,Sset)$
  and that the canonical morphism $\Psi$ corresponds to the co-unit
  of the adjunction of proposition~\ref{point}. Since this adjunction
  is a Quillen equivalence, the proposition is proved.
\end{proof}

Recall now from remark~\ref{remar} that we have a canonical equivalence of
pointed derivators
$$ \mathsf{L}_{\Sigma,P}\mathsf{Hot}_{\mathsf{dgcat}_f}
\stackrel{\sim}{\longrightarrow}
{\mathsf{L}_{\Sigma,P}\mathsf{Hot}_{\mathsf{dgcat}_f}}_{\bullet}\,,$$
where ${\mathsf{L}_{\Sigma,P}\mathsf{Hot}_{\mathsf{dgcat}_f}}_{\bullet}$
is the derivator associated with the Quillen model category
$\mathsf{L}_{\Sigma,P}\mathsf{Fun}(\mathsf{dgcat}_f^{op},Sset_{\bullet})$.
From now on, we will consider this Quillen model. We have the
following morphism of derivators
$$\Phi \circ \mathbb{R}\underline{h} : \mathsf{HO}(\mathsf{dgcat}) \longrightarrow
{\mathsf{L}_{\Sigma,P}\mathsf{Hot}_{\mathsf{dgcat}_f}}_{\bullet}\,,$$
which commutes with filtered homotopy colimits and preserves the point.

\begin{notation}
\begin{itemize}
\item[-] We denote by $\mathcal{E}^s$ the set of retractions of dg categories
$$ \xymatrix{
\mathcal{G} \ar[r]_{i_{\mathcal{G}}} & \mathcal{H} \ar@<-1ex>[l]_R \,,
}
$$
where $\mathcal{G}$ and $\mathcal{H}$ are strict finite $I$-cell
objects in $\mathsf{dgcat}$, $i_{\cg}$ is a fully faithful dg functor, $R$ is a right adjoint to $i_{\cg}$ and $R\circ i_{\mathcal{G}} =Id_{\mathcal{G}}$. 
\item[-] We denote by $\mathcal{E}^s_{un}$ the set of morphisms $S_L$
  in
  ${\mathsf{L}_{\Sigma,P}\mathsf{Hot}_{\mathsf{dgcat}_f}}_{\bullet}(e)$, see section~\ref{chapquotient}, where $L$ belongs to the set $\mathcal{E}^s$.
\end{itemize}
\end{notation}

Now choose for each element of the set $\mathcal{E}^s_{un}$ a
representative in the category
$\mathsf{L}_{\Sigma,P}\mathsf{Fun}(\mathsf{dgcat}_f^{op},Sset_{\bullet})$.
We denote this set of representatives by
$\widetilde{\mathcal{E}^s_{un}}$. Since \newline
$\mathsf{L}_{\Sigma,P}\mathsf{Fun}(\mathsf{dgcat}_f^{op},Sset_{\bullet})$
is a left proper, cellular Quillen model category, see
\cite{Hirschhorn}, its left Bousfield localization by
$\widetilde{\mathcal{E}^s_{un}}$ exists. We denote it by
$\mathsf{L}_{\widetilde{\mathcal{E}^s_{un}}}
\mathsf{L}_{\Sigma,P}\mathsf{Fun}(\mathsf{dgcat}_f^{op},Sset_{\bullet})$
and by ${\mathsf{L}_{\widetilde{\mathcal{E}^s_{un}}}
\mathsf{L}_{\Sigma,P}\mathsf{Hot}_{\mathsf{dgcat}_f}}_{\bullet}$ the
associated derivator. We have the following morphism of derivators
$$
\Psi: {\mathsf{L}_{\Sigma,P}\mathsf{Fun}\mathsf{Hot}_{\mathsf{dgcat}_f}}_{\bullet} \rightarrow
 {\mathsf{L}_{\widetilde{\mathcal{E}^s_{un}}}
\mathsf{L}_{\Sigma,P}\mathsf{Fun}\mathsf{Hot}_{\mathsf{dgcat}_f}}_{\bullet}\,.
$$
\begin{remark}\label{fib}
\begin{itemize}

\item[-] Notice that by construction the domains and codomains of the
  set $\widetilde{\mathcal{E}^s_{un}}$ are homotopically finitely
  presented objects. Therefore by lemma~\ref{gener} the set 
$$ \mathcal{G}= \{\mathbf{F}^X_{\Delta[n]_+ / \partial \Delta[n]_+ } |
\, X \in
\mathcal{M}_f\,, n\geq 0 \} \,,$$
of cofibers of the generating cofibrations in
$\mathsf{Fun}(\mathcal{M}_f^{op},Sset_{\bullet})$ is a set of small weak
generators in $\mathsf{Ho}(\mathsf{L}_{\widetilde{\mathcal{E}^s_{un}}}
\mathsf{L}_{\Sigma,P}\mathsf{Fun}(\mathcal{M}_f^{op},Sset_{\bullet}))$.

\item[-] Notice also that proposition~\ref{aproxsplit} implies that variants of
proposition~\ref{cons} and theorem~\ref{invert} are also verified:
simply consider the set $\mathcal{E}^s$ instead of $\mathcal{E}$ and a
retraction of dg categories instead of an inclusion of a full dg
subcategory. The proofs are exactly the same.
\end{itemize}
\end{remark}

\begin{definition}
\begin{itemize}
\item[-] The {\em Unstable Motivator} of dg categories
  $\mathcal{M}^{unst}_{dg}$ is the derivator associated with the Quillen model category 
$$\mathsf{L}_{\widetilde{\mathcal{E}^s_{un}}} \mathsf{L}_{\Sigma,P} \mathsf{Fun}(\mathsf{dgcat}_f^o,Sset_{\bullet})\,.$$
\item[-] The {\em Universal unstable invariant} of dg categories is the canonical morphism of
  derivators $$ \mathcal{U}_u : \mathsf{HO}(\mathsf{dgcat}) \rightarrow \mathcal{M}^{unst}_{dg}\,.$$
\end{itemize}
\end{definition}

Let $\mathcal{M}$ be a left proper cellular model category, $S$ a
set of maps in $\mathcal{M}$ and $\mathsf{L}_S\mathcal{M}$ the left
Bousfield localization of $\mathcal{M}$ with respect to $S$, see \cite{Hirschhorn}. Recall
from \cite[4.1.1.]{Hirschhorn} that an object $X$ in
$\mathsf{L}_S\mathcal{M}$ is fibrant if $X$ is fibrant in
$\mathcal{M}$ and for every element $f:A \rightarrow B$ of $S$ the
induced map of homotopy function complexes $f^{\ast}:\mathsf{Map}(B,X)
\rightarrow \mathsf{Map}(A,X)$ is a weak equivalence.
\begin{proposition}\label{fib1}
An object $F \in \mathsf{L}_{\widetilde{\mathcal{E}^s_{un}}}
\mathsf{L}_{\Sigma,P}\mathsf{Fun}(\mathsf{dgcat}_f^{op},Sset_{\bullet})$
is fibrant if and only if it satisfies the following conditions
\begin{itemize}
\item[1)] $F(\mathcal{B}) \in Sset_{\bullet}$ is fibrant, for all
  $\mathcal{B} \in \mathsf{dgcat}_f$.
\item[2)] $F(\emptyset) \in Sset_{\bullet}$ is contractible.
\item[3)] For every Morita equivalence $\mathcal{B}
  \stackrel{\sim}{\rightarrow} \mathcal{B}'$ in $\mathsf{dgcat}_f$ the
  morphism $F(\mathcal{B}') \stackrel{\sim}{\rightarrow}
  F(\mathcal{B})$ is a weak equivalence in $Sset_{\bullet}$.
\item[4)] Every split short exact sequence
$$ 
\xymatrix{
0 \ar[r] & \mathcal{B}' \ar[r]_{i_{\mathcal{B}'}} & \mathcal{B} \ar@<-1ex>[l]_R
\ar[r]_P & \mathcal{B}'' \ar@<-1ex>[l]_{i_{\mathcal{B}''}} \ar[r] & 0 
}
$$
in $\mathsf{dgcat}_f$ induces a homotopy fiber sequence
$$ F(\mathcal{B}'') \rightarrow F(\mathcal{B}) \rightarrow
F(\mathcal{B}')$$
in $Sset$.
\end{itemize}
\end{proposition}

\begin{proof}
Clearly condition $1)$ corresponds to the fact that $F$ is fibrant in
$\mathsf{Fun}(\mathsf{dgcat}_f^{op},Sset_{\bullet})$. Now observe that $\mathsf{Fun}(\mathsf{dgcat}_f^{op},Sset_{\bullet})$ is a
simplicial Quillen model category with the simplicial action given by
$$
\begin{array}{ccc}
Sset \times \mathsf{Fun}(\mathsf{dgcat}_f^{op},Sset_{\bullet}) &
\rightarrow & \mathsf{Fun}(\mathsf{dgcat}_f^{op},Sset_{\bullet}) \\
(K,F) & \mapsto & K_+\wedge F\,,
\end{array}
$$
where $K_+ \wedge F$ denotes the componentwise smash product. This simplicial structure and the construction of the
localized Quillen model category \newline
$\mathsf{L}_{\widetilde{\mathcal{E}^s_{un}}}
\mathsf{L}_{\Sigma,P}\mathsf{Fun}(\mathsf{dgcat}_f^{op},Sset_{\bullet})$,
see section~\ref{small}, allow us to recover conditions $2)$ and
$3)$. Condition $4)$ follows from the construction of the set
$\widetilde{\mathcal{E}^s_{un}}$ and from the fact that the functor
$$ \mathsf{Map}(?,F) :
\mathsf{Fun}(\mathsf{dgcat}_f^{op},Sset_{\bullet})^{op} \rightarrow
Sset$$
transforms homotopy cofiber sequences into homotopy fiber sequences.

This proves the proposition.
\end{proof}
Let $\mathcal{A}$ be a Morita fibrant dg category. Recall from notation~\ref{simplicialK} that $S_{\bullet}\ca$ denotes the simplicial Morita fibrant dg category obtained by applying Waldhausen's $S_{\bullet}$-construction to the exact category $\mathsf{Z}^0(\ca)$ and remembering the enrichment in complexes.

\begin{notation}
We denote by $K(\mathcal{A}) \in \mathsf{Fun}(\mathsf{dgcat}_f^{op},Sset_{\bullet})$ the simplicial presheaf
$$
\begin{array}{rcl}
\mathcal{B} & \mapsto &
|N.wS_{\bullet}\mathsf{rep}_{mor}(\mathcal{B},\mathcal{A})|\,,
\end{array}
$$
where $|-|$ denotes the fibrant realization functor of bisimplicial sets.
\end{notation}

\begin{proposition}\label{fibrant}
The simplicial presheaf $K(\mathcal{A})$ is fibrant in $\mathsf{L}_{\widetilde{\mathcal{E}^s_{un}}} \mathsf{L}_{\Sigma,P}\mathsf{Fun}(\mathsf{dgcat}_f^{op},Sset_{\bullet})$.
\end{proposition}

\begin{proof}
Observe that $K(\mathcal{A})$ satisfies conditions $(1)$-$(3)$ of proposition~\ref{fib1}. We now
prove that Waldhausen's fibration theorem \cite[1.6.4]{Wald} implies condition
$(4)$. Apply the contravariant functor
$\mathsf{rep}_{mor}(?,\mathcal{A})$ to the split short exact sequence
$$ 
\xymatrix{
0 \ar[r] & \mathcal{B}' \ar[r]_{i_{\mathcal{B}'}} & \mathcal{B} \ar@<-1ex>[l]_R
\ar[r]_P & \mathcal{B}'' \ar@<-1ex>[l]_{i_{\mathcal{B}''}} \ar[r] & 0 
}
$$
and obtain a split short exact sequence
$$ 
\xymatrix{
0 \ar[r] & \mathsf{rep}_{mor}(\mathcal{B}'',\mathcal{A}) \ar[r] & \mathsf{rep}_{mor}(\mathcal{B},\mathcal{A}) \ar@<-1ex>[l]
\ar[r] & \mathsf{rep}_{mor}(\mathcal{B}',\mathcal{A}) \ar@<-1ex>[l] \ar[r] & 0 \,.
}
$$
Now consider the Waldhausen category
$v\mathsf{rep}_{mor}(\mathcal{B},\mathcal{A}):=\mathsf{Z}^0(\mathsf{rep}_{mor}(\mathcal{B},\mathcal{A}))$, where the weak equivalences
are the morphisms $f$ such that $\mathsf{cone}(f)$ is
contractible. Consider also the Waldhausen category $w\mathsf{rep}_{mor}(\mathcal{B},\mathcal{A})$,
which has the same cofibrations as $v\mathsf{rep}_{mor}(\mathcal{B},\mathcal{A})$ but the weak
equivalences are the morphisms $f$ such that $\mathsf{cone}(f)$
belongs to $\mathsf{rep}_{mor}(\mathcal{B}'',\mathcal{A})$. Observe that we have the inclusion
$v\mathsf{rep}_{mor}(\mathcal{B},\mathcal{A}) \subset w\mathsf{rep}_{mor}(\mathcal{B},\mathcal{A})$ and an equivalence $\mathsf{rep}_{mor}(\mathcal{B},\mathcal{A})^w
\simeq \mathsf{Z}^0(\mathsf{rep}_{mor}(\mathcal{B}',\mathcal{A}))$, see section $1.6$ from
\cite{Wald}. The conditions of theorem $1.6.4$ from \cite{Wald}
are satisfied and so we have a homotopy fiber sequence
$$
|N.wS_{\bullet}\mathsf{rep}_{mor}(\mathcal{B}'',\mathcal{A})| \rightarrow
|N.wS_{\bullet}\mathsf{rep}_{mor}(\mathcal{B},\mathcal{A})| \rightarrow |N.wS_{\bullet} \mathsf{rep}_{mor}(\mathcal{B}',\mathcal{A})| $$
in $Sset$. This proves the proposition.
\end{proof}

Let $p:\Delta \rightarrow e$ be the projection functor.

\begin{proposition}\label{clef}
The objects
$$ 
\begin{array}{ccc}
S^1 \wedge N.w\mathsf{rep}_{mor}(?,\mathcal{A}) & \mbox{and} &
|N.wS_{\bullet} \mathsf{rep}_{mor}(?,\mathcal{A})|=K(\mathcal{A})
\end{array}
$$
are canonically isomorphic in $\mathsf{Ho}(\mathsf{L}_{\widetilde{\mathcal{E}^s_{un}}} \mathsf{L}_{\Sigma,P}\mathsf{Fun}(\mathsf{dgcat}_f^{op},Sset_{\bullet}))$.
\end{proposition}

\begin{proof}
As in \cite[3.3]{MacCarthy}, we consider the sequence in $\mathsf{HO}(\mathsf{dgcat})(\Delta)$
$$ \xymatrix{0 \ar[r] & \mathcal{A}_{\bullet} \ar[r]^-I &
PS_{\bullet}\mathcal{A} \ar[r]^Q &  S_{\bullet}\mathcal{A} \ar[r]& 0} \,,
$$
where $\mathcal{A}_{\bullet}$ denotes the constant simplicial dg
category with value $\mathcal{A}$ and $PS_{\bullet}\mathcal{A}$ the
path object of $S_{\bullet}\mathcal{A}$. By applying the morphism of
derivators $\mathcal{U}_u$ to this sequence, we obtain the canonical
morphism
$$ S_I: \mbox{cone} (\mathcal{U}_u(\mathcal{A}_{\bullet}
\stackrel{I}{\rightarrow} PS_{\bullet}\mathcal{A})) \rightarrow
\mathcal{U}_u(S_{\bullet}\mathcal{A})$$
in $\mathcal{M}^{unst}_{dg}(\Delta)$. We now prove that for each point
$n:e \rightarrow \Delta$, the $n$th component of $S_I$ is an
isomorphism in $\mathsf{L}_{\widetilde{\mathcal{E}^s_{un}}}
\mathsf{L}_{\Sigma,P}
\mathsf{Fun}(\mathsf{dgcat}_f^{op},Sset_{\bullet})$. For each point $n:e
\rightarrow \Delta$, we have a split short exact sequence in
$\mathsf{Ho}(\mathsf{dgcat})$:
$$
\xymatrix{
0 \ar[r] & \mathcal{A} \ar[r]_-{I_n} & PS_n\mathcal{A}=S_{n+1}\mathcal{A} \ar@<-1ex>[l]_-{R_n}
\ar[r]_-{Q_n} & S_n\mathcal{A} \ar@<-1ex>[l]_-{S_n} \ar[r] & 0 \,,
}
$$
where $I_n$ maps $A \in \mathcal{A}$ to the constant sequence
$$
0 \rightarrow A \stackrel{Id}{\rightarrow} A \stackrel{Id}{\rightarrow}
\cdots \stackrel{Id}{\rightarrow} A \,,$$
$Q$ maps a sequence
$$ 0 \rightarrow A_0 \rightarrow A_1 \rightarrow \cdots \rightarrow
A_n$$
to
$$ A_1/A_0 \rightarrow \cdots \rightarrow A_n/A_0\,,$$
$S_n$ maps a sequence
$$ 0 \rightarrow A_0 \rightarrow A_1 \rightarrow \cdots \rightarrow
A_{n-1}$$
to
$$ 0 \rightarrow 0 \rightarrow A_0 \rightarrow \cdots \rightarrow
A_{n-1}$$
and $R_n$ maps a sequence
$$ 0 \rightarrow A_0 \rightarrow A_1 \rightarrow \cdots \rightarrow
A_{n-1}$$
to $A_0$. Now, by construction of
$\mathsf{L}_{\widetilde{\mathcal{E}^s_{un}}} \mathsf{L}_{\Sigma,P}
\mathsf{Fun}(\mathsf{dgcat}_f^{op},Sset_{\bullet})$ the canonical
morphisms
$$ S_{I_n}: \mbox{cone}(\mathcal{U}_u (\mathcal{A}
\stackrel{I_n}{\rightarrow} PS_n\mathcal{A})) \rightarrow
\mathcal{U}_u(S_n\mathcal{A}),\,\, n \in \mathbb{N} $$
are isomorphisms in $\mathcal{M}_{dg}^{unst}(e)$. Since homotopy
colimits in $\mathsf{L}_{\widetilde{\mathcal{E}^s_{un}}} \mathsf{L}_{\Sigma,P}
\mathsf{Fun}(\mathsf{dgcat}_f^{op},Sset_{\bullet})$ are calculated
objectwise the $n$th component of $S_I$ identifies with $S_{I_n}$ and
so by the conservativity axiom $S_I$ is an isomorphism in
$\mathcal{M}_{dg}^{unst}(\Delta)$. This implies that we obtain the
homotopy cocartesian square
$$
\xymatrix{
p_!\mathcal{U}_u(\mathcal{A}_{\bullet}) \ar[d] \ar[r] &
p_!(\mathcal{U}_u(PS_{\bullet}\mathcal{A})) \ar[d] \\
\ast \ar[r] & p_!(\mathcal{U}_u(S_{\bullet}\mathcal{A}))\,.
}
$$
As in the proof of proposition~\ref{real}, we show that
$p_!\mathcal{U}_u(\mathcal{A}_{\bullet})$ identifies with $N.w
\mathsf{rep}_{mor}(?,\mathcal{A})=\mathcal{U}_u(\mathcal{A})$ and that
$p_!(\mathcal{U}_u(PS_{\bullet}\mathcal{A}))$ is contractible. Since
we have the equivalence
$$p_!(\mathcal{U}_u(S_{\bullet}\mathcal{A})) = p_!(N.w
\mathsf{rep}_{mor}(?,S_{\bullet}\mathcal{A}))
\stackrel{\sim}{\rightarrow} |N.wS_{\bullet} \mathsf{rep}_{mor}(?,\mathcal{A})|$$
and $N.w\mathsf{rep}_{mor}(?,\mathcal{A})$ is cofibrant in  $\mathsf{L}_{\widetilde{\mathcal{E}^s_{un}}} \mathsf{L}_{\Sigma,P}
\mathsf{Fun}(\mathsf{dgcat}_f^{op},Sset_{\bullet})$ the proposition is proven.
\end{proof}

\begin{proposition}\label{k-theoryunst}
We have the following weak equivalence of simplicial sets
$$ \mathsf{Map}(\mathcal{U}_u(k),S^1 \wedge \mathcal{U}_u(\mathcal{A}))
\stackrel{\sim}{\rightarrow} |N.wS_{\bullet}\mathcal{A}_f|$$
in $\mathsf{L}_{\widetilde{\mathcal{E}^s_{un}}} \mathsf{L}_{\Sigma,P}
\mathsf{Fun}(\mathsf{dgcat}_f^{op}, Sset_{\bullet})$. In particular, we
have the following isomorphisms
$$\pi_{i+1}
\mathsf{Map}(\mathcal{U}_u(k),S^1 \wedge \mathcal{U}_u(\mathcal{A}))
\stackrel{\sim}{\rightarrow} K_i(\mathcal{A}), \, \,\, \forall i \geq
0\,.$$
\end{proposition}

\begin{proof}
This follows from propositions \ref{fibrant}, \ref{clef} and from the
following weak equivalences
$$
\begin{array}{rcl}
\mathsf{Map}(\mathcal{U}_u(k),S^1 \wedge \mathcal{U}_u(\mathcal{A})) & \simeq
& \mathsf{Map}(\mathbb{R}\underline{h}(k), K(\mathcal{A}))\\
 & \simeq & (K(\mathcal{A}))(k)\\
 & \simeq & |N.wS_{\bullet} \mathcal{A}_f|\,.
\end{array}
$$
\end{proof}
\section{The universal additive invariant}\label{univadit}
Consider the Quillen model category
$\mathsf{L}_{\widetilde{\mathcal{E}^s_{un}}} \mathsf{L}_{\Sigma,P}
\mathsf{Fun}(\mathsf{dgcat}_f^{op}, Sset_{\bullet})$ constructed in the
previous section. The definition of the
  set $\widetilde{\mathcal{E}^s_{un}}$ and the same arguments as those
  of example~\ref{exem} and example~\ref{ex2} allows us to conclude
  that $\mathsf{L}_{\widetilde{\mathcal{E}^s_{un}}} \mathsf{L}_{\Sigma,P}
\mathsf{Fun}(\mathsf{dgcat}_f^{op},Sset_{\bullet})$ satisfies the
  conditions of theorem~\ref{repre}. In particular we have an
  equivalence of triangulated derivators
$$ \mathsf{St}(\mathcal{M}_{dg}^{unst}) \stackrel{\sim}{\rightarrow}
\mathsf{HO}(\mathsf{Sp}^{\mathbb{N}}(\mathsf{L}_{\widetilde{\mathcal{E}^s_{un}}}
\mathsf{L}_{\Sigma,P}
\mathsf{Fun}(\mathsf{dgcat}_f^{op}, Sset_{\bullet})))\,.
$$

\begin{definition}
\begin{itemize}
\item[-] The {\em Additive motivator} of dg categories $\mathcal{M}^{add}_{dg}$ is the triangulated
derivator associated with the stable Quillen model category 
$$ \mathsf{Sp}^{\mathbb{N}}(\mathsf{L}_{\widetilde{\mathcal{E}^s_{{un}}}}\mathsf{L}_{\Sigma,P}
\mathsf{Fun}(\mathsf{dgcat}_f^{op},Sset_{\bullet}))\,.$$
\item[-] The {\em Universal additive invariant} of dg categories is the canonical morphism of
  derivators $$ \mathcal{U}_a : \mathsf{HO}(\mathsf{dgcat}) \rightarrow \mathcal{M}^{add}_{dg}\,.$$
\end{itemize}
\end{definition}

\begin{remark}
Observe that remark~\ref{fib} and remark~\ref{important} imply that
$\mathcal{M}^{add}_{dg}$ is a compactly generated triangulated derivator.
\end{remark}

We sum up the construction of $\mathcal{M}^{add}_{dg}$ in the following diagram

$$
\xymatrix{
\underline{\mathsf{dgcat}_f}[S^{-1}] \ar[r] \ar[d]_{\mathsf{Ho}(h)} &
\mathsf{HO}(\mathsf{dgcat}) \ar[dl]^{\mathbb{R}\underline{h}}
\ar@/^3pc/[ddddl]^{\mathcal{U}_a}  \ar@/^2pc/[dddl]_{\mathcal{U}_u}
 \\
\mathsf{L}_{\Sigma}\mathsf{Hot}_{\mathsf{dgcat}_f}
\ar[d]_{\Phi}  \ar@<1ex>[ur]^{\mathbb{L}Re}  & \\
{\mathsf{L}_{\Sigma,P}\mathsf{Hot}_{\mathsf{dgcat}_f}}_{\bullet}
\ar[d]_{\Psi} & \\
\mathcal{M}_{dg}^{unst}
\ar[d]_{\varphi} & \\
\mathcal{M}^{add}_{dg} & 
}
$$
Observe that the morphism of derivators $\mathcal{U}_a$ is pointed, commutes with filtered
homotopy colimits and satisfies the following condition :

\begin{itemize}
\item[A)] For every split short exact sequence
$$ 
\xymatrix{
0 \ar[r] & \mathcal{A} \ar[r]_{i_{\mathcal{A}}} & \mathcal{B} \ar@<-1ex>[l]_R
\ar[r]_P & \mathcal{C} \ar@<-1ex>[l]_{i_{\mathcal{C}}} \ar[r] & 0 
}
$$
in $\mathsf{Ho}(\mathsf{dgcat})$, we have a split triangle
$$  
\xymatrix{ 
\mathcal{U}_a(\mathcal{A}) \ar[r]_{i_{\mathcal{A}}} & \mathcal{U}_a(\mathcal{B}) \ar@<-1ex>[l]_R
\ar[r]_P & \mathcal{U}_a(\mathcal{C}) \ar@<-1ex>[l]_{i_{\mathcal{C}}} \ar[r] & \mathcal{U}_a(\mathcal{A})[1] 
}
$$
in $\mathcal{M}^{add}_{dg}(e)$. (This implies that the dg functors
$i_{\mathcal{A}}$ and $i_{\mathcal{C}}$ induce an isomorphism
$$ \mathcal{U}_a(\mathcal{A}) \oplus \mathcal{U}_a(\mathcal{C})
\stackrel{\sim}{\rightarrow} \mathcal{U}_a(\mathcal{B})$$
in $\mathcal{M}^{add}_{dg}(e)$).
\end{itemize}

\begin{remark}
Since the dg category $\mathcal{B}$ in the above split short exact
sequence is Morita equivalent to the dg category
$E(\mathcal{A},\mathcal{B},\mathcal{C})$, see section $1.1$ from
\cite{Wald}, condition $A)$ is equivalent to the additivity property
stated by Waldhausen in \cite[1.3.2]{Wald}.
\end{remark}

Let $\mathbb{D}$ be a strong triangulated derivator.
\begin{theorem}\label{principal1}
The morphism $\mathcal{U}_a$ induces an equivalence of categories
$$
\underline{\mathsf{Hom}}_!(\mathcal{M}^{add}_{dg},
\mathbb{D}) \stackrel{\mathcal{U}_a^{\ast}}{\longrightarrow}
\underline{\mathsf{Hom}}_{flt,\,A),\,p}(\mathsf{HO}(\mathsf{dgcat}),\mathbb{D})\,,$$
where
$\underline{\mathsf{Hom}}_{flt,\,A)\,, p}(\mathsf{HO}(\mathsf{dgcat}),\mathbb{D})$
denotes the category of morphisms of derivators which commute with filtered
homotopy colimits, satisfy condition $A)$ and preserve the point.
\end{theorem}

\begin{proof}
By theorem~\ref{HellerT}, we have an equivalence of categories
$$
\underline{\mathsf{Hom}}_!(\mathcal{M}^{add}_{dg},
\mathbb{D}) \stackrel{\sim}{\longrightarrow}
\underline{\mathsf{Hom}}_!(\mathcal{M}_{dg}^{unst},\mathbb{D})\,.$$
By theorem~\ref{Cisinsk}, we have an equivalence of categories
$$  \underline{\mathsf{Hom}}_!(\mathcal{M}_{dg}^{unst},\mathbb{D})
\stackrel{\sim}{\longrightarrow} \underline{\mathsf{Hom}}_{!,\mathcal{E}^s_{un}}({\mathsf{L}_{\Sigma,P}
\mathsf{Hot}_{\mathsf{dgcat}_f}}_{\bullet}, \mathbb{D})\,.$$
Now, we observe that since $\mathbb{D}$ is a strong triangulated derivator, the category  $\underline{\mathsf{Hom}}_{!,\mathcal{E}^s_{un}}({\mathsf{L}_{\Sigma,P}
\mathsf{Hot}_{\mathsf{dgcat}_f}}_{\bullet}, \mathbb{D})$ identifies
$\underline{\mathsf{Hom}}_{flt,\,A),\,p}(\mathsf{HO}(\mathsf{dgcat}),\mathbb{D})$. 
This proves the theorem. 
\end{proof}

\begin{notation}
We call an object of the right hand side category of
theorem~\ref{principal1} an {\em additive invariant} of dg categories.
\end{notation}

\begin{example}
\begin{itemize}
\item[-] The Hochschild and cyclic homology and the non-connective
  $K$-theory defined in section~\ref{labuniv} are examples of additive
  invariants.
\item[-] Another example is given by the classical Waldhausen's connective
  $K$-theory spectrum
$$K^c: \mathsf{HO}(\mathsf{dgcat}) \rightarrow \mathsf{HO}(Spt)\,,$$
see \cite{Wald}.
\end{itemize}
\end{example}

\begin{remark}
By theorem~\ref{principal1}, the morphism of derivators $K^c$ factors
through $\mathcal{U}_a$ and so gives rise to a morphism of derivators
$$ K^c:\mathcal{M}^{add}_{dg} \rightarrow \mathsf{HO}(Spt)\,.$$
\end{remark}
We now will prove that this morphism of derivators is co-representable in
$\mathcal{M}^{add}_{dg}$.

Let $\mathcal{A}$ be a small dg category.
\begin{notation}\label{not6}
We denote by $K(\mathcal{A})^c \in \mathsf{Sp}^{\mathbb{N}}(\mathsf{L}_{\widetilde{\mathcal{E}^s_{un}}} \mathsf{L}_{\Sigma,P}
\mathsf{Fun}(\mathsf{dgcat}_f^{op}, Sset_{\bullet}))$ the spectrum such
that
$$ K(\mathcal{A})^c_n:=
|N.wS_{\bullet}^{(n+1)}\mathsf{rep}_{mor}(?,\mathcal{A})|, \, n\geq 0\,,$$
endowed with the natural structure morphisms
$$ \beta_n: S^1\wedge |N.wS_{\bullet}^{(n+1)}\rep_{mor}(?,\ca)| \stackrel{\sim}{\longrightarrow} |N.wS_{\bullet}^{(n+2)}\rep_{mor}(?,\ca)|, \,\, n \geq 0,$$
see \cite{Wald}.
\end{notation}

Notice that $\mathcal{U}_a(\mathcal{A})$ identifies in
$\mathsf{Ho}(\mathcal{M}^{add}_{dg})$ with the suspension spectrum given
by 
$$ (\Sigma^{\infty}|N.w\mathsf{rep}_{mor}(?,\mathcal{A})|)_n := S^n
\wedge |N.w\mathsf{rep}_{mor}(?,\mathcal{A})|\,.$$
Now proposition~\ref{clef} and the fact that the morphism of derivators $\varphi$ commutes with homotopy
colimits implies that we have an isomorphism
$$\mathcal{U}_a(\mathcal{A})[1] \stackrel{\sim}{\rightarrow}
p_!\mathcal{U}_a(S_{\bullet}\mathcal{A})\,.$$
In particular, we have a natural morphism
$$\eta : \mathcal{U}_a(\mathcal{A})[1] \rightarrow
K(\mathcal{A})^c$$
in $\mathsf{Sp}^{\mathbb{N}}(\mathsf{L}_{\widetilde{\mathcal{E}^s_{un}}} \mathsf{L}_{\Sigma,P}
\mathsf{Fun}(\mathsf{dgcat}_f^{op}, Sset_{\bullet}))$ induced by the
identity in degree $0$.

\begin{theorem}\label{fibres}
The morphism $\eta$ is a fibrant resolution of $\mathcal{U}_a(\mathcal{A})[1]$.
\end{theorem}

\begin{proof}
We prove first that $K(\mathcal{A})^c$ is a fibrant object in $\mathsf{Sp}^{\mathbb{N}}(\mathsf{L}_{\widetilde{\mathcal{E}^s_{un}}} \mathsf{L}_{\Sigma,P}
\mathsf{Fun}(\mathsf{dgcat}_f^{op}, Sset_{\bullet}))$. By \cite{Hovey}\cite{Schwede} we need to show that $K(\mathcal{A})^c$ is an
$\Omega$-spectrum, i.e. that $K(\mathcal{A})^c_n$ is a fibrant object in $\mathsf{L}_{\widetilde{\mathcal{E}^s_{un}}} \mathsf{L}_{\Sigma,P}
\mathsf{Fun}(\mathsf{dgcat}_f^{op}, Sset_{\bullet})$ and that the induced
map
$$ K(\mathcal{A})^c_n \rightarrow \Omega K(\mathcal{A})^c_{n+1}$$ 
is a weak equivalence. By Waldhausen's
additivity theorem, see \cite{Wald}, we have weak equivalences
$$ K(\mathcal{A})^c_n \stackrel{\sim}{\rightarrow}
\Omega K(\mathcal{A})^c_{n+1}$$
in $\mathsf{Fun}(\mathsf{dgcat}_f^{op},Sset_{\bullet})$. Now observe that
for every integer $n$, $K^c(\mathcal{A})_n$ satisfies conditions
$(1)$-$(3)$ of proposition~\ref{fibrant}. Condition $(4)$ follows from
Waldhausen's fibration theorem, as in the proof of
proposition~\ref{clef}, applied to the $S.$-construction. This shows
that $K(\mathcal{A})^c$ is an $\Omega$-spectrum.

We now prove that $\eta$ is a (componentwise) weak equivalence in \newline
$\mathsf{Sp}^{\mathbb{N}}(\mathsf{L}_{\widetilde{\mathcal{E}^s_{un}}} \mathsf{L}_{\Sigma,P}
\mathsf{Fun}(\mathsf{dgcat}_f^{op}, Sset_{\bullet}))$. For this, we prove
first that the structural morphisms
$$ \beta_n: S^1 \wedge N.wS_{\bullet}^{(n+1)}\mathsf{rep}_{mor}(?,\mathcal{A})
\stackrel{\sim}{\longrightarrow}
|N.wS_{\bullet}^{(n+2)}\mathsf{rep}_{mor}(?,\mathcal{A})|,\, n\geq
0\,,$$ 
see notation~\ref{not6}, are weak equivalences in $\mathsf{L}_{\widetilde{\mathcal{E}^s_{un}}} \mathsf{L}_{\Sigma,P}
\mathsf{Fun}(\mathsf{dgcat}_f^{op}, Sset_{\bullet})$. By considering the same argument as in the proof of
proposition~\ref{clef}, using $S_{\bullet}^{(n+1)}\mathcal{A}$ instead
of $\mathcal{A}$, we obtain the following homotopy cocartesian square
$$
\xymatrix{
K(\mathcal{A})^c_n \ar[d] \ar[r] &
p_!(\mathcal{U}_u(PS_{\bullet}^{(n+2)}\mathcal{A})) \ar[d] \\
\ast \ar[r] & K(\mathcal{A})^c_{n+1}
}
$$
in $\mathsf{L}_{\widetilde{\mathcal{E}^s_{un}}} \mathsf{L}_{\Sigma,P}
\mathsf{Fun}(\mathsf{dgcat}_f^{op}, Sset_{\bullet})$ with
$p_!(\mathcal{U}_u(PS_{\bullet}^{(n+2)}\mathcal{A}))$
contractible. Since $K(\mathcal{A})^c_{n+1}$ is fibrant, proposition
$1.5.3$ in \cite{Wald}, implies that the previous square is also
homotopy cartesian and so the canonical morphism
$$ \beta_n^{\sharp}: K(\mathcal{A})^c_n \rightarrow
\Omega K(\mathcal{A})^c_{n+1}$$
is a weak equivalence in $\mathsf{L}_{\widetilde{\mathcal{E}^s_{un}}} \mathsf{L}_{\Sigma,P}
\mathsf{Fun}(\mathsf{dgcat}_f^{op}, Sset_{\bullet})$. We now show that
the structure morphism $\beta_n$, which corresponds to
$\beta_n^{\sharp}$ by adjunction, see \cite{Wald}, is also a
weak equivalence. The derived adjunction $(S^1\wedge-, \mathbb{R} \Omega(-))$
induces the following commutative diagram
$$
\xymatrix{
S^1 \wedge K(\mathcal{A})^c_n \ar[dr]_{S^1 \wedge^{\mathbb{L}}
  \beta_n^{\sharp}} \ar[r] &
K(\mathcal{A})^c_{n+1} \\
 & S^1 \wedge^{\mathbb{L}} \Omega K(\mathcal{A})_{n+1}^c \ar[u]^{\sim}
}
$$
in $\mathsf{Ho}(\mathsf{L}_{\widetilde{\mathcal{E}^s_{un}}} \mathsf{L}_{\Sigma,P}
\mathsf{Fun}(\mathsf{dgcat}_f^{op}, Sset_{\bullet}))$ where the vertical
arrow is an isomorphism since the previous square is homotopy
bicartesian. This shows that the induced morphism
$$S^1 \wedge K(\mathcal{A})_n^c  \longrightarrow
K(\mathcal{A})^c_{n+1}$$
is an isomorphism in $\mathsf{Ho}(\mathsf{L}_{\widetilde{\mathcal{E}^s_{un}}} \mathsf{L}_{\Sigma,P}
\mathsf{Fun}(\mathsf{dgcat}_f^{op}, Sset_{\bullet}))$ and so $\beta_n$ is
a weak equivalence.

Now to prove that $\eta$ is a componentwise weak equivalence, we
proceed by induction~: observe that the zero component of the morphism $\eta$ is the
identity. Now suppose that the $n$-component of $\eta$ is a weak
equivalence. The $n+1$-component of $\eta$ is the composition
of $\beta_{n+1}$, which is a weak equivalence, with the suspension of the
$n$-component of $\eta$, which by
proposition~\ref{Cisin} is also a weak equivalence.

This proves the theorem.
\end{proof}

Let $\mathcal{A}$ and $\mathcal{B}$ be small dg categories with $\ca \in \dgcat_f$.
We denote by
$\mathsf{Hom}^{\mathsf{Sp}^{\mathbb{N}}}(-,-)$ the
spectrum of morphisms in $\mathsf{Sp}^{\mathbb{N}}(\mathsf{L}_{\widetilde{\mathcal{E}^s_{un}}} \mathsf{L}_{\Sigma,P}
\mathsf{Fun}(\mathsf{dgcat}_f^{op}, Sset_{\bullet}))$.

\begin{theorem}\label{corep}
We have the following weak equivalence of spectra
$$ \mathsf{Hom}^{\mathsf{Sp}^{\mathbb{N}}}(\mathcal{U}_a(\mathcal{A}),\mathcal{U}_a(\mathcal{B})[1])
\stackrel{\sim}{\rightarrow}
K^c(\mathsf{rep}_{mor}(\mathcal{A},\mathcal{B}))\,,$$
where $K^c(\mathsf{rep}_{mor}(\mathcal{A},\mathcal{B}))$ denotes Waldhausen's connective $K$-theory spectrum of
$\mathsf{rep}_{mor}(\mathcal{A},\mathcal{B})$.

In particular, we
have the following weak equivalence of simplicial sets
$$\mathsf{Map}(\mathcal{U}_a(\mathcal{A}),\mathcal{U}_a(\mathcal{B})[1])
\stackrel{\sim}{\rightarrow}
|N.wS_{\bullet}\mathsf{rep}_{mor}(\mathcal{A},\mathcal{B})|$$
and so the isomorphisms
$$\pi_{i+1}
\mathsf{Map}(\mathcal{U}_a(A),\mathcal{U}_a(\mathcal{B})[1])
\stackrel{\sim}{\rightarrow} K_i(\mathsf{rep}_{mor}(\mathcal{A},\mathcal{B})), \,\,\, \forall i \geq 0\,.$$
\end{theorem}

\begin{proof}
Notice that $\mathcal{U}_a(\mathcal{A})$ identifies with the suspension spectrum
$$\Sigma^{\infty}|N.w\mathsf{rep}_{mor}(?,\mathcal{A})|$$
which is cofibrant in  $\mathsf{Sp}^{\mathbb{N}}(\mathsf{L}_{\widetilde{\mathcal{E}^s_{un}}} \mathsf{L}_{\Sigma,P}
\mathsf{Fun}(\mathsf{dgcat}_f^o, Sset_{\bullet}))$. By
theorem~\ref{fibres} we have the following equivalences
$$
\begin{array}{rcl}
\mathsf{Hom}^{\mathsf{Sp}^{\mathbb{N}}}(\mathcal{U}_a(\mathcal{A}),\mathcal{U}_a(\mathcal{B})[1]) & \simeq
& \mathsf{Hom}^{\mathsf{Sp}^{\mathbb{N}}}(\mathcal{U}_a(\mathcal{A}), K^c(\mathcal{B}))\\
 & \simeq & K^c(\mathcal{B})(\mathcal{A})\\
 & \simeq & K^c(\mathsf{rep}_{mor}(\mathcal{A},\mathcal{B}))\,.
\end{array}
$$
This proves the theorem.
\end{proof}

\begin{remark}
Notice that if in the above theorem, we consider $\mathcal{A}=k$, we
have
$$ \mathsf{Hom}^{\mathsf{Sp}^{\mathbb{N}}}(\mathcal{U}_a(k),\mathcal{U}_a(\mathcal{B})[1])
\stackrel{\sim}{\rightarrow}
K^c(\mathcal{B})\,.$$
This shows that Waldhausen's connective $K$-theory spectrum becomes
co-representable in $\mathcal{M}_{dg}^{add}$. To the best of the author's knowledge, this is the first
conceptual characterization of Quillen-Waldhausen $K$-theory
\cite{Quillen1} \cite{Wald} since its definition in the early $70$'s. This result gives us a completely new way to think about algebraic $K$-theory.
\end{remark}

\section{Higher Chern characters}\label{Chern}
In this chapter we apply our main co-representability theorem~\ref{corep} in the construction of the higher Chern characters \cite{Loday}.

Let $\ca$ and $\cb$ be small dg categories with $\ca \in \mathsf{dgcat}_f$.

\begin{proposition}\label{true}
We have the following isomorphisms of abelian groups
$$ \mathsf{Hom}_{\cm^{add}_{dg}(e)} (\cu_a(\ca), \cu_a(\cb)[-n]) \stackrel{\sim}{\longrightarrow} K_n(\mathsf{rep}_{mor}(\ca, \cb)),\,\, \forall n \geq 0\,.$$
\end{proposition}
\begin{proof}
In first place, notice that the abelian group
$$ \mathsf{Hom}_{\cm^{add}_{dg}(e)} (\cu_a(\ca), \cu_a(\cb)[-n])$$
identifies with
$$ \pi_0 \mathsf{Map}(\cu_a(\ca), \cu_a(\cb)[-n]),$$
where $\mathsf{Map}$ denotes the mapping space in  $\mathsf{Sp}^{\mathbb{N}}(\mathsf{L}_{\widetilde{\mathcal{E}^s_{un}}} \mathsf{L}_{\Sigma,P}\mathsf{Fun}(\mathsf{dgcat}_f^{op}, Sset_{\bullet}))$. By theorem~\ref{fibres} the morphism
$$ \eta: \cu_a(\cb)[1] \longrightarrow K(\cb)^c$$
is a fibrant resolution of $\cu_a(\cb)[1]$. This implies that in $\cm_{dg}^{add}(e)$, $\cu_a(\cb)[-n]$ identifies with the spectrum $\Omega^{n+1}(K^c(\cb))$. Since $\cu_a(\ca)$ is cofibrant and
$$ \Omega^{n+1}(K^c(\cb))_0= \Omega^{n+1}|N.w S_{\bullet}\mathsf{rep}_{mor}(?,\cb)|$$
we conclude that
$$ \pi_0 \mathsf{Map}(\cu_a(\ca), \cu_a(\cb)[-n]) \simeq \pi_0 \Omega^{n+1}|N.wS_{\bullet}\mathsf{rep}_{mor}(\ca,\cb)|\,.$$
Finally notice that
$$
\begin{array}{rcl} \pi_0 \Omega^{n+1}|N.wS_{\bullet}\mathsf{rep}_{mor}(\ca,\cb)| & \simeq & \pi_{n+1} |N.wS_{\bullet}\mathsf{rep}_{mor}(\ca,\cb)|\\
& \simeq & K_n(\mathsf{rep}_{mor}(\ca,\cb))\,.
\end{array}
$$
This proves the proposition.
\end{proof}
\begin{remark}\label{true1}
Notice that if in the above proposition, we consider $\ca=k$, we have the isomorphisms
$$ \mathsf{Hom}_{\cm^{add}_{dg}(e)} (\cu_a(k), \cu_a(\cb)[-n]) \stackrel{\sim}{\longrightarrow} K_n(\cb), \,\, \forall n\geq 0\,.$$
This shows that the algebraic $K$-theory groups $K_n(-), \, n\geq 0$ are co-representable in the triangulated category $\cm_{dg}^{add}(e)$.
\end{remark}
Now let 
$$ K_n(-): \mathsf{Ho}(\dgcat) \longrightarrow \mbox{Mod}\text{-}\mathbb{Z}, \,\, n \geq 0$$
be the $n$th $K$-theory group functor, see theorem~\ref{thmC}, and 
$$ HC_j(-): \mathsf{Ho}(\dgcat) \longrightarrow \mbox{Mod}\text{-}\mathbb{Z}, \,\, j \geq 0$$
the $j$th cyclic homology group functor, see theorem~\ref{thmK}.
\begin{theorem}\label{Chern1}
The co-representability theorem~\ref{corep}  furnishes us the higher Chern characters \cite{Loday}
$$ ch_{n,r}: K_n(-) \longrightarrow HC_{n+2r}(-), \,\, n, r \geq 0\,.$$
\end{theorem}
\begin{proof}
By theorem~\ref{thmK} the morphism of derivators 
$$ C:\mathsf{HO}(\dgcat) \longrightarrow \mathsf{HO}(\Lambda\text{-}\mbox{Mod})$$
is an additive invariant and so descends to $\cm^{add}_{dg}$ and induces a functor (still denoted by $C$)
$$ C: \cm^{add}_{dg}(e) \longrightarrow \cd(\Lambda)\,.$$
By \cite{Kassel} the cyclic homology functor $HC_j(-), \,\,j\geq 0$ is obtained by composing $C$ with the functor
$$ H^{-j}(k\overset{\mathbb{L}}{\otimes}_{\Lambda}-): \cd(\Lambda) \longrightarrow \mbox{Mod}\text{-}\mathbb{Z}, \,\, j \geq 0$$
Now, by proposition~\ref{true} and remark~\ref{true1} the functor
$$ K_n(-): \cm_{dg}^{add}(e) \longrightarrow \mbox{Mod}\text{-}\mathbb{Z}$$
is co-represented by $\cu_a(k)[n]$. This implies, by the Yoneda lemma, that
$$ \mathsf{Nat}(K_n(-), HC_j(-)) \simeq HC_j(\cu_a(k)[n])\,.$$
Since we have the following isomorphisms
$$
\begin{array}{ccl}
HC_j(\cu_a(k)[n]) & \simeq & H^{-j}(k \overset{\mathbb{L}}{\otimes}_{\Lambda} C(\cu_a(k)[n]))\\
& \simeq & H^{-j}(k \overset{\mathbb{L}}{\otimes}_{\Lambda}C(k)[n])\\
& \simeq & H^{-j+n}(k \overset{\mathbb{L}}{\otimes}_{\Lambda}C(k))\\
& \simeq & HC_{j-n}(k)\,.
\end{array}
$$ 
and since 
$$ HC_{\ast}(k)\simeq k[u], \,\, |u|=2$$
we conclude that
\[
HC_j(\cu_a(k)[n]) = \left\{ \begin{array}{rcl}
k & \text{if} & j=n+2r, \,\, r \geq 0\\
0 & & \text{otherwise} \,.
\end{array} \right.
\]
Finally notice that the canonical element $1 \in k$ furnishes us the higher Chern characters and so the theorem is proven.
\end{proof}

\section{Concluding remarks}\label{vistas}

By the universal properties of $\mathcal{U}_u$, $\mathcal{U}_a$ and $\mathcal{U}_l$, we
obtain the following diagram:
$$
\xymatrix{
\mathsf{HO}(\mathsf{dgcat}) \ar[r]^-{\mathcal{U}_u}
\ar[dr]^-{\mathcal{U}_a} \ar[ddr]_-{\mathcal{U}_l}
& \mathcal{M}^{unst}_{dg} \ar@{.>}[d]\\
& \mathcal{M}^{add}_{dg}
\ar@{.>}[d]^{\Phi} \\
& \mathcal{M}^{loc}_{dg} \,.
}
$$
Notice that Waldhausen's connective $K$-theory is an example of an
additive invariant which is NOT a derived one, see
\cite{ICM}. Waldhausen's connective $K$-theory becomes co-representable
in $\mathcal{M}_{dg}^{add}$ by theorem~\ref{corep}. 

An analogous result should be
true for non-connective $K$-theory and the morphism 
$$
\xymatrix{ \Phi: \cm_{dg}^{add} \ar@{.>}[r] & \cm^{loc}_{dg}\,.}$$
 should be thought of as co-representing {\em `the passage from additivity to localization'}.

\end{document}